\newif\ifsiam

\siamfalse

\newif\ifarxiv
\ifsiam \arxivfalse \else \arxivtrue \fi

\ifsiam
\documentclass[review,onefignum,onetabnum]{siamonline220329}
\fi

\ifarxiv
\documentclass[english]{myarticle}
\fi

\usepackage{lipsum}
\usepackage{amsfonts}
\usepackage{graphicx}
\usepackage{epstopdf}
\usepackage{algorithmic}
\usepackage{amsmath}
\ifpdf
  \DeclareGraphicsExtensions{.eps,.pdf,.png,.jpg}
\else
  \DeclareGraphicsExtensions{.eps}
\fi

\usepackage{enumitem}
\setlist[enumerate]{leftmargin=.5in}
\setlist[itemize]{leftmargin=.5in}

\ifsiam
\newsiamremark{remark}{Remark}
\newsiamremark{hypothesis}{Hypothesis}
\newsiamthm{claim}{Claim}
\fi

\ifarxiv
\newtheorem{hypothesis}{Hypothesis}
\fi

\crefname{hypothesis}{Hypothesis}{Hypotheses}

\ifsiam
\headers{Periodic Center Manifolds for Delay Equations in the Light of Suns and Stars}{B. Lentjes, L. Spek, M. M. Bosschaert, Yu. A. Kuznetsov}
\fi

\title{Periodic Center Manifolds for DDEs in the Light of Suns and Stars}

\author{Bram Lentjes\thanks{Department of Mathematics, Hasselt University, Diepenbeek Campus, 3590 Diepenbeek, Belgium \email{(bram.lentjes@uhasselt.be)}.}
\and Len Spek\thanks{Department of Applied
Mathematics, University of Twente,
7500 AE Enschede, The Netherlands \email{(l.spek@utwente.nl)}}
\and Maikel M. Bosschaert\thanks{Department of Mathematics, Hasselt University, Diepenbeek Campus, 3590 Diepenbeek, Belgium \email{(maikel.bosschaert@uhasselt.be)}.}
\and Yuri A. Kuznetsov\thanks{Department of Mathematics, Utrecht University, 3508 TA Utrecht, The Netherlands and Department of Applied
Mathematics, University of Twente,
7500 AE Enschede, The Netherlands \email{(i.a.kouznetsov@uu.nl)}.}}

\usepackage{amsopn}

\DeclareMathOperator{\NBV}{NBV}

\DeclareMathOperator{\Lip}{Lip}

\DeclareMathOperator{\loc}{loc}
\DeclareMathOperator{\BC}{BC}

\DeclareMathOperator{\ran}{ran}

\DeclareMathOperator{\evo}{ev}
\DeclareMathOperator{\minn}{min}
\DeclareMathOperator{\maxx}{max}
\newcommand{\w}{\mbox{w$^\star$-}}

\usepackage{mathtools,thmtools}

\ifpdf
\hypersetup{
  pdftitle={Periodic Center Manifolds for DDEs in the Light of Suns and Stars},
  pdfauthor={B. Lentjes, L. Spek, M. M. Bosschaert, and Yu. A. Kuznetsov}
}
\fi

\begin{document}

\maketitle

\begin{abstract}
In this paper we prove the existence of a periodic smooth finite-dimensional center manifold near a nonhyperbolic cycle in classical delay differential equations by using the Lyapunov-Perron method. The results are based on the rigorous functional analytic perturbation framework for dual semigroups (sun-star calculus). The generality of the dual perturbation framework shows that the results extend to a much broader class of evolution equations.
\end{abstract}


\begin{sloppypar}

\markboth{B. Lentjes, L. Spek, M.M. Bosschaert, and Yu.A. Kuznetsov}{PERIODIC CENTER MANIFOLDS AND NORMAL FORMS FOR DDES}

\section{Introduction} \label{sec:introduction}
Bifurcation theory allows us to analyze the behavior of complicated high dimensional nonlinear dynami\-cal systems near bifurcations by reducing the system to a low dimensional invariant manifold, called the center manifold. Using normal form theory, the dynamics on the center manifold can be described by a simple canonical equation called the normal form. These bifurcations and normal forms can be categorized, and their properties can be understood in terms of certain coefficients of the normal form, see \cite{Kuznetsov2004} for more details. Methods to compute these normal form coefficients have been implemented in software  like \verb|MatCont| \cite{Matcont} and \verb|DDE-BifTool| \cite{Engelborghs2002,Sieber2014} to study various classes of dynamical systems.

For bifurcations of limit cycles in continuous-time dynamical systems, there are three generic codimension one bifurcations: fold (or limit point), period-doubling (or flip) and Neimark-Sacker (or torus) bifurcation. These bifurcations are well understood for ordinary differentials equations (ODEs) \cite{Iooss1988,Iooss1999,Kuznetsov2005,Witte2013}, but for delay differential equations (DDEs) the theory is still lacking. To understand these bifurcations, one should first prove the existence of a center manifold on which one can study the dynamics near a nonhyperbolic cycle via a normal form reduction.

The aim of this paper is to show for classical DDEs that such a center manifold near a nonhyperbolic cycle does exist and is sufficiently smooth. The method of the proof is based on a well-defined variation-of-constants formula in the framework of dual semigroups. In two upcoming papers, we will derive periodic normal forms for bifurcations of limit cycles in classical DDEs and present explicit computational formulas for the critical normal form coefficients of all codimension one bifurcations of limit cycles, completely avoiding Poincar\'{e} maps. Finally, we plan to implement the obtained computational formulas into a software package like \verb|DDE-BifTool|.

\subsection{Background}
Consider a {\it classical delay differential equation} (DDE) 
\begin{equation} \label{eq:IntroDDE}
    \dot{x}(t) = F(x_t), \quad t \geq 0,
\end{equation}
where $x(t) \in \mathbb{R}^n$ and 
\begin{equation*} 
    x_t(\theta) := x(t+\theta), \quad \theta \in [-h,0],
\end{equation*}
represents the \emph{history} at time $t$ of the unknown $x$, and $0<h<\infty$ denotes the upper bound of (finite) delays. The $\mathbb{R}^n$-valued smooth operator $F$ is defined on the Banach space $X := C([-h,0],\mathbb{R}^n)$ consisting of $\mathbb{R}^n$-valued continuous functions on the compact interval $[-h,0]$, endowed with the supremum norm.

Using the perturbation framework of dual semigroups, called sun-star calculus, developed in \cite{Clement1988,Clement1987,Clement1989,Clement1989a,Diekmann1991}, the existence of a smooth finite-dimensional center manifold near a nonhyperbolic equilibrium of \eqref{eq:IntroDDE} can be rigorously established using the Lyapunov-Perron method, see \cite{Diekmann1995,Diekmann1991CM} for the critical center manifold and \cite{Bosschaert2020} for the parameter-dependent center manifold.  Furthermore, in \cite{Bosschaert2020,Diekmann1995,Janssens2010}, the authors derive explicit computational formulas for the normal form coefficients for all generic codimension one and two bifurcations for equilibria. These have been implemented in the \verb|MATLAB| package \verb|DDE-BifTool|. The question arises if this whole construction can be repeated for a nonhyperbolic periodic orbit (cycle) $\Gamma := \{ \gamma_t \in X \ : \ t \in \mathbb{R} \}$, where $\gamma : \mathbb{R} \to \mathbb{R}^n$ is a $T$-periodic solution of \eqref{eq:IntroDDE}.

In this paper, we  build a promising framework to generalize the described construction towards nonhyperbolic cycles in DDEs. Therefore, our first goal is to prove the existence of a smooth finite-dimensional periodic center manifold in a neighborhood of $\Gamma$ using the Lyapunov-Perron method, but now in a time-dependent setting. To achieve this, we prove the existence of a smooth finite-dimensional periodic center manifold $\mathcal{W}_{\loc}^c$ in the neighborhood of the origin of the time-dependent translated system
\begin{equation} \label{eq:IntroTDDE}
    \dot{y}(t) = L(t)y_t + G(t,y_t),
\end{equation}
where $x = \gamma + y$, $L(t) := DF(\gamma_t)$ denotes the Fr\'echet derivative of $F$ evaluated at $\gamma_t$ and $G(t,\cdot) := F(\gamma_t + \cdot) - F(\gamma_t) - L(t)$ consists of solely nonlinear terms. Note that both $L$ and $G$ are $T$-periodic in the variable $t$. Afterwards, we translate the manifold $\mathcal{W}_{\loc}^c$, defined near the origin of \eqref{eq:IntroTDDE}, back towards the original cycle $\Gamma$. Hence, we obtain a smooth finite-dimensional periodic center manifold $\mathcal{W}_{\loc}^c(\Gamma)$ defined near the nonhyperbolic cycle $\Gamma$.

The first attempt to use a periodic center manifold for classical DDEs was made in the very interesting paper \cite{Szalai2010} by Szalai and St\'ep\'an, who heuristically applied the Lyapunov-Perron method for equilibria from \cite{Diekmann1995} towards the periodic setting using sun-star calculus. However, no proof of the existence of such a center manifold was given, and in addition their results were only applicable when the period of the cycle $T$ precisely equals the delay $h$, which is a major restriction. 

The existence of a finite-dimensional periodic center manifold for \eqref{eq:IntroTDDE} was recently established in \cite{Church2018} by Church and Liu using the Lyapunov-Perron method for a specific class of delay equations, namely impulsive DDEs. These delay equations have a countable number of discontinuities in their solutions, and therefore it is in general not possible for the obtained center manifold to be smooth in the time direction. However, this smoothness will be crucial to derive in two upcoming articles the periodic normal forms and computational formulas for the critical normal form coefficients. Therefore, we will apply in this paper the Lyapunov-Perron method, along the lines of \cite{Diekmann1995,Church2018}, to prove the existence of a smooth finite-dimensional periodic center manifold for \eqref{eq:IntroDDE} in the sun-star calculus setting, by utilizing a well-defined variation-of-constants formula. The framework used by Church and Liu is a rigorous adaption of the formal adjoint approach \cite{Hale1993}, where they work with the space of right-continuous regulated functions, see \cite{Church2018,Church2019} for more information. Furthermore, as already remarked in \cite[Section 8.2]{Hale1993}, the traditional bilinear form used in the formal adjoint approach is not applicable to study linear behavior of solutions near periodic orbits. Therefore, it seems difficult to derive the critical normal form coefficients for codimension one bifurcations of limit cycles using the formal adjoint approach. However, Church and Liu obtained such computational formulas, but employing the Poincar\'e maps \cite{Church2019}. When one is interested in studying numerically the local behavior of solutions in the vicinity of $\Gamma$ via the Poincar\'e map, it is necessary to compute (higher order) derivatives of this map \cite{Kuznetsov2005}, which already does not look very promising for ODEs, let alone (impulsive) DDEs.

Furthermore, we mention the work by Hupkes and Verduyn Lunel on the existence and smoothness of center manifolds near equilibria \cite{Hupkes2006} and periodic orbits \cite{Hupkes2008} for so-called functional differential equations of mixed type (MFDEs). These differential equations involve retarded but also advanced arguments and impose in general ill-defined initial value problems. As a consequence, they can not apply directly the Lyapunov-Perron method on a variation-of-constants formula and therefore use other methods, like Laplace transforms and Fourier analysis.

\subsection{Overview}
The paper is organized as follows. In \Cref{sec:dual perturbation theory} we review and extend the theory of dual semigroups (sun-star calculus) with time-dependent (nonlinear) perturbations, both on an abstract level and applied to the analysis of time-dependent (nonlinear) delay differential equations. 

In \Cref{sec: Existence} we use the theory from the previous section to prove the existence of a smooth finite-dimensional periodic center manifold for \eqref{eq:IntroTDDE} near the origin, see \Cref{cor:CMT DDE} for the final result. Due to the dual perturbation framework, the proven results apply to a way more general class of evolution equations, as for example renewal equations \cite{Diekmann2008} and systems consisting of infinite delay \cite{Diekmann2012}, see \Cref{thm:LCMT} for the general result. Additional material on spectral decompositions can be found in \Cref{appendix: spectral} and some technical proofs on increasing smoothness and periodicity are relegated to \Cref{appendix: smoothness and periodicty}. To apply the general theory to classical DDEs, we also use the material presented in \Cref{appendix: variation constants}.

\section{Dual perturbation theory} \label{sec:dual perturbation theory}
We start by briefly recalling the general elements of (time-dependent) dual perturbation theory that are useful to study classical DDEs as dynamical systems. Standard references for this entire section are the book \cite{Diekmann1995} together with the article \cite{Clement1988} on time-dependent perturbations. All unreferenced claims relating to basic properties of time-dependent perturbations of delays equations can be found here.

\subsection{Duality structure} \label{subsec: duality structure}
Let $T_0 := \{T_0(t)\}_{t \geq 0}$ be a $\mathcal{C}_0$-semigroup of bounded linear operators defined on a real or complex Banach space $X$ that has $A_0$ as (infinitesimal) generator with domain $\mathcal{D}(A_0)$. Then the \emph{dual semigroup} $T_0^\star := \{T_0^\star(t) \}_{t \geq 0}$, where $T_0^\star(t) : X^\star \to X^\star$ is the adjoint of $T_0(t)$, is a semigroup on the topological dual space $X^\star$ of $X$. We denote the duality paring between $X$ and $X^\star$ as
\begin{equation*}
    \langle x^\star,x \rangle := x^\star(x), \quad \forall x^\star \in X^\star, \ x \in X.
\end{equation*}
If $X$ is not reflexive, then $T_0^\star$ is in general only weak$^\star$ continuous on $X^\star$. This is also visible on the generator level, as the adjoint $A_0^\star$ of $A_0$ is only the weak$^\star$ generator of $T_0^\star$ and has in general a non-dense domain. The maximal subspace of strong continuity
\begin{equation*}
    X^\odot := \{ x^\star \in X^\star : t \mapsto T_0^\star(t) x^\star \mbox{ is norm continuous on } [0,\infty) \}
\end{equation*}
is a norm closed $T_0^\star(t)$-invariant weak$^\star$ dense subspace of $X^\star$ and we have the characterization
\begin{equation} \label{eq:Xsuncharac}
    X^\odot = \overline{\mathcal{D}(A_0^\star)},
\end{equation}
where the bar denotes the norm closure in $X^\star$. The restriction of $T_0^\star$ to $X^\odot$ is a $\mathcal{C}_0$-semigroup on $X^\odot$ and its generator $A_0^\odot$ is the part of $A_0^\odot$ in $X^\odot$
\begin{equation*}
    \mathcal{D}(A_0^{\odot}) = \{ x^{\odot} \in \mathcal{D}(A_0^{\star}) \ : \ A_0^{\star}x^{\odot} \in X^{\odot} \}, \quad A_0^{\odot}x^{\odot} = A_0^{\star}x^{\odot}.
\end{equation*}
We have at this moment a $\mathcal{C}_0$-semigroup $T_0^\odot$ with generator $A_0^\odot$ on the Banach space $X^\odot$, which are precisely the ingredients we started with. Repeating the construction once more, we obtain on the dual space $X^{\odot \star}$ the weak$^\star$ continuous adjoint semigroup $T_0^{\odot \star}$ with weak$^\star$ generator $A_0^{\odot \star}$. The restriction of $T_0^{\odot \star}$ to the maximal subspace of strong continuity $X^{\odot \odot}$ gives a $\mathcal{C}_0$-semigroup $T_0^{\odot \odot}$ with generator $A_0^{\odot \odot}$ that is the part of $A_0^{\odot \star}$ in $X^{\odot \odot}$. The canonical continuous embedding $j : X \to X^{\odot \star}$ defined by
\begin{equation} \label{eq: j}
    \langle jx, x^\odot \rangle := \langle x^\odot, x \rangle, \quad \forall x \in X, \ x^\odot \in X^\odot,
\end{equation}
maps $X$ into $X^{\odot \odot}$. If $j$ maps $X$ onto $X^{\odot \odot}$ then $X$ is called $\odot$-\emph{reflexive} with respect to $T_0$. $\odot$-reflexivity with respect to $T_0$ will be assumed throughout as this is not a restriction for studying classical DDEs, see \Cref{subsec: classical DDEs}.

\subsection{Time-dependent bounded linear perturbations} \label{subsec: linear perturbation}
Let us now turn our attention to perturbations. We will show how a time-dependent perturbation is handled in the setting of dual perturbation theory. 

A \emph{time-dependent bounded linear perturbation} can be represented as a Lipschitz continuous map $B: J \to \mathcal{L}(X,X^{\odot \star})$, where $J \subseteq \mathbb{R}$ is an interval and $\mathcal{L}(X,X^{\odot \star})$ stands for the Banach space of all bounded linear operators from $X$ to $X^{\odot \star}$, equipped with the operator norm. We will be interested in linear abstract ODEs formulated on the space $X^{\odot \star}$, where the perturbation $B$ appears in the right-hand side of the ODE. Recall from \Cref{subsec: duality structure} that we primarily work with the weak$^\star$ topology on $X^{\odot \star}$. To formulate a well-posed abstract ODE on this space, we need to work with weak$^\star$ differentiability.

Let $s \in J$ be a given \emph{starting time} and consider the initial value problem
\begin{equation} \label{eq:T-LAODEphi} \tag{T-LAODE}
\begin{dcases}
        d^\star (j\circ u)(t) = A_0^{\odot \star}ju(t) + B(t)u(t), \quad &t \geq s, \\
        u(s) = \varphi, \quad  & \varphi \in X,
\end{dcases}
\end{equation}
where $d^\star$ represents the weak$^\star$ differential operator \cite[Definition 15]{Janssens2020}. We define a \emph{subinterval} $I$ of $J$ to be an interval such that $s \in I \subseteq [s, \sup J)$. A \emph{solution of} \eqref{eq:T-LAODEphi} on a subinterval $I$ is a function $u : I \to X$ taking values in $j^{-1} \mathcal{D}(A_0^{\odot \star})$ such that $j \circ u$ is weak$^\star$ continuously differentiable on $I$ and satisfies \eqref{eq:T-LAODEphi} here. According to the literature \cite{Diekmann1995,Clement1988}, it is more convenient to study the formally integrated problem as the time-dependent linear abstract integral equation
\begin{equation} \label{eq:T-LAIEphi} \tag{T-LAIE}
    u(t) = T_0(t-s)\varphi + j^{-1}\int_s^t T_0^{\odot \star}(t-\tau) B(\tau)u(\tau) d\tau, \quad \varphi \in X,
\end{equation}
with $t \geq s$ where the integral has to be interpreted as a weak$^\star$ Riemann integral \cite[Chapter III]{Diekmann1995} and takes values in $j(X)$ under the running assumption of $\odot$-reflexivity, see \cite[Lemma 2.2]{Clement1988}. A \emph{solution of} \eqref{eq:T-LAIEphi} on a subinterval $I$ is a function $u : I \to X$ that is continuous on $I$ and satisfies \eqref{eq:T-LAIEphi} here. Let $\Omega_J := \{(t,s) \in J \times J \ : \ t \geq s \}$, then the unique solution of \eqref{eq:T-LAIEphi} on a subinterval $I$ is generated by a strongly continuous forward evolutionary system  $U := \{U(t,s)\}_{(t,s) \in \Omega_J}$ on $X$ in the sense that $u(t) = U(t,s)\varphi$ for all $t \in I$. The definition of $U$ can be found in \cite[Theorem 2.3]{Clement1988} and the definition of a (strongly continuous) forward (or backward) evolutionary system can be found in \cite[Definition 2.1]{Clement1988}. If one defines for any $s \in J$ the \emph{(generalized) generator} $A^{\odot \star}(s) : \mathcal{D}(A^{\odot \star}(s)) \to X^{\odot \star}$ as
\begin{equation*}
    A^{\odot \star}(s)jx := \w \lim_{t \downarrow s} \frac{1}{t-s}(jU(t,s)x - jx),
\end{equation*}
for any $jx$ in the \emph{(generalized) domain}
\begin{equation*}
    \mathcal{D}(A^{\odot \star}(s)) := \bigg\{ jx \in X^{\odot \star} \ : \ \w \lim_{t \downarrow s} \frac{1}{t-s}(jU(t,s)x - jx) \mbox{ exists in $X^{\odot \star}$} \bigg\},
\end{equation*}
it is known that the perturbation $B$ enters additively in the action of the generator \cite[Lemma 4.3]{Clement1988} as
\begin{equation} \label{eq:defAsunstar s}
    \mathcal{D}(A^{\odot \star}(s)) =  \mathcal{D}(A_0^{\odot \star}), \quad A^{\odot \star}(s) = A_0^{\odot \star} + B(s)j^{-1}, \quad \forall s \in J.
\end{equation}
Let us now go back to \eqref{eq:T-LAODEphi}. If the initial condition $\varphi \in j^{-1}\mathcal{D}(A_0^{\odot \star})$ then, due to the Lipschitz continuity of $B$, it is known that $u = U(\cdot,s)\varphi : I \to X$ is  the unique solution of \eqref{eq:T-LAODEphi} on a subinterval $I$, see \cite[Theorem 4.6, Theorem 4.9 and Theorem 4.14]{Clement1988}.

As we have defined $U(t,s)$ for all $(t,s) \in \Omega_J$, we are interested in the associated (sun) dual(s). It is clear that one can define $U^\star(s,t) := U(t,s)^\star \in \mathcal{L}(X^\star):=\mathcal{L}(X^\star,X^\star)$ and that $U^\star := \{U^\star(s,t)\}_{(s,t) \in \Omega_J^\star}$ forms a backward evolutionary system on $X^\star$, with $\Omega_J^\star := \{(s,t) \in J^2 : t \geq s \}$. Furthermore, the Lipschitz continuity on $B$ ensures that the restriction $U^{\odot}(s,t) := U^\star(s,t) |_{X^\odot}$ leaves $X^\odot$ invariant \cite[Theorem 5.3]{Clement1988} and, by construction, $U^\odot := \{U^{\odot}(s,t)\}_{(s,t) \in \Omega_J^\star}$ is a strongly continuous backward evolutionary system, see \cite[Theorem 5.4]{Clement1988}. This allows us to define $U^{\odot \star}(t,s) := (U^{\odot}(s,t))^\star$ and it is clear that $U^{\odot \star} := \{U^{\odot \star}(t,s) \}_{(t,s) \in \Omega_J}$ is a forward evolutionary system on $X^{\odot \star}$ that extends $U$, which was previously defined on $X$.

In the upcoming sections, we will have to deal with a particular weak$^\star$ integral involving $U^{\odot \star}$ that will be studied in the following lemma. This integral is crucial in the variation-of-constants formulation of \eqref{eq:IntroDDE}.
\begin{lemma} \label{lemma:wk*integral Usunstar}
Let $g:J\rightarrow X^{\odot \star}$ be continuous and denote the set $\{(t,r,s)\in J^3 \ : \ s\leq r\leq t\}$ by $\Theta_J$. Then the map $v(\cdot,\cdot,\cdot,g):\Theta_J \rightarrow X^{\odot \star}$ defined as the weak$^\star$ integral
\begin{equation*}
v(t,r,s,g):= \int_s^r U^{\odot \star}(t,\tau)g(\tau)d\tau,  \quad \forall (t,r,s) \in \Theta_J,
\end{equation*}
is continuous and takes values in $j(X)$. Furthermore, if $J$ is unbounded from below and $v(\cdot,\cdot,\cdot,g)$ is also bounded in norm on $\Theta_J$, then the limiting function $v(\cdot,\cdot,-\infty,g)$ converges in norm, is continuous, and its range is contained in $j(X)$.
\end{lemma}
\begin{proof}
Let $(t_1,r_1,s_1),(t_2,r_2,s_2)\in \Theta_J$ and performing the change of variables $\sigma =t-\tau$ yields
\begin{equation*}
    v(t,r,s,g)= \int_{t-r}^{t-s} U^{\odot \star}(t,t-\sigma) g(t-\sigma) d\sigma.
\end{equation*}
Let $I_i=[t_i-r_i,t_i-s_i]$ for $i=1,2$. We can split the following difference into four integrals.
\begin{align*}
v(t_1,r_1,s_1,g)-v(t_2,r_2,s_2,g) &= \int_{I_2/I_1}
\hspace{-3.5pt}U^{\odot \star}(t_2,t_2-\sigma) g(t_2-\sigma) d\sigma\\
&-\int_{I_1/I_2} \hspace{-3.5pt} U^{\odot \star}(t_2,t_2-\sigma) g(t_2-\sigma) d\sigma\\
&+ \int_{I_1\cap I_2} \hspace{-3.5pt} (U^{\odot \star}(t_2,t_2-\sigma)-U^{\odot \star}(t_1,t_1-\sigma)) g(t_2-\sigma) d\sigma\\
&+ \int_{I_1\cap I_2} \hspace{-3.5pt} U^{\odot \star}(t_1,t_1-\sigma) (g(t_2-\sigma)-g(t_1-\sigma)) d\sigma,
\end{align*}
and using the triangle inequality, we get the following estimate 
\begin{align*}
&\|v(t_1,r_1,s_1,g)-v(t_2,r_2,s_2,g)\| \\
&\leq (|I_1/I_2|+|I_2/I_1|)\sup_{\sigma\in I_1/I_2 \cup I_2/I_1}\|U^{\odot \star}(t_2,t_2-\sigma) g(t_2-\sigma)\|\\
&+ |I_1\cap I_2| \sup_{\sigma \in I_1\cap I_2} \|U^{\odot \star}(t_2,t_2-\sigma)-U^{\odot \star}(t_1,t_1-\sigma)\| \|g(t_2-\sigma)\| \\
&+ |I_1\cap I_2| \sup_{t,\sigma \in I_1\cap I_2} \|U^{\odot \star}(t,t-\sigma)\| \|g(t_2-\sigma)-g(t_1-\sigma)\|,
\end{align*}
where $|\cdot|$ denotes the Lebesgue measure on $J \subseteq \mathbb{R}$. If we let $(t_1,r_1,s_1)\rightarrow(t_2,r_2,s_2)$ in norm, then the first term vanishes by definition. The second term vanishes as $U(t,s)$ (and so $U^{\odot \star}(t,s)$) is uniformly continuous along paths that keep $t-s$ constant \cite[Lemma 5.2]{Clement1988} and the last term vanishes due to the continuity of $g$. Hence, $v(\cdot,\cdot,\cdot,g)$ is continuous. Note that the second term does not appear for semigroups, as they are invariant under time translations.

Next, we will prove that the range of $v(\cdot,\cdot,\cdot,g)$ is contained in $j(X)$. Let $(t,r,s)\in \Theta_J$ and recall from \eqref{eq:defAsunstar s} that $\mathcal{D}(A^{\odot \star}(t)) = \mathcal{D}(A_0^{\odot \star})$. Taking the closure with respect to the norm defined on $X^{\odot \star}$ we get that
\begin{align*}
    \{x^{\odot \star} \in X^{\odot \star} \ : \ \lim_{h\downarrow 0} \|U^{\odot \star}(t+h,t)x^{\odot \star} - x^{\odot \star}\| = 0\} &= \overline{\mathcal{D}(A^{\odot \star}(t))} \\
    &= \overline{\mathcal{D}(A_0^{\odot \star})} = X^{\odot \odot} = j(X),
\end{align*}
where the first equality holds due to a sun-variant of \cite[Lemma 3.1]{Clement1988}. The last two equalities follow from the sun-variant of \eqref{eq:Xsuncharac} and $\odot$-reflexivity of $X$ with respect to $T_0$. We want to show that $v(t,r,s,g)$ is an element of this first set. 
\noindent
Using the continuity of $v(\cdot,\cdot,\cdot,g)$ we find that
\[\lim_{h\downarrow 0} \|U^{\odot \star}(t+h,t)v(t,r,s,g) - v(t,r,s,g)\| = \lim_{h\downarrow 0} \|v(t+h,r,s) - v(t,r,s)\|= 0,\]
and so we conclude that $v(t,r,s)\in j(X)$.

Finally, let $J$ be unbounded from below and suppose that $v(\cdot,\cdot,\cdot,g)$ in bounded in norm on $\Theta_J$. Define the map $w(\cdot,\cdot,g) : \Omega_J \to X^{\odot \star}$ as
\begin{equation*}
    w(t,r,g) := \lim_{n \to \infty} \int_{r-n}^r U^{\odot \star}(t,\tau)g(\tau)d\tau, \quad \forall (t,r) \in \Omega_J,
\end{equation*}
which is well-defined due to \cite[Lemma 9]{Janssens2020} and the boundedness of $v(\cdot,\cdot,\cdot,g)$ in norm. To see this, notice that for any fixed $t \in J$, the integrand of $w(\cdot,\cdot,g)$ is weak$^\star$ continuous, which implies weak$^\star$ Lebesgue measurability, since for any $\tau \in J$ and $h \in \mathbb{R}$ such that $t \geq \max\{\tau,\tau+h\}$ and $\tau+h \in J$ we obtain that for all $x^\odot \in X^\odot$
\begin{align*}
    &| \langle U^{\odot \star}(t,\tau+h)g(\tau+h),x^\odot \rangle - \langle U^{\odot \star}(t,\tau)g(\tau),x^\odot \rangle| \\
    &\leq | \langle g(\tau+h), U^{\odot}(\tau+h,t)x^\odot \rangle - \langle g(\tau+h), U^{\odot}(\tau,t)x^\odot \rangle | \\
    &+ | \langle g(\tau+h), U^{\odot}(\tau,t)x^\odot \rangle - \langle g(\tau),U^{\odot}(\tau,t)x^\odot \rangle| \\
    &\leq \|g(\tau + h)\| \ \|U^{\odot}(\tau+h,t)x^\odot - U^{\odot}(\tau,t)x^\odot\| + \|g(\tau+h) - g(\tau)\| \ \|U^\odot(\tau,t)\| \ \|x^\odot\|\\
    & \to 0, \quad \mbox{ as } h \to 0,
\end{align*}
since $g$ is norm continuous and $U^\odot$ is a strongly continuous backward evolutionary system. Furthermore, boundedness of $v$ implies uniform continuity, hence $w$ also continuous. Since $[r-n,r]$ is compact for any fixed $n \in \mathbb{N}$, each integral inside the limit of $w(t,r,g)$ lies in $j(X)$ by the reasoning above. As $j(X)=X^{\odot \odot}$ is closed $w(t,r,g) = v(t,r,-\infty) \in j(X)$.
\end{proof}

\subsection{Time-dependent nonlinear perturbations} \label{subsec: nonlinear perturbation}
The strongly continuous forward evolutionary system $U$ arises as a time-dependent bounded linear perturbation of the original $\mathcal{C}_0$-semigroup $T_0$, see \eqref{eq:T-LAODEphi} and \eqref{eq:T-LAIEphi}. The next logical step is to introduce a time-dependent nonlinear perturbation on $U$ itself. We can formulate solutions to this problem via a time-dependent nonlinear abstract integral equation. 

A \emph{time-dependent nonlinear perturbation} on an interval $J \subseteq \mathbb{R}$ can be represented as a $C^k$-smooth operator $R : J \times X \to X^{\odot \star}$ for some $k \geq 1$ that satisfies
\begin{equation} \label{eq:nonlinearterms}
    R(t,0) = 0, \quad D_2R(t,0) = 0, \quad \forall t \in J,
\end{equation}
where the $D_2R(t,0)$ denotes the partial Fr\'echet derivative of $R$ with respect to the second component evaluated $(t,0)$. Consider now the time-dependent nonlinear abstract integral equation
\begin{equation} \label{eq:T-AIEphi} \tag{T-AIE}
    u(t) = U(t,s)\varphi + j^{-1}\int_s^t U^{\odot \star}(t,\tau) R(\tau,u(\tau)) d\tau, \quad \varphi \in X,
\end{equation}
with $t \geq s$ and $u(s) = \varphi$, where $s$ plays the role of a starting time. It follows from \Cref{lemma:wk*integral Usunstar} that the weak$^\star$ integral in \eqref{eq:T-AIEphi} takes values in $j(X)$ and hence \eqref{eq:T-AIEphi} is well-defined. A solution to \eqref{eq:T-AIEphi} is similarly defined as in \Cref{subsec: linear perturbation}. We would like to show that \eqref{eq:T-AIEphi} admits a solution, and therefore the following result is a first step in the right direction. In the proof of the following result, we use some results from \Cref{subsec:interplay} and the proof is inspired by \cite[Proposition 24]{Janssens2020}.

\begin{proposition} \label{prop:equi nonlinear}
Let $I$ be a subinterval of $J$. A function $u : I \to X$ is a solution to \eqref{eq:T-AIEphi} if and only if $u$ is a solution to
\begin{equation} \label{eq:T-AIE2}
    u(t) = T_0(t-s)\varphi + j^{-1} \int_s^t T_0^{\odot \star}(t-\tau)[B(\tau)u(\tau) + R(\tau,u(\tau))] d\tau, \quad \varphi \in X.
\end{equation}
\end{proposition}
\begin{proof}
Suppose that $u$ is a solution to \eqref{eq:T-AIE2} on $I$. Then $u$ is a solution of \eqref{eq:T-LAIEf2} on $I$ with $f = R(\cdot,u(\cdot))$. \Cref{prop:AIEf} implies that $u$ is given by \eqref{eq:T-LAIEf} on $I$ with $f = R(\cdot,u(\cdot))$, so $u$ satisfies \eqref{eq:T-AIEphi} on $I$. The converse is proven by reversing the order of steps.
\end{proof}

Since the nonlinearity $R$ is $C^k$-smooth, we know from the mean value inequality in Banach spaces \cite[Corollary 3.2]{Coleman2012} that $R$ is locally Lipschitz in the second component. One can use now a standard contraction argument \cite[Theorem VII.3.1 and VII.3.4]{Diekmann1995} on \eqref{eq:T-AIE2} to prove that for any $\varphi \in X$ there exists a unique (maximal) solution $u_\varphi$ of \eqref{eq:T-AIE2} on some (maximal) subinterval $I_\varphi =[s,t_\varphi)$ of $J$ with $s < t_\varphi \leq \infty$. \Cref{prop:equi nonlinear} shows that $u_\varphi$ is then also a unique (maximal) solution of \eqref{eq:T-AIEphi} on $I_\varphi$.

In this time-dependent setting, one expects the existence of a \emph{time-dependent semiflow}, that is the nonlinear analogue of a forward evolutionary system and the time-dependent analogue of a semiflow introduced in \cite[Definition VII.2.1]{Diekmann1995}. Time-dependent semiflows also known as \emph{processes}, see \cite{Church2018,Church2019} for more information.

\begin{definition}
Let $J \subseteq \mathbb{R}$ be an interval. A \emph{time-dependent semiflow} on a Banach space $X$ is a map $S : \mathcal{D}(S) \subseteq \Omega_J \times X \to X$, that has the following properties:
\begin{enumerate}
    \item For any $s \in J$ and $x \in X$, there exists a $t_x \in [s,\infty] \cap J $ such that $\mathcal{D}(S) = \{(t,s,x) \in \Omega_J \times X \ : \ t \in [s,t_x) \}$.
    \item For any $s \in J$ and $x \in X$ we have $S(s,s,x) = x.$
    \item For any $t,v,s \in J$  with $t \geq v \geq s$ and $x \in X$ it holds
    \begin{equation*}
        S(t,s,x) = S(t,v,S(v,s,x)).
    \end{equation*}
\end{enumerate}
\end{definition}

With the family of (maximal) solutions to \eqref{eq:T-AIEphi}, one can associate a time-dependent semiflow on $X$ via the map $S : \mathcal{D}(S) \to X$ defined by
\begin{equation} \label{eq:semiflowtimedep}
    \mathcal{D}(S) := \{(t,s,\varphi) \in \Omega_J \times X \ : \ t \in I_{\varphi}\}, \quad S(t,s,\varphi) := u_\varphi(t),
\end{equation}
where $I_\varphi$ denotes the (maximal) subinterval of $J$ and $u_\varphi$ is the unique (maximal) solution of \eqref{eq:T-AIEphi}.

\section{Existence of the center manifold} \label{sec: Existence}
In this section, we prove the existence of a periodic smooth finite-dimensional center manifold near the origin of \eqref{eq:T-AIEphi} and apply afterwards the obtained results to classical DDEs. 

To specify the setting, let $X$ be a real Banach space that is $\odot$-reflexive with respect to a given $\mathcal{C}_0$-semigroup $T_0$ defined on $X$. Let $B : J \to \mathcal{L}(X,X^{\odot \star})$ be a time-dependent bounded linear perturbation defined on an interval $J \subseteq \mathbb{R}$ and define the strongly continuous forward evolutionary system $U$ as the unique solution of \eqref{eq:T-LAIEphi} together with the (sun) dual(s) $U^\star,U^\odot$ and $U^{\odot \star}$. Assume that $R : J \times X \to X^{\odot \star}$ is a time-dependent nonlinear perturbation that is $C^k$-smooth for some $k \geq 1$. Furthermore, let $S : \mathcal{D}(S) \to X$ denote the time-dependent semiflow defined in \eqref{eq:semiflowtimedep} that corresponds a local unique solution of \eqref{eq:T-AIEphi}.

It turns out that these assumptions are not sufficient to prove the existence of a periodic smooth finite-dimensional center manifold for \eqref{eq:T-AIEphi}. Therefore, we invoke in \Cref{subsec: Spectral} a hypothesis about the spectral structure of $X$ and $U$. It turns out that we can lift the time-specific spectral decomposition of $X$ towards a spectral decomposition of $X^{\odot \star}$, using some technical lemmas presented in \Cref{appendix: lifting} and \Cref{appendix: Spectral DDE}. We show boundedness of solutions of the abstract integral equation in \Cref{subsec: bounded solutions} and \Cref{subsec: modification nonlinearity}. This allows us to prove the existence of a  Lipschitz center manifold in \Cref{subsec: Lipschitz CM} using a fixed point argument. In \Cref{subsec: properties CM} we show smoothness and periodicity using the theory of scales of Banach spaces \cite{Vanderbauwhede1987}, where the details can be found in \Cref{appendix: smoothness and periodicty}. Finally, in \Cref{subsec: classical DDEs}, we explain how the setting of classical DDEs fits naturally in the perturbation framework,  see \Cref{cor:CMT DDE} for the main result.

\subsection{Spectral decompositions of $X$ and $X^{\odot \star}$} \label{subsec: Spectral}
The construction of a local center manifold has been established for equilibria under the assumption of the existence of a topological direct sum decomposition of $X^{\odot \star}$, see \cite[Section IX.2]{Diekmann1995}. The motivation behind this follows from the fact that the nonlinearity maps into $X^{\odot \star}$. However, depending on the evolution equation of interest, one should always first compute $X^{\odot \star}$ and its associated $\odot \star$-tools to check the underlying assumptions. It is therefore more convenient to state a hypothesis in $X$ and lift this towards $X^{\odot \star}$, which was also the observation made in \cite[Section 5.1]{Janssens2020}. Furthermore, the decompositions in $X$ and $X^{\odot \star}$ allows us to move back and forth between the two. The following hypothesis on the time-dependent spectral decompositions is inspired by \cite{Janssens2020,Church2018}.

\begin{hypothesis} \label{hyp:CMT}
The space $X$ and the forward evolutionary system $U$ have the following properties:
\begin{enumerate}
    \item $X$ admits a direct sum decomposition
    \begin{equation} \label{eq:decomposition X hyp}
        X = X_{-}(s) \oplus X_0(s) \oplus X_{+}(s), \quad \forall s \in \mathbb{R},
    \end{equation}
    where each summand is closed.
    \item There exist three continuous time-dependent (spectral) projectors $P_{i} : \mathbb{R} \to \mathcal{L}(X)$ with $\ran(P_i(s))= X_i(s)$ for any $s \in \mathbb{R}$ and $i \in \{-,0,+\}$. 
    \item There exists a constant $N \geq 0$ such that $\sup_{s \in \mathbb{R}}(\|P_{-}(s)\| + \|P_{0}(s)\| + \|P_{+}(s)\|) = N < \infty$.
    \item The projections are mutually orthogonal, meaning that $P_{i}(s)P_j(s) = 0$ for all $i \neq j$ and $s \in \mathbb{R}$ with $i,j \in \{-,0,+\}$.
    \item The projections commute with the forward evolutionary system: $U(t,s)P_i(s) = P_i(t)U(t,s)$ for all $i \in \{-,0,+\}$ and $t \geq s$.
    \item Define the restrictions $U_{i}(t,s) : X_{i}(s) \to X_{i}(t)$ for $i \in \{-,0,+\}$ and $t \geq s$. The operators $U_{0}(t,s)$ and $U_{+}(t,s)$ are invertible and also backward evolutionary systems. Specifically, for any $t,\tau,s \in \mathbb{R}$ it holds
    \begin{equation} \label{eq:U0U+}
        U_0(t,s) = U_0(t,\tau)U_0(\tau,s), \quad U_+(t,s) = U_+(t,\tau)U_+(\tau,s).
    \end{equation}
    \item The decomposition \eqref{eq:decomposition X hyp} is an exponential trichotomy on $\mathbb{R}$ meaning that there exist $a < 0 < b$ such that for every $\varepsilon > 0$ there exists a $K_\varepsilon > 0$ such that
    \begin{align*}
        \|U_-(t,s)\| &\leq K_\varepsilon e^{a(t-s)}, \quad t \geq s,\\
        \|U_0(t,s)\| &\leq K_\varepsilon e^{\varepsilon|t-s|}, \quad t,s \in \mathbb{R},\\
        \|U_+(t,s)\| &\leq K_\varepsilon e^{b(t-s)}, \quad t \leq s.
    \end{align*}
\end{enumerate}
We call $X_{-}(s),X_{0}(s)$ and $X_{+}(s)$ the \emph{stable subspace}, \emph{center subspace} and \emph{unstable subspace} (at time $s$) respectively.
\end{hypothesis}

As the stable-, center- and unstable subspace are only defined at a specific time $s\in \mathbb{R}$, it is convenient to introduce  the sets 
\begin{equation*}
    X_{i} := \{(t,\varphi) \in \mathbb{R} \times X \ : \ \varphi \in X_i(t) \},
\end{equation*}
for $i \in \{-,0,+\}$ and call them the \emph{stable fiber bundle, center fiber bundle} and \emph{unstable fiber bundle} respectively. It is explained in \Cref{appendix: lifting} how  \Cref{hyp:CMT} can be lifted to $X^{\odot \star}$, see \Cref{prop:Xsunstar hyp} for the main result. We also impose the following hypothesis which will often be used in several upcoming proofs to transfer information from $X^{\odot \star}$ towards $X$. 

\begin{hypothesis} \label{hyp:X0+sunstar}
The subspaces $X_{0}^{\odot \star}(s)$ and $X_{+}^{\odot \star}(s)$ are contained in $j(X_0(s))$ and $j(X_{+}(s))$ respectively, for all $s \in \mathbb{R}$.
\end{hypothesis}
For the setting of classical DDEs, we even have an equality between the spaces presented in \Cref{hyp:X0+sunstar}, see \Cref{appendix: Spectral DDE}.

As part of the construction of a center manifold, we will be interested in solutions that exist for all time. It is therefore helpful to write \eqref{eq:T-AIEphi} in translation invariant form
\begin{equation} \label{eq:variation constants CMT}
    u(t) = U(t,s)u(s) + j^{-1}\int_s^t U^{\odot \star}(t,\tau) R(\tau,u(\tau)) d\tau, \quad -\infty < s \leq t < \infty.
\end{equation}
One of the problems that occur in developing a center manifold theory for infinite-dimensional systems is that the linearized equation of \eqref{eq:variation constants CMT} can have unbounded solutions in $X_0$. This leads to working in a function space that allows limited exponential growth both at plus and minus infinity. To do this, let $E$ be a Banach space, $\eta,s \in \mathbb{R}$ and define
\begin{equation*}
    \BC_{s}^{\eta}(\mathbb{R},E) := \bigg \{ f \in C(\mathbb{R},E) \ : \ \sup_{t \in \mathbb{R}} e^{-\eta|t-s|}\|f(t)\| < \infty \bigg \},
\end{equation*}
with the weighted supremum norm
\begin{equation*}
    \|f\|_{\eta,s} := \sup_{t \in \mathbb{R}} e^{-\eta|t-s|}\|f(t)\|,
\end{equation*}
such that $\BC_{s}^{\eta}(\mathbb{R},E)$ becomes a Banach space. Before we start working with the inhomogeneous equation \eqref{eq:variation constants CMT}, let us first derive some properties of the homogeneous equation
\begin{equation} \label{eq:homogeneous CMT}
    u(t) = U(t,s)u(s), \quad (t,s) \in \Omega_{J},
\end{equation}
on some interval $J \subseteq \mathbb{R}$. A solution of \eqref{eq:homogeneous CMT} is defined similarly as in \Cref{subsec: linear perturbation}. We have the following result that connects the center eigenspace $X_0(s)$ with $\BC_{s}^{\eta}(\mathbb{R},X)$ and the proof is inspired by \cite[Lemma 29]{Janssens2020} and \cite[Lemma 5.2.1]{Church2018}.
\begin{proposition} \label{prop:X0s}
Let $\eta \in (0,\min \{-a,b\})$ and $s \in \mathbb{R}$. Then
\begin{align*}
    X_0(s) = \{ \varphi \in X \ : & \  \mbox{there exists a solution of \eqref{eq:homogeneous CMT} on $\mathbb{R}$ through $\varphi$ belonging to } \BC_{s}^{\eta}(\mathbb{R},X) \}.
\end{align*}
\end{proposition}
\begin{proof}
Let $\varphi \in X_0(s)$, then $u_\varphi : \mathbb{R} \to X$ defined by $u_\varphi(t) := U(t,s) \varphi = U_{0}(t,s) \varphi$ is a solution of \eqref{eq:homogeneous CMT} on $\mathbb{R}$ through $\varphi$. Let us now show that $u_\varphi \in \BC_{s}^\eta(\mathbb{R},X)$. Let $\varepsilon \in (0,\eta]$ be given. It follows from the exponential trichotomy of \Cref{hyp:CMT} that 
\begin{equation*}
    e^{-\eta |t-s|} \|u_\varphi(t)\| =  e^{-\eta |t-s|} \|U_0(t,s) \varphi\| \leq K_\varepsilon e^{(\varepsilon -\eta) |t-s|} \|\varphi\| \leq K_\varepsilon\|\varphi\|, \quad \forall t,s \in \mathbb{R},
\end{equation*}
since $\varepsilon- \eta < 0$. Taking the supremum over $t \in \mathbb{R}$ yields $u_\varphi \in \BC_{s}^\eta(\mathbb{R},X)$.

Conversely, suppose that $\varphi \in X$ admits a solution $u_\varphi \in \BC_{s}^\eta(\mathbb{R},X)$ of \eqref{eq:homogeneous CMT} on $\mathbb{R}$ that goes through $\varphi$ at time $s$ i.e. $u_\varphi(s) = \varphi$. We want to show that $P_{\pm}(s) \varphi = 0$ because then $\varphi =  (P_{-}(s) + P_0(s) + P_{+}(s)) \varphi = P_0(s)\varphi$ so $\varphi \in X_0(s)$. To do this, let us first show that $P_{+}(s) \varphi = 0$. Take $t \geq s$ and $\varepsilon \in (0,\eta]$, then
\begin{equation*}
    \|P_{+}(s) \varphi \| = \|U_{+}(s,t) P_{+}(t) u_\varphi(t) \| \leq K_\varepsilon e^{b(s-t)} N \|u_\varphi(t)\|, \quad \forall t \geq s.
\end{equation*}
It follows for $t \geq \max \{s,0\}$ that
\begin{equation*}
    e^{-\eta t} \|u_\varphi(t)\| \geq \frac{e^{-bs}}{K_\varepsilon N} e^{(b-\eta)t} \|P_{+}(s) \varphi \| \to \infty, \quad \mbox{ as } t \to \infty,
\end{equation*}
unless $P_{+}(s) \varphi = 0$. To prove $P_{-}(s) \varphi = 0$, take $t \leq s$ and $\varepsilon \in (0,\eta]$, then
\begin{equation*}
    \|P_{-}(s) \varphi \| = \|U_{-}(s,t) P_{-}(t) u_\varphi(t) \| \leq K_\varepsilon e^{a(s-t)}N \|u_\varphi(t)\|.
\end{equation*}
It follows for $t \leq \min \{s,0\}$ that
\begin{equation*}
    e^{-\eta t} \|u_\varphi(t)\| \geq \frac{e^{-as}}{K_\varepsilon N} e^{(a+ \eta)t} \|P_{-}(s)\varphi\| \to \infty, \quad \mbox{ as } t \to -\infty,
\end{equation*}
unless $P_{-}(s) \varphi = 0$. Hence $P_\pm(s) = 0$ and so $\varphi \in X_0(s)$.
\end{proof}

\subsection{Bounded solutions of the linear inhomogeneous equation} \label{subsec: bounded solutions}
Let $f : \mathbb{R} \to X^{\odot \star}$ be a continuous function and consider the linear inhomogeneous integral equation
\begin{equation} \label{eq:inhomogeneous CMT}
    u(t) = U(t,s)u(s) + j^{-1} \int_s^t U^{\odot \star}(t,\tau)f(\tau) d\tau, \quad (t,s) \in \Omega_J,
\end{equation}
on an interval $J \subseteq \mathbb{R}$. A solution of \eqref{eq:inhomogeneous CMT} is defined similarly as in \Cref{subsec: linear perturbation}. To prove existence of a center manifold, we need a pseudo-inverse of bounded solutions of \eqref{eq:inhomogeneous CMT}. To do this, define (formally) for any $\eta \in (0, \min \{-a,b \})$ and $s \in \mathbb{R}$ the operator $\mathcal{K}_{s}^{\eta} : \BC_s^\eta(\mathbb{R},X^{\odot \star}) \to \BC_s^\eta(\mathbb{R},X)$ as
\begin{align*}
    (\mathcal{K}_{s}^{\eta} f)(t) &:= j^{-1} \int_{s}^{t}U^{\odot \star}(t,\tau) P_0^{\odot \star}(\tau) f(\tau) d\tau 
    +  j^{-1} \int_{\infty}^{t}U^{\odot \star}(t,\tau) P_+^{\odot \star}(\tau) f(\tau) d\tau\\
    & \ +  j^{-1} \int_{-\infty}^{t}U^{\odot \star}(t,\tau) P_{-}^{\odot \star}(\tau) f(\tau) d\tau, \quad \forall f \in \BC_s^\eta(\mathbb{R},X^{\odot \star}),
\end{align*}
and we have to check that this is indeed a well-defined operator. This will be proven in the following proposition and also the fact that $\mathcal{K}_{s}^{\eta}$ is precisely the pseudo-inverse we are looking for. The proof is inspired by \cite[Proposition 30]{Janssens2020} and \cite[Lemma 5.2.3]{Church2018}.
\begin{proposition} \label{prop:ketas}
Let $\eta \in (0, \min \{-a,b \})$ and $s \in \mathbb{R}$. The following properties hold.
\begin{enumerate}
    \item $\mathcal{K}_{s}^{\eta}$ is a well-defined bounded linear operator. Moreover, the operator norm $\|\mathcal{K}_{s}^{\eta}\|$ is bounded above independent of $s$.
    \item $\mathcal{K}_{s}^{\eta}f$ is the unique solution of \eqref{eq:inhomogeneous CMT} in $\BC^{\eta}_s(\mathbb{R},X)$ with vanishing $X_0(s)$-component at time $s$.
    \item The map from  $\BC_s^0(\mathbb{R},X^{\odot \star})$ to $\BC_s^0(\mathbb{R},X)$ given by $f \mapsto (I-P_0(\cdot))(\mathcal{K}_s^{0}f)(\cdot)$ is well-defined, linear and bounded above independent of $s$.
\end{enumerate}
\end{proposition}
\begin{proof}
We start by proving the first assertion. Let $\varepsilon \in (0,\eta)$ be given and notice that for a given $f \in \BC^{\eta}_s(\mathbb{R},X^{\odot \star})$, the three integrals in the definition of $\mathcal{K}_s^\eta$ define functions $I_{0}(\cdot,s) :  \mathbb{R} \to X^{\odot \star}$ and $I_{i} :  \mathbb{R} \to X^{\odot \star}$ for $i \in \{+,-\}$. We have to show that $I_0(\cdot,s)$ and $I_{i}$ are well-defined continuous functions that take values in $j(X)$ and satisfy certain estimates.\\

\noindent
$I_0(\cdot,s)$: The straightforward estimate
\begin{equation} \label{eq:I0}
    \|I_0(t,s)\| \leq K_\varepsilon N \|f\|_{\eta,s} \frac{e^{\eta |t-s|}}{\eta - \varepsilon} < \infty, \quad \forall t \in \mathbb{R},
\end{equation}
proves that $I_0(\cdot,s)$ is a well-defined weak$^\star$ integral. Let $\tau \in [s,t]$ be given. By \Cref{hyp:X0+sunstar} we know that $P_0^{\odot \star}(\tau) f(\tau) \in j(X_0(\tau))$ and so
\begin{equation*}
    U^{\odot \star}(t,\tau)P_0^{\odot \star}(\tau) f(\tau) = U^{\odot \star}(t,\tau)j j^{-1} P_0^{\odot \star}(\tau) f(\tau) = jU_{0}(t,\tau) j^{-1} P_0^{\odot \star}(\tau) f(\tau).
\end{equation*}
Hence,
\begin{equation*}
    I_0(t,s) = j \int_s^{t} U_0(t,\tau) j^{-1} P_0^{\odot \star}(\tau)f(\tau) d\tau \in j(X), \quad \forall t \in \mathbb{R}.
\end{equation*}
The map $I_0(\cdot,s)$ is continuous due to \Cref{lemma:wk*integral Usunstar} because $[s,t]$ is compact and the maps $P_{0}^{\odot \star}$ and $f$ are is continuous. \\

\noindent
$I_{+}:$ Notice that
\begin{equation} \label{eq:estimate I+}
    \|I_+(t)\| \leq K_\varepsilon N \|f\|_{\eta,s}e^{bt} \int_t^\infty e^{-b \tau + \eta |\tau - s|} d\tau, \quad \forall t \in \mathbb{R},
\end{equation}
and to prove norm boundedness of $I_+$, we have to evaluate the integral in the last estimate above. A calculation shows that
\begin{equation} \label{eq:expo integral}
    \int_t^\infty e^{-b \tau + \eta |\tau - s |} d\tau =
    \begin{dcases}
    \frac{e^{-bt}}{b- \eta} e^{\eta(t-s)}, \quad & t \geq s \\
    \frac{e^{-bt}}{b+\eta}e^{\eta (s-t)} - \frac{e^{-bs}}{b+\eta} + \frac{e^{-bs}}{b-\eta}, \quad &t \leq s.
    \end{dcases}
\end{equation}
We want to estimate the $t\leq s$ case. Notice that for real numbers $ \alpha \geq \beta$ we have
\begin{equation*}
    (\alpha - \beta) \bigg( \frac{1}{b+\eta} - \frac{1}{b-\eta} \bigg) = \frac{-2 \eta (\alpha - \beta)}{(b+ \eta)(b-\eta)} \leq 0,
\end{equation*}
since $\eta < b$ by assumption. Hence,
\begin{equation*}
    \frac{\alpha}{b+\eta} + \frac{\beta}{b-\eta} \leq \frac{\alpha}{b-\eta} + \frac{\beta}{b+\eta}.
\end{equation*}
We want to replace $\alpha$ by $e^{-b t + \eta s - \eta t}$ and $\beta$ by $e^{-bs}$ and therefore we have to show that $-bt + \eta s - \eta t + bs \geq 0$ which is true because $-bt + \eta s - \eta t + bs = (s-t)(b + \eta) \geq 0$ since $s-t \geq 0$. Filling this into \eqref{eq:expo integral} yields
\begin{equation*}
    \int_t^\infty e^{-b \tau + \eta |\tau - s |} d\tau \leq \frac{e^{-bt}}{b-\eta} e^{\eta|t-s|}, \quad \forall t,s \in \mathbb{R}.
\end{equation*}
Filling this back into \eqref{eq:estimate I+} yields
\begin{equation} \label{eq:I+}
    \|I_{+}(t)\| \leq K_\varepsilon N \|f\|_{\eta,s}\frac{e^{\eta |t-s|}}{b-\eta} < \infty, \quad \forall t \in \mathbb{R}.
\end{equation}
and so we conclude that $I_{+}$ is well-defined. Let $\tau \in [t,\infty)$ be given. By \Cref{hyp:X0+sunstar} we know that $P_+^{\odot \star}(\tau) f(\tau) \in j(X_+(\tau))$ and so
\begin{equation*}
    I_{+}(t) = j \int_t^\infty U_{+}(t,\tau)j^{-1}P_{+}^{\odot \star}(\tau)f(\tau) d\tau \in j(X), \quad \forall t \in \mathbb{R}.
\end{equation*}
As $U^{\odot \star}(t,\tau)$ restricted to $j(X_{0}^{+}(\tau))$ is invertible, we can adjust the proof from \Cref{lemma:wk*integral Usunstar} to prove continuity of the limiting function $v(\cdot,\infty,\cdot,g)$ for a continuous function $g : [t,\infty) \to X^{\odot \star}$ under the assumption that $I_{+}$ is bounded in norm. The fact that $I_{+}$ is bounded in norm follows from \eqref{eq:I+} and the continuity of $g$ holds because $P_+^{\odot \star}$ and $f$ are continuous. \\

\noindent
$I_{-}:$ Notice that
\begin{align*}
    \|I_{-}(t)\| &\leq K_\varepsilon N \|f\|_{\eta,s}e^{at}\int_{-\infty}^t e^{-a\tau + \eta |\tau - s|} d\tau, \quad \forall t \in \mathbb{R},
\end{align*}
where this last integral is closely related to \eqref{eq:expo integral}. A similar calculation shows that
\begin{equation} \label{eq:I-}
    \|I_{-}(t)\| \leq  K_\varepsilon N \|f\|_{\eta,s} \frac{e^{\eta|t - s|}}{-a - \eta} < \infty, \quad \forall t \in \mathbb{R},
\end{equation}
which proves that $I_{-}$ is well-defined. With the notation from \Cref{lemma:wk*integral Usunstar} we have that $I_{-}(t) = v(t,t,-\infty,g)$ with the continuous map $g$ defined as $g (\tau) = P_{-}^{\odot \star}(\tau)f(\tau)$ for all $\tau \in (-\infty,t]$, since $P_{-}^{\odot \star}$ and $f$ are assumed to be continuous. We conclude from this lemma that $I_{-}(t)$ takes values in $j(X)$ for all $t \in \mathbb{R}$ and that $I_{-}$ is continuous.

Due to linearity, we have that $\mathcal{K}_s^\eta f \in C(\mathbb{R},X)$ and combining the estimates \eqref{eq:I0}, \eqref{eq:I+} and \eqref{eq:I-} yield
\begin{equation*}
    \|\mathcal{K}_s^\eta\|_{\eta,s} \leq \|j^{-1}\| K_\varepsilon N \bigg( \frac{1}{\eta - \varepsilon} + \frac{1}{b - \eta} + \frac{1}{-a - \eta} \bigg) < \infty,
\end{equation*}
which implies $\mathcal{K}_s^\eta$ is a bounded linear operator from $\BC_s^\eta(\mathbb{R},X^{\odot \star})$ to $\BC_s^\eta(\mathbb{R},X)$ .

Let us now prove the second assertion by showing first that $\mathcal{K}_s^\eta$ is indeed a solution of \eqref{eq:inhomogeneous CMT}. Let $f \in \BC_s^\eta(\mathbb{R},X)$ and set $u = \mathcal{K}_s^\eta f$. Then, a straightforward computation shows that
\begin{equation*}
    U(t,s)u(s) + j^{-1} \int_s^t U^{\odot \star}(t,\tau)f(\tau) d\tau  = u(t),
\end{equation*}
and so $u$ is indeed a solution of \eqref{eq:inhomogeneous CMT}. Let us now prove that $u$ has vanishing $X_0(s)$-component at time $s$ i.e. $P_0(s)u(s) = 0$. The mutual orthogonality of the projections implies
\begin{align*}
    P_0(s)u(s) &= P_0(s) \bigg( j^{-1} \int_\infty^s U^{\odot \star}(s,\tau)P_{+}^{\odot \star}(\tau)f(\tau) d\tau \\
    &+ j^{-1} \int_{-\infty}^s U^{\odot \star}(s,\tau)P_{-}^{\odot \star}(\tau)f(\tau) d\tau\bigg) \\
    &= j^{-1} \int_\infty^s U^{\odot \star}(s,\tau ) P_0^{\odot \star}(\tau)P_{+}^{\odot \star}(\tau)f(\tau) d\tau \\
    &+ j^{-1} \int_{-\infty}^s  U^{\odot \star}(s,\tau) P_0^{\odot \star}(\tau) P_{-}^{\odot \star}(\tau)f(\tau) d\tau \\
    &= 0.
\end{align*}
It only remains to show that $u$ is the unique solution of \eqref{eq:inhomogeneous CMT} in $\BC_s^\eta(\mathbb{R},X)$. Let $v \in \BC_s^\eta(\mathbb{R},X)$ be another solution of \eqref{eq:inhomogeneous CMT} with vanishing $X_0(s)$-component at time $s$. Then the function $w := u - v$ is an element of $\BC_s^\eta(\mathbb{R},X)$ and satisfies $w(t) = U(t,s)w(s)$ for $(s,t) \in \Omega_\mathbb{R}$. \Cref{prop:X0s} shows us that $w(s) \in X_0(s)$ and notice that $P_0(s)w(s) = 0$ since $u$ and $v$ have both vanishing $X_0(s)$-component at time $s$. From \Cref{hyp:CMT} we know that $w(t) = U_0(t,s)w(s)$ is in $X_0(t)$ for all $t \in \mathbb{R}$. Hence,
\begin{equation*}
    P_0(t)w(t) = P_0(t)U_0(t,s)w(s) = U_0(t,s)P_0(s)w(s) = 0, \quad \forall t \in \mathbb{R},
\end{equation*}
and so $w = 0$ i.e. $u = v$.

Let us now prove the third assertion. Take $f \in  \BC_{s}^0(\mathbb{R},X^{\odot \star})$, then 
\begin{equation*}
    \|(\mathcal{K}_s^0f)(t)\| \leq \|j^{-1}\| K_\varepsilon N \|f\|_{0,s} \bigg( \frac{1}{-a} + \frac{1}{b} \bigg), \quad \forall t \in \mathbb{R}.
\end{equation*}
and because $\mathcal{K}_s^0f$ has vanishing $X_0(s)$-component at time $s$, we get
\begin{equation*}
    \|(I-P_0(t)(\mathcal{K}_s^0f)(t)\| \leq \|j^{-1}\| K_\varepsilon N \|f\|_{0,s} \bigg( \frac{1}{-a} + \frac{1}{b} \bigg),
\end{equation*}
and so $\|(I-P_0(\cdot)(\mathcal{K}_s^0f)(\cdot)\|_{0,s} \leq \|j^{-1}\| K_\varepsilon N \|f\|_{0,s} ( \frac{1}{-a} + \frac{1}{b})$ which shows that $(I-P_0(\cdot)(\mathcal{K}_s^0f)(\cdot)$ is in $\BC_s^0(\mathbb{R},X)$. Because the projections are linear and $\mathcal{K}_s^0$ is linear, we have that $f \mapsto (I-P_0(\cdot)(\mathcal{K}_s^0f)(\cdot)$ is linear. Clearly the operator norm of $f \mapsto (I-P_0(\cdot))(\mathcal{K}_s^0f)(\cdot)$  is bounded above by $\|j^{-1}\| K_\varepsilon N ( \frac{1}{-a} + \frac{1}{b} ) < \infty$ and so this map is bounded, independent of $s$.
\end{proof}

\subsection{Modification of the nonlinearity}
\label{subsec: modification nonlinearity}
To prove the existence of a center manifold, a key step will be to use Banach fixed point theorem on some specific fixed point operator. This operator we will be of course linked to the inhomogeneous equation \eqref{eq:inhomogeneous CMT}. However, we can not expect that any nonlinear operator $R(t,\cdot) : X \to X^{\odot \star}$ for fixed $t \in \mathbb{R}$ will impose a Lipschitz condition on the fixed point operator that will be constructed. As we are only interested in the local behavior of solutions near zero, we can modify the nonlinearity $R(t,\cdot)$ outside a ball of radius $\delta > 0$ such that eventually the fixed point operator will become a contraction. To modify this nonlinearity, introduce the $C^{\infty}$-smooth cut-off function $\xi : [0,\infty) \to \mathbb{R}$ as
\begin{equation*}
    \xi(s) \in
    \begin{cases}
    \{ 1 \}, \quad &0 \leq s \leq 1, \\
    [0,1], \quad &0 \leq s \leq 2,\\
    \{ 0 \}, \quad & s \geq 2,
    \end{cases}
\end{equation*}
and define then for any $\delta > 0$ and $s\in \mathbb{R}$ the \emph{$\delta$-modification} of $R$ as the operator $R_{\delta,s} : \mathbb{R} \times X \to X^{\odot \star}$ with action
\begin{equation*}
   R_{\delta,s}(t,u) := R(t,u) \xi \bigg( \frac{\|P_0(s)u\|}{N \delta} \bigg) \xi \bigg( \frac{\|(P_{-}(s) + P_{+}(s))u\|}{N \delta} \bigg), \quad \forall (t,u) \in \mathbb{R} \times X. 
\end{equation*}
This $\delta$-modification of $R$ will ensure that the nonlinearity is globally Lipschitz. The proof is very similar to that of \cite[Proposition 32]{Janssens2020} and therefore omitted.
\begin{proposition} \label{prop:Lrdelta}
For $s \in \mathbb{R}$ and sufficiently small $\delta > 0$, the operator $R_{\delta,s}(t,\cdot)$ is globally Lipschitz continuous for any $t\in \mathbb{R}$ with Lipschitz constant $L_{R_{\delta}} \to 0$ as $\delta \downarrow 0$ independent of $s$.
\end{proposition}

Let us introduce now for a given $\delta$-modification of $R$ the \emph{substitution operator} $\tilde{R}_{\delta,s} : \BC_s^{\eta}(\mathbb{R},X) \to \BC_s^{\eta}(\mathbb{R},X^{\odot \star})$ as
\begin{equation*} 
\tilde{R}_{\delta,s}(u) := R_{\delta,s}(\cdot,u(\cdot)), \quad \forall u \in \BC_s^{\eta}(\mathbb{R},X),
\end{equation*}
and we show that this operator inherits the same properties as $R_{\delta,s}$. The proof is analogous to that of \cite[Corollary 33]{Janssens2020} and therefore omitted.
\begin{corollary} \label{cor:lipschitz2}
For $s \in \mathbb{R}$ and sufficiently small $\delta > 0$, the map  $\tilde{R}_{\delta,s}$ is well-defined, globally Lipschitz continuous with Lipschitz constant $L_{{R}_{\delta}} \to 0$ as $\delta \downarrow 0$ independent of $\eta$ and $s$.
\end{corollary}

\subsection{Existence of a Lipschitz center manifold}
\label{subsec: Lipschitz CM}
Our next goal is to define a parameterized fixed point operator such that its fixed points correspond to exponentially bounded solutions on $\mathbb{R}$ of the modified equation
\begin{equation} \label{eq:variation constants Rdelta}
    u(t) = U(t,s)u(s) + j^{-1} \int_s^t U^{\odot \star}(t,\tau)R_{\delta,s}(\tau, u(\tau)) d\tau, \quad -\infty < s \leq t < \infty,
\end{equation}
for some small $\delta > 0$. For a given $\eta \in (0,\min \{-a,b \})$ and $s\in \mathbb{R}$, we define the fixed point operator $ \mathcal{G}_s : \BC_s^\eta(\mathbb{R},X) \times X_0(s) \to \BC_s^\eta(\mathbb{R},X^{\odot \star})$ as
\begin{equation*}
   \mathcal{G}_s(u,\varphi) := U(\cdot,s)\varphi + \mathcal{K}_s^\eta(\tilde{R}_{\delta,s}(u)), \quad \forall (u,\varphi) \in \BC_s^\eta(\mathbb{R},X) \times X_0(s),
\end{equation*}
where its second argument in $X_0(s)$ is treated as a parameter. We first show that $\mathcal{G}_s$ has a unique fixed point and is globally Lipschitz.

\begin{theorem} \label{thm:sol lipschitz}
Let $\eta \in (0,\min \{-a,b \})$ and $s\in \mathbb{R}$ be given. If $\delta > 0$ is sufficiently small, then the following two statements hold.

\begin{enumerate}
    \item For every $\varphi \in X_0(s)$ the equation $ u = \mathcal{G}_s(u,\varphi)$ has a unique solution $u = u_s^\star(\varphi)$.
    \item The map $u_s^\star : X_0(s) \to \BC_s^\eta(\mathbb{R},X)$ is globally Lipschitz and satisfies $u_s^\star(0) = 0$.
\end{enumerate}
\end{theorem}
\begin{proof}
Let $\varepsilon \in (0,\eta) $ be given. Take $u,v \in \BC_s^\eta(\mathbb{R},X)$ and $\varphi,\psi \in X_0(s)$ because then
\begin{align*}
    \|\mathcal{G}_s(u,\varphi) - \mathcal{G}_s(v,\psi)\|_{\eta,s} & \leq \sup_{t \in \mathbb{R}}e^{-\eta|t-s|}\|U_0(t,s)(\varphi - \psi)\| + \|\mathcal{K}_s^\eta\| L_{{R}_{\delta}} \|u-v\|_{\eta,s} \\
    &\leq K_\varepsilon\|\varphi- \psi\| + \|\mathcal{K}_s^\eta\| L_{{R}_{\delta}} \|u-v\|_{\eta,s},
\end{align*}
where we used the fact that $\varepsilon < \eta$ and the exponential trichotomy on the center eigenspace since $\varphi - \psi \in X_0(s)$. By \Cref{cor:lipschitz2} there exists a $\delta_1 > 0$ such that for all $0 < \delta \leq \delta_1$ we have that $ L_{{R}_{\delta}} \|\mathcal{K}_s^\eta\| \leq \frac{1}{2}$. 

1. Set $\psi = \varphi$ in the previous estimate, because then for $0 \leq \delta \leq \delta_1$ we have that
\begin{equation*}
    \|\mathcal{G}_s(u,\varphi) - \mathcal{G}_s(v,\varphi)\|_{\eta,s} \leq \frac{1}{2} \|u-v\|_{\eta,s},
\end{equation*}
which means that $\mathcal{G}_s(\cdot,\varphi)$ is a contraction on the Banach space $\BC_s^\eta(\mathbb{R},X)$ equipped with the $\|\cdot\|_{\eta,s}$-norm. It follows from the contraction mapping principle that $\mathcal{G}_s(\cdot,\varphi)$ has a unique fixed point $u_s^\star(\varphi)$. 

2. Let $u_s^\star(\varphi)$ and $u_s^\star(\psi)$ be unique fixed points of the operators $\mathcal{G}_s(u,\varphi)$ and $\mathcal{G}_s(u,\psi)$ respectively. Then,
\begin{align*}
    \|u_s^\star(\varphi) - u_s^\star(\psi)\|_{\eta,s} &= \|\mathcal{G}_s(u_s^\star(\varphi),\varphi) - \mathcal{G}_s(u_s^\star(\psi),\psi)\|_{\eta,s} \\
    &\leq K_\varepsilon \|\varphi - \psi\| + \frac{1}{2}\|u_s^\star(\varphi) - u_s^\star(\psi)\|_{\eta,s}.
\end{align*}
This implies that $\|u_s^\star(\varphi) - u_s^\star(\psi)\|_{\eta,s} \leq 2 K_\varepsilon \|\varphi - \psi\|$ and so $u^\star_s$ is globally Lipschitz. Since $u_s^\star(0) = \mathcal{G}_s(u_s^\star(0),0) = 0$ the second assertion follows.
\end{proof}

The map $\mathcal{C} : X_0 \to X$ defined by
\begin{equation} \label{eq:map C}
   \mathcal{C}(t,\varphi) := u_t^\star(\varphi)(t), \quad \forall (t,\varphi) \in X_0,
\end{equation}
ensures the existence of a center manifold in the following way.
\begin{definition}
The \emph{global center manifold} for \eqref{eq:variation constants Rdelta} is defined as the image
\begin{equation*}
        \mathcal{W}^c := \mathcal{C}(X_0),
\end{equation*}
whose $s$-fibers are defined as $\mathcal{W}^c(s) := \{\mathcal{C}(s,\varphi) \in X \ : \ \varphi \in X_0(s) \}$.
\end{definition}

\begin{proposition} \label{prop:Wcs karak}
If $\eta \in (0,\min \{-a,b\})$ and $s \in \mathbb{R}$, then
\begin{align*}
    \mathcal{W}^c(s)  =  \{ \varphi \in X    &:    \mbox{there exists a solution of \eqref{eq:variation constants Rdelta} on $\mathbb{R}$ through $\varphi$ belonging to } \BC_{s}^{\eta}(\mathbb{R},X) \}.
\end{align*}
\end{proposition}
\begin{proof}
Let $\varphi \in \mathcal{W}^c(s)$, then $\varphi = \mathcal{C}(s,\psi) = u_{s}^\star(\psi)(s)$ for some $\psi \in X_0(s)$. We show that $u = u_{s}^\star(\psi)$ is a solution of \eqref{eq:variation constants Rdelta} on $\mathbb{R}$ through $\varphi$ which belongs to  $\BC_{s}^{\eta}(\mathbb{R},X)$. Part 2 of \Cref{prop:ketas} shows us that $\mathcal{K}_s^\eta \tilde{R}_{\delta,s}(u)$ is the unique solution of \eqref{eq:inhomogeneous CMT} in $\BC_{s}^{\eta}(\mathbb{R},X)$ with $f = \tilde{R}_{\delta,s}(u)$. Because $u = u_{s}^\star(\psi)$ is a fixed point of $\mathcal{G}_s(\cdot,\psi)$ we obtain
\begin{align*}
    u(t) &= U(t,s)\psi + (\mathcal{K}_s^\eta \tilde{R}_{\delta,s}(u))(t) \\
    &= U(t,s)\psi + U(t,s)(\mathcal{K}_s^\eta \tilde{R}_{\delta,s}(u))(s) + j^{-1} \int_s^t U^{\odot \star}(t,\tau) R_{\delta,s}(\tau,u(\tau)) d\tau \\
    &= U(t,s)\psi + U(t,s)(u(s) - \psi) + j^{-1} \int_s^t U^{\odot \star}(t,\tau) R_{\delta,s}(\tau,u(\tau)) d\tau \\
    &= U(t,s)u(s) + j^{-1} \int_s^t U^{\odot \star}(t,\tau) R_{\delta,s}(\tau,u(\tau)) d\tau
\end{align*}
for all $(t,s) \in \Omega_{\mathbb{R}}$. This shows that $u$ is a solution of \eqref{eq:variation constants Rdelta} on $\mathbb{R}$ through $\varphi =  u_s^\star(\psi)(s) = u(s)$ which belongs to $\BC_{s}^{\eta}(\mathbb{R},X)$.

To show the converse, let $\varphi \in X$ be given such that there exists a solution $u$ in $\BC_{s}^{\eta}(\mathbb{R},X)$ of \eqref{eq:variation constants Rdelta} that satisfies $u(s) = \varphi$. For $(t,s) \in \Omega_\mathbb{R}$ it is possible to rewrite \eqref{eq:variation constants Rdelta} as
\begin{align*}
    u(t) &= U(t,s)P_0(s) u(s) + U(t,s)(I-P_0(s))u(s) + j^{-1} \int_s^t U^{\odot \star}(t,\tau) R_{\delta,s}(\tau,u(\tau)) d\tau\\
    &= U(t,s)P_0(s) u(s) + (\mathcal{K}_s^\eta \tilde{R}_{\delta,s}(u))(t)
\end{align*}
where part 2 of \Cref{prop:ketas} was used in the last equality. Hence, if we define $\psi := P_0(s) u(s)$, then
\begin{equation*}
    u(t) = U(t,s)\psi + (\mathcal{K}_s^\eta \tilde{R}_{\delta,s}(u))(t), \quad \forall (t,s) \in \Omega_\mathbb{R},
\end{equation*}
which implies $u = \mathcal{G}_s(u,\psi)$. As we know from \Cref{thm:sol lipschitz} that this fixed point problem has a unique solution $u = u_s^\star(\psi)$, we have that $\varphi = u(s) = u_s^\star(\psi) = \mathcal{C}(s,\psi) \in \mathcal{W}^c(s)$, which completes the proof.
\end{proof}

Recall from part 2 of \Cref{thm:sol lipschitz} that for a fixed $t \in \mathbb{R}$ the map $u_t^\star:X_0(t) \to \BC_t^\eta(\mathbb{R},X)$ is globally Lipschitz. Hence, from the definition of the map $\mathcal{C}$ given in \eqref{eq:map C}, we see that the map $\mathcal{C}(t,\cdot) : X_0(t) \to X$ is globally Lipschitz, where the Lipschitz constant depends on $t$ and so this shows that the map $\mathcal{C}$ is only \emph{fiberwise Lipschitz}. However, it is proven in \Cref{cor:lipschitzCMT} that the Lipschitz constant can be chosen independently of the fiber, and so we can say that $\mathcal{W}^c$ is the \emph{global Lipschitz center manifold}.

Let $B_\delta(X)$ denote the open ball centered around the origin in $X$ with radius $\delta > 0$. From the cut-off function $\xi$, it is clear that the restrictions of $R(t,\cdot)$ and $R_{\delta,s}(t,\cdot)$ to this ball are equal for any $t \in \mathbb{R}$. Hence, if we restrict the unknown function $u$ to take only values in $B_\delta(X)$, then  \eqref{eq:variation constants CMT} and \eqref{eq:variation constants Rdelta} coincide as well.

\begin{definition}
The \emph{local center manifold} $\mathcal{W}_{\loc}^c$ for \eqref{eq:variation constants CMT} is defined as the image
\begin{equation*}
    \mathcal{W}_{\loc}^c := \mathcal{C}(\{ (t,\varphi) \in X_0 \ : \ \mathcal{C}(t,\varphi) \in B_\delta(X) \}).
\end{equation*}
\end{definition}

By construction, the center manifolds $\mathcal{W}^c$ and $\mathcal{W}_{\loc}^c$ non-canonically depend on the choice of $\delta$ and the cut-off function $\xi$. This is the famous non-uniqueness property of the center manifold.

\subsection{Properties of the center manifold} \label{subsec: properties CM}
We will show some important properties that the local center manifold $\mathcal{W}_{\loc}^c$ enjoys. We start off with the following result, that is inspired by {\cite[Theorem 5.4.2]{Church2018}} and \cite[Corollary 38]{Janssens2020}.

\begin{theorem} \label{thm:CMT properties}
The local center manifold $\mathcal{W}_{\loc}^c$ has the following properties.
\begin{enumerate}
    \item $\mathcal{W}_{\loc}^c$ is locally positively invariant: if $(s,\varphi) \in \mathbb{R} \times \mathcal{W}_{\loc}^c$ and $s < t_\varphi \leq \infty$ are such that $S(t,s,\varphi) \in B_\delta(X)$ for all $t \in [s,t_\varphi),$ then $S(t,s,\varphi) \in \mathcal{W}_{\loc}^c$.
    \item $\mathcal{W}_{\loc}^c$ contains every solution of \eqref{eq:variation constants CMT} that exists on $\mathbb{R}$ and remains sufficiently small for all positive and negative time i.e. if $u : \mathbb{R} \to B_\delta(X)$ is a solution of \eqref{eq:variation constants CMT}, then $u(t) \in \mathcal{W}_{\loc}^c$ for all $t\in \mathbb{R}$.
    \item If $(s,\varphi) \in \mathbb{R} \times \mathcal{W}_{\loc}^c$, then $S(t,s,\varphi) = u_t^\star(P_0(t)S(t,s,\varphi))(t) = \mathcal{C}(t,P_0(t)S(t,s,\varphi))$ for all $t \in [s,t_\varphi)$.
    \item $0 \in \mathcal{W}_{\loc}^c$ and $\mathcal{C}(t,0) = 0$ for all $t\in \mathbb{R}$.
\end{enumerate}
\end{theorem}
\begin{proof}
1. \Cref{prop:Wcs karak} implies that there exists a solution $u \in \BC_s^{\eta}(\mathbb{R},X)$ of \eqref{eq:variation constants Rdelta} through $\varphi$ that can be chosen to be $u(s) = \varphi$. So, $S(\cdot,s,\varphi)$ and $u$ are both solutions of \eqref{eq:variation constants Rdelta} on $[s,t_\varphi)$ and $S(s,s,\varphi) = \varphi = u(s)$. This means $S(\cdot,s,\varphi)$ and $u$ coincide on $[s,t_\varphi)$ by uniqueness of solutions. This means $S(t,s,\varphi) \in \mathcal{W}^c(t) \subset \mathcal{W}^{c}$ for all $t \in [s,t_\varphi)$. Since $\mathcal{W}_{\loc}^{c} = \mathcal{W}^{c} \cap B_\delta(X)$ the result follows. 

2. If $u$ is such a solution, then $u \in \BC_{s}^\eta(\mathbb{R},X)$. The assumption that $u$ takes values in $B_\delta(X)$ and \Cref{prop:Wcs karak} together imply the result. 

3. Since $\varphi \in X_0(s)$ we have that $S(s,s,\varphi)= \varphi = u_s^\star(P_0(s)\varphi)(s) = \mathcal{C}(s,P_0(s)\varphi)$ and so the asserted equation holds at time $t = s$. Since $\mathcal{W}_{\loc}^{c}$ is locally positively invariant, we have that $S(t,s,\varphi) = u_t^\star(\psi(t))(t) \in \mathcal{W}_{\loc}^{c}$ for some $\psi(t) \in X_0(t)$ and we also have $u_t^\star(P_0(t)S(t,s,\varphi))(t) \in \mathcal{W}_{\loc}^{c}$. Because both solutions started at $\varphi$, we must have by uniqueness of solutions that $S(t,s,\varphi) = u_t^\star(P_0(t)S(t,s,\varphi))(t) = \mathcal{C}(t,P_0(t)S(t,s,\varphi))$.

4. Let $t \in \mathbb{R}$ be given. Notice that $\mathcal{C}(t,0) = u_t^\star(0)(t) = 0$, where the last equality follows from part 2 of \Cref{thm:sol lipschitz}. Clearly, $0 = \mathcal{C}(t,0) \in \mathcal{W}_{\loc}^{c}$.
\end{proof}

The next step is to show that the map $\mathcal{C}$ inherits the same order of smoothness as the time-dependent nonlinear perturbation $R$, namely the preselected integer $k \geq 1$. Proving additional smoothness of center manifolds requires work. A well-known technique to achieve smoothness is via the theory of scales of Banach spaces that is presented in \Cref{appendix: smoothness and periodicty}. We refer to \Cref{appendix: smoothness and periodicty} for the statements of the results and additional proofs. The main result is the following, and the proofs can be found in \Cref{cor:tangent} and \Cref{thm:periodic}.
\begin{theorem} \label{thm:smoothness}
The center manifolds $\mathcal{W}^c$ and $\mathcal{W}_{\loc}^c$ are $C^k$-smooth and their tangent bundle is $X_0$ i.e. $D_2 \mathcal{C}(t,0)\varphi = \varphi$ for all $(t,\varphi) \in X_0$. Furthermore, if the time-dependent nonlinear perturbation $R : \mathbb{R} \times X \to X^{\odot \star}$ is $T$-periodic in the first variable, then there exists a $\delta > 0$ such that $\mathcal{C}(t+T,\varphi) = \mathcal{C}(t,\varphi)$ for all $t \in \mathbb{R}$ whenever $\|\varphi\| < \delta$.
\end{theorem}

To summarize, we have proven the following center manifold theorem in a $T$-periodic setting.

\begin{theorem}[{Local center manifold}] \label{thm:LCMT}
Let $T_0$ be a $\mathcal{C}_0$-semigroup on a $\odot$-reflexive real Banach space $X$ and let $U$ be the strongly forward evolutionary system defined by \eqref{eq:T-LAIEphi} that satisfies \Cref{hyp:CMT} and \Cref{hyp:X0+sunstar}, where $B$ is a $T$-periodic time-dependent bounded linear perturbation. Suppose that the real center eigenspace $X_0(t)$, defined for all $t \in \mathbb{R}$, has dimension $1 \leq n_0 + 1 < \infty$. Furthermore, suppose that the time-dependent nonlinear perturbation $R$ is $T$-periodic in the first component, $C^k$-smooth and satisfies \eqref{eq:nonlinearterms}.

Then there exists a $ \delta > 0$ and a $C^k$-smooth map $\mathcal{C} : X_0 \to X$ such that the manifold $\mathcal{W}_{\loc}^c := \mathcal{C}(\{ (t,\varphi) \in X_0 \ : \ \mathcal{C}(t,\varphi) \in B_\delta(X) \})$ is $T$-periodic, $C^k$-smooth, $(n_0 + 1)$-dimensional and locally positively invariant for the time-dependent semiflow $S$ generated by \eqref{eq:variation constants CMT}.
\end{theorem}

\subsection{The special case of classical DDEs} \label{subsec: classical DDEs}
Let us now specify the setting of classical DDEs, such that we can apply \Cref{thm:LCMT}. Choose the Banach space $X := C([-h,0],\mathbb{R}^n)$ as the state space for some finite maximal delay $h > 0$ equipped with the supremum norm $\|\cdot\|_{\infty}$. For a given $k \geq 0$, consider a $C^{k+1}$-smooth operator $F : X \to \mathbb{R}^n$ together with the initial value problem
\begin{equation} \label{eq:DDEphi} \tag{DDE} 
    \begin{cases}
    \dot{x}(t) = F(x_t), \quad &t\geq 0,\\
    x_0 = \varphi, \quad &\varphi \in X,
    \end{cases}
\end{equation}
where the \emph{history of $x$ at time $t \geq 0$}, denoted by $x_t \in X$ is defined as
\begin{equation*}
    x_t(\theta) := x(t+\theta), \quad \forall\theta \in [-h,0].
\end{equation*}
By a \emph{solution} of $\eqref{eq:DDEphi}$ we mean a continuous function $x : [-h,t_{\varphi}) \to \mathbb{R}^n$ for some \emph{final time} $0 < t_{\varphi} \leq \infty$ that is continuously differentiable on $[0,t_{\varphi})$ and satisfies \eqref{eq:DDEphi}. When $t_{\varphi} = \infty$, we call $x$ a \emph{global solution}. We say that a function $\gamma : \mathbb{R} \to \mathbb{R}^n$ is a \emph{periodic solution} of \eqref{eq:DDEphi} if there exists a minimal $T > 0$, called the \emph{period of $\gamma$} such that $\gamma_T = \gamma_0$. We call $\Gamma := \{ \gamma_t \in X \ : \ t \in \mathbb{R} \}$ a \emph{periodic orbit} or \emph{(limit) cycle} in $X$. It follows from \cite[Corollary 10.3.1]{Hale1993} that $\gamma \in C^{k+2}(\mathbb{R},\mathbb{R}^n)$.

We want to study \eqref{eq:DDEphi} near the periodic solution $\gamma$, and it is therefore more convenient to translate $\gamma$ towards the origin. More specifically, if $x$ is a solution of \eqref{eq:DDEphi}, then for $y$ defined as $x = \gamma + y$, we have that $y$ satisfies the nonlinear time-dependent DDE 
\begin{equation} \label{eq:T-DDEphi} \tag{T-DDE}
\dot{y}(t) = L(t)y_t + G(t,y_t),\\
\end{equation}
where $L(t) := DF(\gamma_t)$ denotes the Fr\'echet derivative of $F$ evaluated at $\gamma_t$ and $G(t,\cdot) := F(\gamma_t + \cdot) - F(\gamma_t) - L(t)$ consists of solely nonlinear terms and is of the class $C^{k}$. 

Before we can understand the relation between \eqref{eq:T-DDEphi} and \eqref{eq:T-AIEphi}, we first have to apply the sun-star calculus machinery onto the setting of classical DDEs. The starting point is the \emph{trivial DDE}
\begin{equation} \label{eq:TDDEphi}
    \begin{cases}
    \dot{x}(t) = 0, \quad &t\geq 0,\\
    x_0 = \varphi, \quad &\varphi \in X,
    \end{cases}
\end{equation}
which has the unique global solution
\begin{equation} \label{eq:solutionTDDE}
    x(t) =
    \begin{cases}
    \varphi(t), \quad &-h \leq t \leq 0,\\
    \varphi(0), \quad &t \geq 0.
    \end{cases}
\end{equation}
Using this solution, we define the $\mathcal{C}_0$-semigroup $T_0$ on $X$, also called the \emph{shift semigroup}, as
\begin{equation} \label{eq:shift}
    (T_0(t)\varphi)(\theta) := 
    \begin{cases}
    \varphi(t+\theta), \quad &-h \leq t+ \theta \leq 0,\\
    \varphi(0), \quad &t+\theta \geq 0,
    \end{cases}
    \quad \forall \varphi \in X, \ t \geq 0, \ \theta \in [-h,0].
\end{equation}
Notice that $T_0$ generates the solution of \eqref{eq:solutionTDDE} in the sense that $T_0(t)\varphi = x_t$ for all $t \geq 0$. For this specific combination of $X$ and $T_0$, the abstract duality structure from \Cref{subsec: duality structure} can be constructed explicitly, see \cite[Section II.5]{Diekmann1995}. We only summarize here the basic results.

For $\mathbb{K} \in \{ \mathbb{R}, \mathbb{C} \}$ let $\mathbb{K}^n$ be the linear space of column vectors while $\mathbb{K}^{n \star}$ denotes the linear space of row vectors, both over $\mathbb{K}$. A representation theorem by F. Riesz \cite{Riesz1914} enables us to identify $X^\star = C([-h,0],\mathbb{R}^n)^\star$ with the Banach space $\NBV([0,h],\mathbb{R}^{n\star})$ consisting of functions $\zeta : [0,h] \to \mathbb{R}^{n \star}$ that are normalized by $\zeta(0) = 0$, are continuous from the right on $(0,h)$ and have bounded variation. From \eqref{eq:Xsuncharac} it turns out that
\begin{equation*}
    X^{\odot} \cong \mathbb{R}^{n \star} \times L^1([0,h],\mathbb{R}^{n\star}),
\end{equation*}
where $\cong$ stands for an isometric isomorphism and $\mathbb{R}^{n \star}$ denotes the linear space of row vectors over $\mathbb{R}$. Computing the dual of $X^{\odot}$ and afterwards the restriction to the maximal space of strong continuity yields
\begin{equation*}
    X^{\odot \star} \cong \mathbb{R}^n \times L^\infty([-h,0],\mathbb{R}^n), \quad X^{\odot \odot} \cong \mathbb{R}^n \times C([-h,0],\mathbb{R}^n).
\end{equation*}
The canonical embedding $j$ defined in \eqref{eq: j} has action $j\varphi = (\varphi(0),\varphi)$ for $\varphi \in X$, mapping $X$ onto $X^{\odot \odot}$, meaning that $X$ is $\odot$-reflexive with respect to the shift semigroup $T_0$. 

Let us now specify the time-dependent bounded linear perturbation $B$ from \Cref{subsec: linear perturbation}. For $i = 1,\dots,n$ we denote $r_i^{\odot \star} := (e_i,0)$, where $e_i$ is the $i$th standard basic vector of $\mathbb{R}^n$. It is conventional and convenient to introduce the shorthand notation
\begin{equation*}
    wr^{\odot \star} := \sum_{i = 1}^{n} w_ir_i^{\odot \star}, \quad \forall w = (w_1,\dots,w_n) \in \mathbb{R}^{n},
\end{equation*}
and note that $wr^{\odot \star} = (w,0) \in X^{\odot \star}$. We specify the time-dependent bounded linear perturbation as
\begin{equation} \label{eq:B(t)}
    B(t)\varphi := [L(t)\varphi]r^{\odot \star}, \quad \forall t \in \mathbb{R}, \ \varphi \in X,
\end{equation}
and since $F \in C^{k + 1}(X,\mathbb{R}^n)$, $t \mapsto \gamma_t$ is $T$-periodic and of the class $C^k$, we have that $B \in C^{k}( \mathbb{R}, \mathcal{L}(X,X^{\odot \star}))$ is $T$-periodic and Lipschitz continuous. It is shown in \cite[Theorem 3.1]{Diekmann1995} that there is a one-to-one correspondence between solutions of the time-dependent linear problem
\begin{equation} \label{eq:T-LDDEphi} \tag{T-LDDE}
    \begin{cases}
        \dot{y}(t) =  L(t)y_t, \quad &t \geq s, \\
        y_s = \varphi, \quad &\varphi \in X,
    \end{cases}
\end{equation}
which is \eqref{eq:T-DDEphi} with $G = 0$, and the time-dependent linear abstract integral equation \eqref{eq:T-LAIEphi}. Hence, $y_t = U(t,s)\varphi$ and so $y(t) = (U(t,s)\varphi)(0)$ for all $t \geq s$. 
Let us now specify the time-dependent nonlinear perturbation $R$ from \Cref{subsec: nonlinear perturbation} as
\begin{equation} \label{eq: R}
    R(t,\varphi) := G(t,\varphi)r^{\odot \star}, \quad \forall t \in \mathbb{R}, \ \varphi \in X,
\end{equation}
which is $T$ periodic in the first component and of the class $C^k$. As in the linear case, we have to show that there exists a one-to-one correspondence between solutions of \eqref{eq:T-DDEphi} and \eqref{eq:T-AIEphi}. A proof for this could not be found in the literature, but is given in \Cref{thm:one-to-one T-DDE AIE} with additional preparatory material presented in \Cref{appendix: variation constants}. Hence, the time-dependent semiflow $S$ presented in \eqref{eq:semiflowtimedep} generates solutions of \eqref{eq:T-DDEphi} in the sense that $y_t = S(t,s,\varphi)$ and so $y(t) = S(t,s,\varphi)(0)$ for all $t \in [s,t_\varphi)$. 

We are in the position to verify \Cref{hyp:CMT} and \Cref{hyp:X0+sunstar}. First we have to decompose $X$ in a topological direct sum \eqref{eq:decomposition X hyp}. To do this, define for any $s \in \mathbb{R}$ the \emph{monodromy operator} $U(s+T,s) \in \mathcal{L}(X)$ (at time $s$), and note that iterates of this map are compact, see \cite[Corollary XII.3.4 and Corollary XIII.2.2]{Diekmann1995}. Hence, the spectrum $\sigma(U(s+T,s))$ is a countable set consisting of $0$ and isolated eigenvalues (called \emph{Floquet multipliers}) that can possibly accumulate to $0$.

It is shown in \cite[Theorem3.3]{Diekmann1995} that the Floquet multipliers are independent of the starting time $s$ and thus well-defined. By compactness, there exist two closed $U(s+T,s)$-invariant subspaces of $X$ denoted by $E_\lambda(s)$ and $R_{\lambda}(s)$ such that
\begin{equation*}
    X = E_\lambda(s) \oplus R_{\lambda}(s).
\end{equation*}
The subspace $E_\lambda(s)$ is called the \emph{(generalized) eigenspace} (at time $s$) associated to the Floquet multiplier $\lambda$. This (generalized) eigenspace is defined as the smallest closed linear subspace that contains all $\ker((\lambda I - U(s+T,s))^j)$ for all integers $j \geq 1$. Due to compactness, it turns out that there exists a smallest integer $k_\lambda$ such that $ \cup_{j \in \mathbb{N}} \ker((\lambda I - U(s+T,s))^j) = \ker((\lambda I - U(s+T,s))^{k_\lambda})$ and hence the dimension of the generalized eigenspace $E_\lambda(s)$ is finite and called the \emph{algebraic multiplicity}.

Due to compactness, the sets of Floquet multipliers outside the unit disk $\Lambda_+ := \{ \lambda \in \sigma(U(s+T,s)) \ : \  |\lambda| > 1 \}$ and on the unit circle $\Lambda_0 := \{ \lambda \in \sigma(U(s+T,s)) \ : \ |\lambda| = 1\}$ are both finite. With each of these sets, we define the \emph{unstable eigenspace} (at time $s$) and \emph{center eigenspace} (at time $s$) as
\begin{equation*}
    X_{+}(s) := \bigoplus_{\lambda \in \Lambda_+} E_{\lambda}(s), \quad X_0(s) := \bigoplus_{\lambda \in \Lambda_0} E_{\lambda}(s)
\end{equation*}
respectively and notice that both eigenspaces are finite-dimensional. The \emph{stable eigenspace} (at time $s$) can be defined as
\begin{equation} \label{eq:X-}
    X_{-}(s) := \bigcap_{\lambda \in \Lambda_0 \cup \Lambda_+} R_{\lambda}(s),
\end{equation}
and has finite codimension. From this construction, the unstable-, center- and stable eigenspace are all closed $T$-periodic $U(s+T,s)$-invariant subspaces of $X$. This decomposition is sufficient to prove that \Cref{hyp:CMT} and \Cref{hyp:X0+sunstar} hold in the setting of classical DDEs, presented in this subsection. The verification of both hypotheses is carried out in \Cref{appendix: Spectral DDE} and hence we obtain the following.

\begin{corollary}[Local center manifold for DDEs] \label{cor:CMT DDE}
Consider \eqref{eq:DDEphi} with a $C^{k + 1}$-smooth right-hand side $F : X \to \mathbb{R}^n$ for a fixed $k \geq 0$ and a given $T$-periodic solution $\gamma$. Define the finite rank Lipschitz continuous $T$-periodic time-dependent bounded linear perturbation $B$ as in \eqref{eq:B(t)} together with the time-dependent nonlinear perturbation $R$ as in \eqref{eq: R}, that is $T$-periodic in the first component. Let $U$ denote the strongly continuous forward evolutionary system that generates solutions of \eqref{eq:T-LDDEphi} with $L(t) = DF(\gamma_t)$. Suppose that there are $1 \leq n_0 + 1 < \infty$ Floquet multipliers on the unit circle, counted with algebraic multiplicity, with corresponding $(n_0 + 1)$-dimensional real center eigenspace $X_0(t)$ defined for all $t \in \mathbb{R}$.

Then there exists a $ \delta > 0$ and a $C^k$-smooth map $\mathcal{C} : X_0 \to X$ such that the manifold $\mathcal{W}_{\loc}^c := \mathcal{C}(\{ (t,\varphi) \in X_0 \ : \ \mathcal{C}(t,\varphi) \in B_\delta(X) \})$ is $T$-periodic, $C^k$-smooth, $(n_0 + 1)$-dimensional and locally positively invariant for the time-dependent semiflow $S$ generated by \eqref{eq:T-DDEphi}.
\end{corollary}

Recall that \eqref{eq:T-DDEphi} was just a time-dependent translation of \eqref{eq:DDEphi} via the given periodic solution. Hence, if $x$ is a solution of \eqref{eq:DDEphi} then $y = x + \gamma$ is a solution of \eqref{eq:T-DDEphi} and so
\begin{equation}
    \mathcal{W}_{\loc}^{c}(\Gamma) := \{ \gamma_t + \mathcal{C}(t, \varphi) \in X \ : \ (t,\varphi) \in X_0 \mbox{ and } \mathcal{C}(t,\varphi) \in B_\delta(X) \}
\end{equation}
is a $T$-periodic $C^k$-smooth $(n_0+1)$-dimensional manifold in $X$ defined in the vicinity of $\Gamma$ for a sufficiently small $\delta > 0$. To see this, recall that $t \mapsto \gamma_t $ is $T$-periodic and $C^k$-smooth together with the fact that $\mathcal{C}$ is $T$-periodic in the first component and $C^k$-smooth (\Cref{cor:CMT DDE}). Recall from \Cref{thm:CMT properties} that $\mathcal{C}(t,0) = 0$ and so $\Gamma \subset \mathcal{W}_{\loc}^{c}(\Gamma)$. We call $\mathcal{W}_{\loc}^{c}(\Gamma)$ a \emph{local center manifold around $\Gamma$} and notice that this manifold inherits all the properties of \Cref{thm:CMT properties}.

\section{Conclusion and outlook} \label{sec:conclusions}
We have proven the existence of a smooth finite-dimensional periodic center manifold near a nonhyperbolic cycle in classical delay differential equations. Due to the broad applicability of the dual perturbation framework, the results apply to a much broader class of evolution equations such as, for example, renewal equations \cite{Diekmann2008,Diekmann1995}. 

In the absence of any delays, \eqref{eq:DDEphi} reduces itself to an ODE defined on the state space $\mathbb{R}^n$. As a consequence, the full sun-star calculus construction becomes trivial and $\mathcal{W}_{\loc}^c(\Gamma)$ is a periodic smooth invariant manifold in $\mathbb{R}^n$ defined near the nonhyperbolic cycle $\Gamma$. Remarkably, no proof of the existence and smoothness of such manifold in ODEs could be found in the literature. This gap is now closed.

The next logical step is to study the dynamics on the center manifold by the means of the standard normal forms. Therefore, in an upcoming article, we will extend the results from Iooss \cite{Iooss1988,Iooss1999} on periodic normal forms for bifurcations of limit cycles in finite-dimensional ODEs towards the setting of infinite-dimensional DDEs. Using these results, we will then derive in a second upcoming article, explicit computational formulas for the critical normal form coefficients for all codimension one bifurcations of limit cycles along the lines of the periodic normalization method \cite{Kuznetsov2005,Witte2013}.

Recall from \Cref{subsec: duality structure} that we assumed $\odot$-reflexivity throughout this article, because we were only interested in the classical DDEs. The interesting question arises how the assumptions and results have to be adapted such that a non-$\odot$-reflexive variant of \Cref{thm:LCMT} still holds. We are already inspired by the work of \cite{Janssens2020} where the existence of a smooth finite-dimensional center manifold near a nonhyperbolic equilibrium has been proven in the non-$\odot$-reflexive setting by means of admissible ranges and perturbations. The existence of a center manifold near a nonhyperbolic cycle would be interesting for studying abstract DDEs, see \cite{Spek2020,Janssens2019,Janssens2020} for examples of such DDEs describing neural fields.

\section*{Acknowledgements}
The authors would like to thank Prof. Odo Diekmann (Utrecht University), Prof. Stephan van Gils (University of Twente) and Dr. Kevin E. M. Church for helpful discussions and suggestions.

\appendix
\section{Spectral decomposition} \label{appendix: spectral}
This appendix consists of two parts. In the first part, we will lift the spectral decomposition (\Cref{hyp:CMT}) from $X$ to $X^{\odot \star}$ and in the second part we show that classical DDEs fulfill the requirements of \Cref{hyp:CMT} and \Cref{hyp:X0+sunstar}.

\subsection{Lifting the spectral decomposition from $X$ to $X^{\odot \star}$} \label{appendix: lifting}
We consider the setting from the preface of \Cref{sec: Existence} and prove that the spectral decomposition on $X$ from \Cref{hyp:CMT} induces a spectral decomposition on $X^\star$, $X^\odot$ and most importantly on $X^{\odot \star}$.

\begin{proposition} \label{prop:Xstar hyp}
Under the assumption of \Cref{hyp:CMT}, the space $X^{\star}$ and the backward evolutionary system $U^\star$ have the following properties:

\begin{enumerate}
    \item $X^{\star}$ admits a direct sum decomposition
    \begin{equation} \label{eq:decompositionXstar}
        X^\star = X^{\star}_{-}(s) \oplus X^{\star}_{0}(s) \oplus X^{\star}_{+}(s), \quad \forall s \in \mathbb{R},
    \end{equation}
    where each summand is closed.
    \item There exist three continuous time-dependent projectors $P_{i}^{\star} : \mathbb{R} \to \mathcal{L}(X^{ \star})$ with $\ran(P_i^{\star}(s)) = X_i^{ \star}(s)$ for any $s \in \mathbb{R}$ and $i \in \{-,0,+\}$. 
    \item There exists a constant $N \geq 0$ such that $\sup_{s \in \mathbb{R}}(\|P_{-}^{ \star}(s)\| + \|P_{0}^{ \star}(s)\| + \|P_{+}^{ \star}(s)\|) = N < \infty$.
    \item The projections are mutually orthogonal meaning that $P_{i}^{ \star}(s)P_j^{ \star}(s) = 0$ for all $i \neq j$ and $s \in \mathbb{R}$ with $i,j \in \{-,0,+\}$.
    \item The projections commute with the backward evolutionary system: $U^{ \star}(s,t)P_i^{ \star}(t) = P_i^{ \star}(s)U^{ \star}(s,t)$ for all $i \in \{-,0,+\}$ and $s \leq t$.
    \item Define the restrictions $U_{i}^{ \star}(s,t) : X_{i}^{ \star}(t) \to X_{i}^{ \star}(s)$ for $i \in \{-,0,+\}$ and $t \geq s$. The operators $U_{0}^{ \star}(s,t)$ and $U_{+}^{ \star}(s,t)$ are invertible and also forward evolutionary systems. Specifically, for any $t,\tau,s \in \mathbb{R}$ it holds
    \begin{equation} \label{eq:U0U+star}
        U_0^{ \star}(s,t) = U_0^{ \star}(s,\tau)U_0^{ \star}(\tau,s), \quad U_+^{ \star}(s,t) = U_+^{ \star}(s,\tau)U_+^{ \star}(\tau,t).
    \end{equation}
    \item The decomposition \eqref{eq:decompositionXstar} is an exponential trichotomy on $\mathbb{R}$ with the same constants as in \Cref{hyp:CMT}.
\end{enumerate}
\end{proposition}
\begin{proof}
We prove this proposition by separately showing that each statement holds. Throughout the proof, we assume that $s \in \mathbb{R}$ is given. 

1. It follows from part 1 and 2 of \Cref{hyp:CMT} that by taking duals
\begin{equation*}
    X^\star = [\ran(P_{-}(s))]^\star \oplus [\ran(P_0(s))]^\star \oplus [\ran(P_+(s))]^\star.
\end{equation*}
If $i \in \{-,0,+\}$, then it follows from \cite[Lemma A.1]{Janssens2020} that the map $\iota_i(s) : \ran(P_{i}(s)^\star) \to [\ran(P_{i}(s))]^\star$ defined as $\iota_i(s) y^\star = y^\star |_{\ran(P_{i}(s))}$ is an isometric isomorphism and $P_{i}(s)^\star \in \mathcal{L}(X^\star)$. From this isometric isomorphism, the space $[X_{i}(s)]^\star = [\ran(P_{i}(s))]^\star$ can be identified with $\ran(P_{i}^\star(s)) =: X_{i}^{\star}(s)$ where we defined $P_{i}^\star(s) := P_{i}(s) ^\star$ for any $s \in \mathbb{R}$. Because $P_{i}^\star(s)$ has closed range, $X_{i}^\star(s)$ is closed. 

2. It only remains to show that $P_{i}$ is continuous for each $i \in \{-,0,+\}$. Consider $h \in \mathbb{R}$, then
\begin{equation*}
    \|P_{i}^\star(s+h) - P_{i}^\star(s)\| = \|[P_{i}(s+h) - P_{i}(s)]^\star\| = \|P_{i}(s+h) - P_{i}(s)\| \to 0, \mbox{ as } h \to 0,
\end{equation*}
because $P_i$ is continuous by part 2 of \Cref{hyp:CMT}.

3. Since $\|P_i^\star(s)\| = \|P_i(s)^\star\| = \|P_i(s)\|$ we have that part 3 holds with the same constant $N$ as in part 3 \Cref{hyp:CMT}.

4. Let $i \neq j$, then $P_i^\star(s)P_j^\star(s) = P_i(s)^\star P_j(s)^\star = (P_j(s) P_i(s))^\star = 0$ because $P_j(s) P_i(s) = 0$ due to part 4 of \Cref{hyp:CMT}.

5. Notice that for any $s \leq t$ we have that
\begin{equation*}
    U^\star(s,t)P_{i}^\star(t) = (P_i(t)U(t,s))^\star = (U(t,s)P_i(s))^\star = P_{i}^{\star}(s) U(t,s)^\star = P_{i}^{\star}(s) U^\star(s,t),
\end{equation*}
where we used part 5 of \Cref{hyp:CMT} in the third equality.

6. The restrictions are well-defined. Because $U_0(t,s)$ and $U_{+}(t,s)$ are invertible we also have that $U_0^\star(s,t) = U_0(t,s)^\star$ and $U_{+}^\star(s,t)^\star = U_{+}(t,s)^\star$ are invertible and so forward evolutionary systems. Let us now prove \eqref{eq:U0U+star}. Let $t,\tau,s \in \mathbb{R}$ be given, then
\begin{align*}
    U_0^\star(s,t) = U_0(t,s)^\star = (U_0(t,\tau)U_0(\tau,s))^\star = U_0(\tau,s)^\star U_0(t,\tau)^\star = U_0^\star(s,\tau) U_0^\star(\tau,s),
\end{align*}
where we used \eqref{eq:U0U+} in the second equality. The proof for $U_+^\star$ is analogous.

7. Let $i = -$ and suppose that $t \geq s$. Let $x^\star \in X_{i}^\star(s) = \ran(P_{i}^\star(s))$ be given. Since $\iota_i(t)$ is an isometry for any $t \in \mathbb{R}$, 
\begin{equation*}
    \|U^\star(s,t)x^\star\| = \|\iota_i(t)[U^\star(s,t)x^\star]\| =  \sup_{ \substack{x \in X_{i}(s) \\ \|x\| \leq 1}} |\langle U_i(t,s)x, x^\star \rangle |  \leq \|U_i(t,s)\| \ \|x^\star\|.
\end{equation*}
Taking the supremum over all $x^\star$ that satisfies $\|x^\star\| \leq 1$ we obtain $\|U_i^\star(s,t)\| \leq \|U_i(t,s)\|$ and this last part can be bounded by one of the three estimates in part 7 of \Cref{hyp:CMT}. The cases for $i \in \{0,+ \}$ are analogous. This completes the proof.
\end{proof}

\begin{proposition} \label{prop:Xsun hyp}
Under the assumption of \Cref{hyp:CMT}, the space $X^{\odot}$ and the backward evolutionary system $U^\odot$ have the following properties:

\begin{enumerate}
    \item $X^{\odot}$ admits a direct sum decomposition
    \begin{equation} \label{eq:decompositionXsun}
        X^\odot = X^{\odot}_{-}(s) \oplus X^{\odot}_{0}(s) \oplus X^{\odot}_{+}(s), \quad \forall s \in \mathbb{R},
    \end{equation}
    where each summand is closed.
    \item There exist three continuous time-dependent projectors $P_{i}^{\odot} : \mathbb{R} \to \mathcal{L}(X^{ \odot})$ with $\ran(P_i^{\odot}(s)) = X_i^{ \odot}(s)$ for any $s \in \mathbb{R}$ and $i \in \{-,0,+\}$. 
    \item There exists a constant $N \geq 0$ such that $\sup_{s \in \mathbb{R}}(\|P_{-}^{ \odot}(s)\| + \|P_{0}^{ \odot}(s)\| + \|P_{+}^{ \odot}(s)\|) = N < \infty$.
    \item The projections are mutually orthogonal meaning that $P_{i}^{ \odot}(s)P_j^{ \odot}(s) = 0$ for all $i \neq j$ and $s \in \mathbb{R}$ with $i,j \in \{-,0,+\}$.
    \item The projections commute with the backward evolutionary system: $U^{ \odot}(s,t)P_i^{ \odot}(t) = P_i^{ \odot}(t)U^{ \odot}(s,t)$ for all $i \in \{-,0,+\}$ and $s \leq t$.
    \item Define the restrictions $U_{i}^{ \odot}(s,t) : X_{i}^{ \odot}(t) \to X_{i}^{ \odot}(s)$ for $i \in \{-,0,+\}$ and $t \geq s$. The operators $U_{0}^{ \odot}(s,t)$ and $U_{+}^{ \odot}(s,t)$ are invertible and also forward evolutionary systems. Specifically, for any $t,\tau,s \in \mathbb{R}$ it holds
    \begin{equation} \label{eq:U0U+sun}
        U_0^{ \odot}(s,t) = U_0^{ \odot}(s,\tau)U_0^{ \odot}(\tau,s), \quad U_+^{ \odot}(s,t) = U_+^{ \odot}(s,\tau)U_+^{ \odot}(\tau,t).
    \end{equation}
    \item The decomposition \eqref{eq:decompositionXsun} is an exponential trichotomy on $\mathbb{R}$ with the same constants as in \Cref{hyp:CMT}.
\end{enumerate}
\end{proposition}
\begin{proof}
Let $s \in \mathbb{R}$ and $i \in \{-,0,+\}$ be given. Notice directly that the Lipschitz continuity of $B$ implies that $U^\odot(s,t)$ is well-defined and $X^\odot$-invariant. We define for any $s$ the map $P_i^\odot(s) := P_i^\star(s) |_{X^\odot}$ and notice that part 6 of \Cref{prop:Xstar hyp} implies that $P_i^\star(s)$ maps $X^\odot$ into itself. We denote the range of $P_i^\odot(s)$ by $X_i^\odot(s)$ and it is clear that
\begin{equation} \label{eq:Xisun}
    X_i^\odot(s) = X_i^\star(s) \cap X^\odot.
\end{equation}
Let us now prove the seven assertions.

1. Notice that $X_i^\odot(s)$ is closed because $X_i^\star(s)$ is closed (part 1 of \Cref{prop:Xstar hyp}) and $X^\odot$ is closed. The result follows from \eqref{eq:Xisun}.

2. As $X^\odot$ is a subspace of $X^\star$, we have for any $h \in \mathbb{R}$ that
\begin{equation*}
    \|P_{i}^\odot(s+h) - P_{i}^\odot(s)\| = \|[P_{i}(s+h) - P_{i}(s)]^\odot\| \leq \|[P_{i}(s+h) - P_{i}(s)]^\star\| \to 0, \mbox{ as } h \to 0,
\end{equation*}
due to part 2 of \Cref{prop:Xstar hyp}. Hence, $P_i^\odot$ is continuous.

3. This follows from part 3 of \Cref{prop:Xstar hyp} because $\|P_{i}^\odot(s)\| \leq \|P_{i}^\star(s)\|$ due to the restriction.

4. This follows from part 4 of \Cref{prop:Xstar hyp} due to the restriction.

5. This claim follows from part 4 of \Cref{prop:Xstar hyp} and recalling the fact that $U^\odot(s,t)$ is $X^\odot$-invariant.

6. For the well-definedness of the restriction, we have to check that $U_i^\odot(s,t)$ takes values in $X_i^\odot(s)$. Since $U_i^\odot(s,t) = U_i^\star(s,t) |_{X^\odot}$ we get from part 6 of \Cref{prop:Xstar hyp} that $U_i^\odot(s,t)$ maps into $X_i^\star(s)$. Because $U^\odot(s,t)$ is $X^\odot$-invariant we also have that the restriction $U_i^\odot(s,t)$ is $X^\odot$-invariant and so $U_i^\odot(s,t)$ takes values in $X^\odot$. To conclude, $U_i^\odot(s,t)$ takes values in $X_i^\star(s) \cap X^\odot = X_i^\odot(s)$ by \eqref{eq:Xisun}. The remaining claims follow immediately because of the restriction. 

7. Because of the restriction we have that $\|U_{i}^\odot(s,t)\| = \|U_i(t,s)^\odot\| \leq \|U_i(t,s)^\star\| = \|U_i^\star(s,t)\|$ and the right-hand side can now be estimated by the upper bounds given in part 7 of \Cref{prop:Xstar hyp}.
\end{proof}

\begin{proposition} \label{prop:Xsunstar hyp}
Under the assumption of \Cref{hyp:CMT}, the space $X^{\odot \star}$ and the forward evolutionary system $U^{\odot \star}$ have the following properties:

\begin{enumerate}
    \item $X^{\odot \star}$ admits a direct sum decomposition
    \begin{equation} \label{eq:decomposition Xsunstar hyp}
        X^{\odot \star} = X^{\odot \star}_{-}(s) \oplus X^{\odot \star}_{0}(s) \oplus X^{\odot \star}_{+}(s), \quad \forall s \in \mathbb{R},
    \end{equation}
    where each summand is closed.
    \item There exist three continuous time-dependent projectors $P_{i}^{\odot \star} : \mathbb{R} \to \mathcal{L}(X^{\odot \star})$ with $\ran(P_i^{\odot \star}(s)) = X_i^{\odot \star}(s)$ for any $s \in \mathbb{R}$ and $i \in \{-,0,+\}$. 
    \item There exists a constant $N \geq 0$ such that $\sup_{s \in \mathbb{R}}(\|P_{-}^{\odot \star}(s)\| + \|P_{0}^{\odot \star}(s)\| + \|P_{+}^{\odot \star}(s)\|) = N < \infty$.
    \item The projections are mutually orthogonal meaning that $P_{i}^{\odot \star}(s)P_j^{\odot \star}(s) = 0$ for all $i \neq j$ and $s \in \mathbb{R}$ with $i,j \in \{-,0,+\}$.
    \item The projections commute with the forward evolutionary system: $U^{\odot \star}(t,s)P_i^{\odot \star}(s) = P_i^{\odot \star}(t)U^{\odot \star}(t,s)$ for all $i \in \{-,0,+\}$ and $t \geq s$.
    \item Define the restrictions $U_{i}^{\odot \star}(t,s) : X_{i}^{\odot \star}(s) \to X_{i}^{\odot \star}(t)$ for $i \in \{-,0,+\}$ and $t \geq s$. The operators $U_{0}^{\odot \star}(t,s)$ and $U_{+}^{\odot \star}(t,s)$ are invertible and also backward evolutionary systems. Specifically, for any $t,\tau,s \in \mathbb{R}$ it holds
    \begin{equation*}
        U_0^{\odot \star}(t,s) = U_0^{\odot \star}(t,\tau)U_0^{\odot \star}(\tau,s), \quad U_+^{\odot \star}(t,s) = U_+^{\odot \star}(t,\tau)U_+^{\odot \star}(\tau,s).
    \end{equation*}
    \item The decomposition \eqref{eq:decomposition Xsunstar hyp} is an exponential trichotomy on $\mathbb{R}$ with the same constants as in \Cref{hyp:CMT}.
\end{enumerate}
\end{proposition}
\begin{proof}
Recall that $X^{\odot}$ is a Banach space and $U^\odot$ a backward evolutionary system on $X^\odot$. Therefore, we can apply \Cref{prop:Xstar hyp} with $X$ replaced by $X^\odot$ and $U$ replaced by $U^\star$ by going over from a forward towards a backward evolutionary system. Hence, we obtain the desired result.
\end{proof}

\subsection{Verification of \Cref{hyp:CMT} and \Cref{hyp:X0+sunstar} for classical DDEs} \label{appendix: Spectral DDE}
In order to verify both hypotheses, we have to construct three time-dependent projectors $P_{i}$ with $i \in \{-,0,+\}$. Before we do this, let us first define the time-dependent \emph{spectral projection} (at time $s$) as $P_\lambda(s) \in \mathcal{L}(X)$ with range $E_\lambda(s)$ and kernel $R_\lambda(s)$ that can be represented via the holomorphic functional calculus as the Dunford integral
\begin{equation*}
    P_\lambda(s) := \frac{1}{2 \pi i} \oint_{\partial C_\lambda} (zI-U(s+T,s))^{-1}dz,
\end{equation*}
where $C_\lambda \subset \mathbb{C}$ is a sufficiently small open disk centered at $\lambda$ with $\partial C_\lambda$ its boundary such that $\lambda$ is the only Floquet multiplier inside $C_\lambda$. Recall from the compactness property of $U(s+T,s)$ that the Floquet multipliers are isolated and hence making such a contour $\partial C_\lambda$ in the complex plane is possible.
\begin{proposition} \label{prop:Plambda continuous}
The map $P_\lambda : \mathbb{R} \to \mathcal{L}(X)$ is continuous and $T$-periodic.
\end{proposition}
\begin{proof}
Let an initial starting time $s \in \mathbb{R}$ be given with arbitrary $h \in \mathbb{R}$. Let $C_\lambda$ be an open disk in $\mathbb{C}$ centered at $\lambda$ such that $\partial C_\lambda$ is a circle with sufficiently small radius $r > 0$ such that $\lambda$ is the only Floquet multiplier in $C_\lambda$. Hence,
\begin{align*}
    \|P_\lambda(s+h) - P_\lambda(s) \| &= \frac{1}{2 \pi} \bigg | \bigg | \oint_{\partial C_\lambda} (zI-U(s+T + h,s + h))^{-1} - (zI-U(s+T,s))^{-1} dz\bigg | \bigg |,
\end{align*}
because the Floquet multipliers are independent of the starting time. Notice that the integrand is just a difference of resolvents and due to the second resolvent identity \cite[Theorem 4.8.2]{Hille1957} we notice that the integrand equals
\begin{equation*}
   R(z,h) [U(s+T+h,s+h) - U(s+T,s)] R(z,0), \quad \forall z \in \partial C_\lambda,
\end{equation*}
where for any $h \in \mathbb{R}$ the resolvent map $R(\cdot,h) : \partial C_\lambda \to \mathcal{L}(X)$ is defined as $R(z,h) = (zI-U(s+T + h,s + h))^{-1}$. Notice that $R(\cdot,h)$ indeed takes values in $\mathcal{L}(X)$ due to the bounded inverse theorem. Filling this back into the expression above yields
\begin{equation*}
    \|P_\lambda(s+h) - P_\lambda(s) \| \leq \frac{1}{2 \pi} \|U(s+T+h,s+h) - U(s+T,s) \|  \oint_{\partial C_\lambda} \|R(z,h)\| \ \|R(z,0)\| dz.
\end{equation*}
We claim that for any fixed $h \in \mathbb{R}$ the map $\partial C_\lambda \ni z \mapsto \|R(z,h)\| \in \mathbb{R}$ is continuous. Indeed, fix a $h \in \mathbb{R}$ and choose $u \in C_\lambda$ such that $|z-u| \to 0$, where $|\cdot|$ represents the arc length on the circle $C_\lambda$. The reverse triangle inequality and the first resolvent identity \cite[Theorem 4.8.1]{Hille1957} implies
\begin{equation*}
    | \ \|R(u,h)\| - \|R(z,h)\| \ | \leq |z-u| \ \|R(u,h)\| \ \|R(z,h)\| \to 0, \quad \mbox{ as } |z-u| \to 0.
\end{equation*}
Since $C_\lambda$ is compact, we have that the image $\{\|R(z,h)\| \ : \ z \in C_\lambda \}$ is a compact subset of $\mathbb{R}$ and hence this set is bounded, say it is contained in the interval $[-M_h,M_h]$ for some constant $M_h > 0$, for a fixed $h \in \mathbb{R}$. We obtain
\begin{align*}
    \|P_\lambda(s+h) - P_\lambda(s) \| &\leq rM_0 M_h \|U(s+T+h,s+h) - U(s+T,s) \| \to 0, \quad \mbox{ as } h \to 0,
\end{align*}
by \cite[Lemma 5.2]{Clement1988} since $(s+T,s) \in \Omega_\mathbb{R}$. The $T$-periodicity holds due to \cite[Corollary XIII.2.2]{Diekmann1995} and the fact that the Floquet multipliers are independent of the starting time \cite[Theorem XIII.3.3]{Diekmann1995}.
\end{proof}

We also need the associated spectral projections on the unstable, center and stable eigenspace. For the unstable and center eigenspace, denote the \emph{spectral projection on the unstable eigenspace} (at time $s$) and the \emph{spectral projection on the center eigenspace} (at time $s$) as the operators $P_+(s) \in \mathcal{L}(X)$ with range $X_{+}(s)$ and $P_0(s) \in \mathcal{L}(X)$ with range $X_0(s)$ defined as
\begin{equation*}
    P_{+}(s) := \sum_{\lambda \in \Lambda_+} P_\lambda(s), \quad P_{0}(s) := \sum_{\lambda \in \Lambda_0} P_{\lambda}(s).
\end{equation*}
Define the \emph{spectral projection on the stable eigenspace} (at time $s$) as $P_{-}(s) := I - P_0(s) - P_+(s) \in \mathcal{L}(X)$ and it holds that $P_{-}(s)$ is indeed the projection on the stable eigenspace $X_{-}(s)$, see \cite[Lemma 7.2.2]{Church2018}. The proof of the following result is almost the same as \cite[Theorem 7.2.1]{Church2018}, but we give it for the sake of completeness.

\begin{proposition} \label{prop:DDE hypothesis}
The setting of \eqref{eq:DDEphi} satisfies \Cref{hyp:CMT}.
\end{proposition}
\begin{proof}
We verify the seven criteria step by step.

1. The decomposition \eqref{eq:decomposition X hyp} can be also used in the case where $E_\lambda(s)$ is replaced with the finite-dimensional vector space $X_{+}(s) \oplus X_{0}(s)$. Then, $X = X_{+}(s) \oplus X_{0}(s) \oplus R(s)$ for some vector space $R(s)$. We have to show $R(s) = X_{-}(s)$. By the decomposition $\eqref{eq:decomposition X hyp}$ we know that $P_{+0}(s):= P_+(s) + P_0(s) $ is a projection with range $X_{+0}(s) = X_{-}(s) \oplus X_{0}(s)$ and $R(s) = \ker P_{+0}(s)$ and notice that $R(s) = \cap_{\lambda \in \Lambda_{0+}} \ker(P_{\lambda}(s)) = X_{-}(s)$. The spaces $X_{+}(s)$ and $X_{0}(s)$ are automatically closed since they are finite-dimensional. To show that $X_{-}(s)$ is closed, notice that for each $\lambda \in \Lambda_{0+}$ the space $R_{\lambda}(s)$ is closed and because the finite intersection of closed sets is closed, the result follows from \eqref{eq:X-}. 

2. For $P_{+}$ and $P_{0}$ the claim about the range follows immediately from their definition and the claim about $P_{-}$ follows from the fact that $P_{-}(s)$ is the projection on $X_{-}(s)$. To show the continuity statement, recall from \Cref{prop:Plambda continuous} that for any Floquet multiplier $\lambda$, the map $P_\lambda$ is continuous. As $P_{+}$ and $P_{0}$ are finite sums of such continuous projectors, it follows that both projectors are continuous. Since $P_{-} = I - P_0 - P_{+}$ it follows that $P_{-}$ is also continuous. 

3. Since $P_- + P_0 + P_+ = I$, we have $\|P_{-}(t)\| \leq \|P_{0}(t)\| + \|P_{+}(t)\|$ for all $t \in \mathbb{R}$, and so it remains to prove that $t\mapsto \|P_{0}(t)\|$ and $t \mapsto \|P_{+}(t)\|$ are uniformly bounded on $[0,T]$ by $T$-periodicity. We will only show the claim for $P_{0}$ since the proof is similar for $P_{+}$. 

Suppose for a moment that part 5 and 7 are satisfied. They will be proven later, independently of this property. Assume that $t \mapsto \|P_{0}(t)\|$ is not uniformly bounded on $[0,T]$, then there exist sequences $(x_{n})_{n \in \mathbb{N}} \subset X$ and $(t_{n})_{n \in \mathbb{N}} \subset [0,T]$ such that $\|x_n\|_{\infty} = 1$ and $\|P_0(t_n)x_n\|_\infty = n$. Then for a given $\varepsilon > 0$, there is a constant $K_\varepsilon > 0$ such that 
\begin{align*}
    n &= \|P_0(t_n)x_n\|_\infty \leq \|U_{0}(t_n,T)\| \ \|P_0(T)\| \ \|U_0(T,t_n)\| \leq K_\varepsilon^2 e^{2 \varepsilon T} \|P_0(T)\|, 
\end{align*}
which is a contradiction, since $n \in \mathbb{N}$ can be taken arbitrary large.

4. Let $i,j \in \{-,0,+\}$ with $i \neq j$ and let $\varphi\in X$. By the decomposition proved in criterion one we have that $\varphi= \varphi_{i}(s) + \varphi_{j}(s) + \varphi_{k}(s)$, where $k \in \{-,0,+ \}$ such that $k \neq i$ and $k \neq j$. Then from the interplay between the ranges and kernels of the projections it follows that
\begin{equation*}
    P_{i}(s)P_{j}(s) \varphi= P_{i}(s)P_{j}(s)[\varphi_{i}(s) + \varphi_{j}(s) + \varphi_{k}(s)] = P_{i}(s) \varphi_{j}(s) = 0,
\end{equation*}
which proves this part. 

5. It is proven in \cite[Theorem XIII.3.3]{Diekmann1995} that
\begin{equation} \label{eq:P and U}
    P(t)U(t+T,s+jT) = U(t+T,s+jT)P(s)
\end{equation}
for $j \in \mathbb{N}$ chosen in such a way that $s+(j-1)T \leq t < s+jT$ and for $P \in \{P_{-},P_0,P_{+} \}$. Hence,
\begin{align*}
    P(t)U(t,s) &= P(t)U(t,s+jT)U(s+jT,s) = U(t,s+jT)U(s+T,s)^j P^j(s) = U(t,s)P(s),
\end{align*}
where we have used that $P(s)$ is a projection that commutes with $U(s+T,s)$. This last claim follows from setting $s=t$ and $j=1$ in \eqref{eq:P and U} together with \cite[Corollary XIII.2.2]{Diekmann1995}.

6. Notice that $U_{+}(t,s)$ and $U_0(t,s)$ are defined for all $t,s \in \mathbb{R}$ because they are restricted to a finite-dimensional space. Since $U_{+}(t,s) U_{+}(s,t) = I = U_{+}(s,t) U_{+}(t,s)$ we have that $U_{+}(t,s)$ is invertible with inverse $U_{+}(t,s)^{-1} = U_{+}(s,t)$. Similarly $U_{0}(t,s)^{-1} = U_{0}(s,t)$. To show the remaining part, that is \eqref{eq:U0U+}, we have six different cases depending on the location of $t,\tau,s \in \mathbb{R}$. This is a straightforward computation and will be omitted. 

7. We will start with the center part. The stable and unstable part will then follow from a similar reasoning. Let $\varepsilon > 0$ and $s \in \mathbb{R}$ be given. As the map $t \mapsto U_0(t,s) \varphi$ is continuous for any $\varphi \in X$ and $t \geq s$, we know
\begin{equation*}
    \sup_{s \leq t \leq s+T} \|U_0(t,s)\varphi\|_\infty < \infty, \quad \forall \varphi \in X.
\end{equation*}
By the principle of uniform boundedness, we get
\begin{equation*}
    \sup_{s \leq t \leq s+T} \|U_0(t,s)\| \leq K,
\end{equation*}
for some $K > 0$. Because the spectrum of $U_0(s+T,s)$ lies on the unit circle, we have by the spectral radius formula also known as the Gelfand-Beurling formula that
\begin{equation*}
    1 = \max_{\lambda \in \sigma(U_0(s+T,s))} |\lambda| = \lim_{j \to \infty} \|U_0(s+T,s)^j\|^{\frac{1}{j}}
\end{equation*}
and so there exists an integer $k_\varepsilon > 0$ such that $\|U_0(s+T,s)^{k_\varepsilon}\| < 1+ \varepsilon T$ and denote
\begin{equation*}
    K_\varepsilon := K \max_{j=0,\dots,k_\varepsilon-1} \|U_0(s+T,s)^j\|.
\end{equation*}
Now, let $m_t$ be the largest integer such that $s+m_t k_\varepsilon T \leq t$ and $m_t^{\star} \in \{0,\dots,k_\varepsilon - 1 \}$ the largest integer such that $s+m_t k_\varepsilon T + m_t^\star \leq t$. Then,
\begin{align*}
    U_0(t,s) &= U_0(t,s+m_tk_\varepsilon T + m_t^\star T)U_0(s+m_tk_\varepsilon T + m_t^\star T,s+m_tk_\varepsilon T)U_0(s+m_t k_\varepsilon T,s)\\
    &= U_0(t- m_tk_\varepsilon T - m_t^\star T,s)U_0(s+m_t^\star T,s)U_0(s+m_t k_\varepsilon T,s)\\
    &= U_0(t- m_tk_\varepsilon T - m_t^\star T,s)  U_0(s+ T,s)^{m_t^\star}U_0(s+T,s)^{m_t k_\varepsilon}.
\end{align*}
We can make the estimate
\begin{equation*}
    \|U_0(t,s)\| \leq K_\varepsilon \|U_0(s+T,s)^{k_\varepsilon}\|^{m_t} \leq K_\varepsilon (1+\varepsilon T)^{\frac{t - s}{T}} = K_\varepsilon [(1 + \varepsilon T)^{\frac{1}{\varepsilon T}}]^{\varepsilon(t-s)} \leq K_{\varepsilon}e^{\varepsilon(t-s)},
\end{equation*}
since the function $(0,\infty) \ni x \mapsto (1+\frac{1}{x})^x \in \mathbb{R}$ is monotonically increasing. The proof is analogous when $t \leq s$ and so we obtain $\|U_0(t,s)\| \leq K_{\varepsilon}e^{\varepsilon|t-s|}$. The proofs for the stable and unstable part are analogous.
\end{proof}

Denote for any Floquet multiplier $\lambda$ and any $s \in \mathbb{R}$ the time-dependent \emph{extended spectral projection} $P_\lambda^{\odot \star}(s) \in \mathcal{L}(X^{\odot \star})$ with range $jE_\lambda(s)$ and kernel $R_\lambda^{\odot \star}(s)$, where $R_\lambda ^{\odot \star} (s)$ is the called the \emph{extended complementary (generalized) eigenspace} (at time $s$) coming from the decomposition $X^{\odot \star} = jE_\lambda(s) \oplus R_\lambda^{\odot \star}(s)$. Define the \emph{extended unstable eigenspace} (at time $s$) and \emph{extended center eigenspace} (at time $s$) as

\begin{equation*}
    X_{+}^{\odot \star}(s) := j(X_+(s)) = \bigoplus_{\lambda \in \Lambda_+} jE_{\lambda}(s), \quad X_0^{\odot \star}(s) := j(X_0(s)) = \bigoplus_{\lambda \in \Lambda_0} jE_{\lambda}(s),
\end{equation*}
and notice via extended complementary (generalized) eigenspaces that the \emph{extended stable eigenspace} (at time $s$) can be defined as
\begin{equation*}
    X_{-}^{\odot \star}(s) := \bigcap_{\lambda \in \Lambda_0 \cup \Lambda_+} R_{\lambda}^{\odot \star}(s).
\end{equation*}
The construction of $X_{+}^{\odot \star}(s)$ and $X_{0}^{\odot \star}(s)$ directly shows that \Cref{hyp:X0+sunstar} is satisfied.

\section{Smoothness and periodicity of the center manifold} \label{appendix: smoothness and periodicty}
This section of the appendix consists of three parts. Firstly, we show that the map $\mathcal{C}$ is not only fiberwise Lipschitz, but Lipschitz continuous in the second component where the Lipschitz constant is independent of the fiber. The proof of this claim is inspired by \cite[Corollary 5.4.1.1]{Church2018}. Secondly, we prove via the theory of contractions on scales of Banach spaces, see \cite[Section IX.6, Appendix IV]{Diekmann1995} and \cite{Vanderbauwhede1987} that the map $\mathcal{C}$ is $C^k$-smooth. To do this, we combine the ideas from \cite[Section IX.7]{Diekmann1995}, \cite[Section 8]{Church2018} and \cite{Hupkes2008}. Lastly, under the assumption of $T$-periodicity of the time-dependent nonlinear perturbation $R$ in the first component, we show that there exists a neighborhood of $0$ in $X$ such that the center manifold is $T$-periodic in this neighborhood. The proof of this result is inspired by \cite[Lemma 8.3.1 and Theorem 8.3.1]{Church2018}.

\begin{corollary} \label{cor:lipschitzCMT}
There exists a constant $L > 0$ such that $\|\mathcal{C}(t,\varphi) - \mathcal{C}(t,\psi)\| \leq L \|\varphi - \psi\|$ for all $t \in \mathbb{R}$ and $\varphi,\psi \in X_0(t)$. 
\end{corollary}
\begin{proof}
Let $t \in \mathbb{R}$ and $\varphi,\psi \in X_0(t)$ be given. Notice that
\begin{equation*}
    \mathcal{C}(t,\varphi) = u_t^\star(\varphi)(t) = [\mathcal{G}_t(u_t^\star(\varphi)(t),\varphi)](t) = \varphi + \mathcal{K}_t^\eta[\tilde{R}_{\delta,t}(u_t^\star(\varphi)(t))](t).
\end{equation*}
By \Cref{prop:ketas}, we know there exists a constant $C_{\eta} > 0,$ independent of $t$ such that $\|\mathcal{K}_t^\eta\| \leq C_\eta$. Hence, from \Cref{cor:lipschitz2} and \Cref{thm:sol lipschitz} we get
\begin{align*}
    \|\mathcal{C}(t,\varphi)  - \mathcal{C}(t,\psi)\| &\leq  \|\varphi - \psi \|  + \| \mathcal{K}_t^\eta [ \tilde{R}_{\delta,t}(u_t^\star(\varphi)(t)) - \tilde{R}_{\delta,t}(u_t^\star(\psi)(t))](t)\| \\
    &\leq \|\varphi - \psi \|  + \| \mathcal{K}_t^\eta \| \sup_{s \in \mathbb{R}} \|[\tilde{R}_{\delta,t}(u_t^\star(\varphi)(t)) - \tilde{R}_{\delta,t}(u_t^\star(\psi)(t))](s)\|e^{-\eta|t-s|} \\
    &\leq \|\varphi - \psi \|  + C_\eta \|\tilde{R}_{\delta,t}(u_t^\star(\varphi)(t)) - \tilde{R}_{\delta,t}(u_t^\star(\psi)(t))\|_{\eta,t} \\
    &= (1 + 2C_\eta L_{R_{\delta}} K_\varepsilon)\|\varphi - \psi \|.
\end{align*}
Hence $L = 1 + 2C_\eta L_{R_{\delta}}K_\varepsilon > 0$ is the Lipschitz constant we were looking for.
\end{proof}

The following lemma will be important to prove smoothness of $\mathcal{C}$ and $\mathcal{W}^c$.

\begin{lemma}[{\cite[Lemma XII.6.6 and XII.6.7]{Diekmann1995}}] \label{lemma:fixedpointsmooth}
Let $Y_0,Y,Y_1$ and $\Lambda$ be Banach spaces with continuous embeddings $J_0 : Y_0 \hookrightarrow Y$ and $J : Y \hookrightarrow Y_1$. Consider the fixed point problem $y = f(y,\lambda)$ for $f : Y \times \Lambda \to Y$. Suppose that the following conditions hold.
\begin{enumerate}
    \item The function $g : Y_0 \times \Lambda \to Y_1$ defined as $ g(y_0,\lambda) := Jf(J_0y_0,\lambda)$ is of the class $C^1$ and there exist mappings
    \begin{align*}
        f^{(1)} &: J_0 Y_0 \times \Lambda \to \mathcal{L}(Y), \\
        f_1^{(1)} &: J_0Y_0 \times \Lambda \to \mathcal{L}(Y_1),
    \end{align*}
    such that
    \begin{equation*}
        D_1 g(y_0,\lambda) \xi = Jf^{(1)}(J_0y_0,\lambda)J_0, \quad \forall (y_0,\lambda,\xi) \in Y_0 \times \Lambda \times Y_0
    \end{equation*}
    and 
    \begin{equation*}
      Jf^{(1)}(J_0y_0,\lambda)y = f_1^{(1)}(J_0y_0,\lambda)Jy, \quad \forall (y_0,\lambda,y) \in Y_0 \times \Lambda \times Y.
    \end{equation*}
    \item There exists a $\kappa \in [0,1)$ such that for all $\lambda \in \Lambda$ the map $f(\cdot,\lambda) : Y \to Y$ is Lipschitz continuous with Lipschitz constant $\kappa$, independent of $\lambda$. Furthermore, for any $\lambda \in \Lambda$ the maps $f^{(1)}(\cdot,\lambda)$ and $f_1^{(1)}(\cdot,\lambda)$ are uniformly bounded by $\kappa$.
    \item Under the previous condition, the unique fixed point $\Psi : \Lambda \to Y$ satisfies $\Psi(\lambda) = f(\Psi(\lambda),\lambda)$ and can be written as $\Psi = J_0 \circ \Psi$ for some continuous $\Psi : \Lambda \to Y_0$.
    \item The function $f_0 : Y_0 \times \Lambda \to Y$ defined by $f_0(y_0,\lambda) = f(J_0y_0,\lambda)$ has continuous partial derivative
    \begin{equation*}
        D_2 f : Y_0 \times \Lambda \to \mathcal{L}(\Lambda,Y).
    \end{equation*}
    \item The mapping $Y_0 \times \Lambda \ni (y,\lambda) \mapsto J \circ f^{(1)}(J_0y,\lambda) \in \mathcal{L}(Y,Y_1)$ is continuous.
\end{enumerate}
Then the map $J \circ \Psi$ is of the class $C^1$ and $D(J \circ \Psi)(\lambda) = J \circ \mathcal{A}(\lambda)$ for all $\lambda \in \Lambda$, where $A = \mathcal{A}(\lambda) \in \mathcal{L}(\Lambda,Y)$ is the unique solution of the fixed point equation
\begin{equation*}
    A = f^{(1)}(\Psi(\lambda),\lambda) A + D_2 f_0(\Psi(\lambda),\lambda),
\end{equation*}
formulated in $\mathcal{L}(\Lambda,Y)$.
\end{lemma}

An important observation between the dependence of $u_s^\star$ on $\delta$ is presented in the following lemma. To make the notation a bit simpler, we define the map $\hat{P}_0 : \BC_s^{\eta}(\mathbb{R},X) \to \BC_s^{\eta}(\mathbb{R},X)$ pointwise as $(\hat{P}_0\varphi)(t) := (P_0(t)\varphi)(t) \in X_0(t)$ for all $t \in \mathbb{R}$ and have the following lemma.
\begin{lemma}\label{lemma:P0hat}
If $\delta > 0$ is sufficiently small, then $\|(I-\hat{P}_0)u_s^\star(\varphi)\|_{0,s} < N\delta$.
\end{lemma}
\begin{proof}
Since $u_s^\star(\varphi) = \mathcal{G}_s(u_s^\star(\varphi),\varphi) = U(\cdot,s)\varphi + \mathcal{K}_s^\eta(\tilde{R}_{\delta,s}(u_s^\star(\varphi)))$ we have that
\begin{equation*}
    (I-\hat{P}_0)u_s^\star(\varphi) = (I-\hat{P}_0)[\mathcal{K}_s^\eta(\tilde{R}_{\delta,s}(u_s^\star(\varphi)))],
\end{equation*}
because for any $t \in \mathbb{R}$ we have that
\begin{equation*}
    [(I-\hat{P}_0)U(\cdot,s)\varphi](t) = U(t,s)\varphi - P_0(t)U(t,s)\varphi = 0,
\end{equation*}
since $U(t,s)\varphi = U_0(t,s)\varphi \in X_0(t)$ due to part 6 of \Cref{hyp:CMT} and the fact hat $\varphi \in X_0(s)$. It follows from the operator norm bounds in \Cref{prop:ketas}, and the bound for $\tilde{R}_{\delta,s}$ in \Cref{cor:lipschitz2} that
\begin{equation*}
    \|(I-\hat{P}_0)u_s^\star(\varphi)\|_{0,s} = \|(I-\hat{P}_0)[\mathcal{K}_s^\eta(\tilde{R}_{\delta,s}(u_s^\star(\varphi)))]\|_{0,s} \leq  4 \delta \|j^{-1}\|K_\varepsilon N  L_{R_\delta} \bigg( \frac{1}{-a} + \frac{1}{b} \bigg),
\end{equation*}
which is less than or equal to $N\delta$ if we choose 
\begin{equation*}
    L_{R_\delta} \leq \frac{1}{4 \|j^{-1}\| K_\varepsilon } \bigg(\frac{1}{-a} + \frac{1}{b} \bigg)^{-1},
\end{equation*}
which is possible since $L_{R_\delta} \to 0$ as $\delta \downarrow 0$.
\end{proof}

Let us introduce some notation. For a Banach space $X$, define the sets $\BC_s^\infty(\mathbb{R},X) := \cup_{\eta > 0} \BC_s^\eta(\mathbb{R},X)$ and $\BC_s^\infty(\mathbb{R},X^{\odot \star}) := \cup_{\eta > 0} \BC_s^\eta(\mathbb{R},X^{\odot \star})$ together with the space
\begin{equation*}
    V_s^\eta(\mathbb{R},X) := \{u \in \BC_s^\eta(\mathbb{R},X) \ : \ \|(I-\hat{P_0})u\|_{0,s} < \infty \},
\end{equation*}
with the norm
\begin{equation*}
    \|u\|_{V_s^{\eta}} := \|\hat{P}_0 u \|_{\eta,s} + \|(I-\hat{P}_0)u\|_{0,s}
\end{equation*}
such that $V_s^\eta(\mathbb{R},X)$ becomes a Banach space and is continuously embedded in $\BC_s^\eta(\mathbb{R},X)$. Define in addition for a sufficiently small $\delta > 0$ the open set
\begin{equation*}
    V_{\delta,s}^\eta(\mathbb{R},X) := \{ u \in V_s^\eta(\mathbb{R},X) \ : \ \|(I-\hat{P_0})u\|_{0,s} < N\delta \},
\end{equation*}
and notice that this set is non-empty due to \Cref{lemma:P0hat}. Define similarly as before the set $V_{\delta,s}^\infty(\mathbb{R},X) := \cup_{\eta > 0}  V_{\delta,s}^\eta(\mathbb{R},X)$. For Banach spaces $E,E_1,E_2,\dots,E_p$ with $p \geq 1$ we denote by $\mathcal{L}^p(E_1 \times \dots \times E_p, E)$ the Banach space of $E$-valued continuous $p$-linear maps defined on the $E_1 \times \dots \times E_p$. When there are $p$ identical copies in this Cartesian product, we simply write $E^p := E \times \dots \times E$, where this notation will also be used with $E$ is just simply a set.

If we chose $\delta$ as in \Cref{lemma:P0hat}, then the map $u \mapsto \tilde{R}_{\delta,s}(u)$ is of the class $C^k$, when $u \in V_{\delta,s}^\infty(\mathbb{R},X)$. Consider any pair of integers $p,q \geq 0$ with $p + q \leq k$ and notice that the norm $\|D_1^p D_2^q R_{\delta,s}(t,\varphi)\|$ is uniformly bounded on $\mathbb{R} \times V_{\delta,s}^\infty(\mathbb{R},X)$. Hence, for any $u \in V_{\delta,s}^\infty(\mathbb{R},X)$ we can define the map $R_{\delta,s}^{(p,q)}(u) : \BC_s^\infty(\mathbb{R},X)^p \to \BC_s^\infty(\mathbb{R},X^{\odot \star})$ as
\begin{equation*}
    R_{\delta,s}^{(p,q)}(u)(v_1,\dots,v_q)(t) := D_1^pD_2^q R_{\delta,s}(t,u(t))(v_1(t),\dots,v_q(t)), \quad \forall v_1,\dots,v_q \in \BC_s^\infty(\mathbb{R},X).
\end{equation*}
The following two lemmas will be crucial for the proof of \Cref{thm:smoothnesscmt}.

\begin{lemma}[{\cite[Lemma XII.7.3]{Diekmann1995} and \cite[Proposition 8.1]{Hupkes2008}}] \label{lemma: smoothness1}
Consider integers $p \geq 0$ and $q \geq 0$ with $p+q \leq k$ together with integers $\mu_1,\dots,\mu_q > 0$ such that $\mu = \mu_1 + \dots +\mu_q$ and consider $ \eta > q \mu > 0$. Then,
\begin{equation*}
    \tilde{R}_{\delta,s}^{(p,q)}(u) \in \mathcal{L}^q(\BC_s^{\mu_1}(\mathbb{R},X) \times \dots \times \BC_s^{\mu_q}(\mathbb{R},X), \BC^\eta(\mathbb{R},X^{\odot \star})), \quad \forall u \in V_{\delta,s}^\infty(\mathbb{R},X).
\end{equation*}
Furthermore, consider any $ 0 \leq l \leq k-(p+q)$ and $\sigma > 0$. If $\eta > q \mu + l \sigma$, then the map $R_{\delta,s}^{(p,q)} : V_{\delta,s}^\sigma(\mathbb{R},X) \to \mathcal{L}^q(\BC_s^{\mu_1}(\mathbb{R},X) \times \dots \times \BC_s^{\mu_p}(\mathbb{R},X), \BC^\eta(\mathbb{R},X^{\odot \star}))$ is $C^l$-smooth, with $D^l R_{\delta,s}^{(p,q)} = R_{\delta,s}^{(p,q + l)}$.
\end{lemma}

\begin{lemma}[{\cite[Lemma XII.7.6]{Diekmann1995} and \cite[Proposition 8.2]{Hupkes2008}}] \label{lemma: smoothness4}
Consider integers $p \geq 0$ and $q \geq 0$ with $p+q < k$ together with integers $\mu_1,\dots,\mu_q > 0$ such that $\mu = \mu_1 + \dots +\mu_q$. Let $ \eta > q \mu + \sigma$ for some $\sigma > 0$ and consider a $C^1$-smooth map $\Phi_s : X_0(s) \to V_{\delta,s}^\sigma(\mathbb{R},X)$. Then the map $\tilde{R}_{\delta,s}^{(p)} \circ \Phi_s : X_0(s) \to \mathcal{L}^q(\BC_s^{\mu_1}(\mathbb{R},X) \times \dots \times \BC_s^{\mu_q}(\mathbb{R},X), \BC^\eta(\mathbb{R},X^{\odot \star}))$ is $C^1$-smooth with
\begin{equation*}
    D(\tilde{R}_{\delta,s}^{(p,q)} \circ \Phi_s)(\varphi)(v_1,\dots,v_q)(t) = \tilde{R}_{\delta,s}^{(p,q+1)}(\Phi_s(\varphi))(\Phi_s'(\varphi)(t),v_1(t),\dots,v_q(t)).
\end{equation*}
\end{lemma}

So far we have only proven that the center manifold is Lipschitz continuous. Recall from \Cref{thm:sol lipschitz} that we solved the fixed point problem $u = \mathcal{G}_s(u,\varphi)$ for a given $\varphi \in X_0(s)$ in the space $\BC_s^\eta(\mathbb{R},X)$ for a given $\eta \in (0,\min\{-a,b\})$. It turns out that the space $\BC_s^\eta(\mathbb{R},X)$ is not really suited to study additional smoothness of the center manifold. The idea to obtain this is by working with another exponent, say $\tilde{\eta}$, which is chosen high enough to guarantee smoothness, but not too high to loose the contraction property. Hence, a trade-off has to be made. To do this, choose an interval $[\eta_{\minn},\eta_{\maxx}] \subset (0,\min\{-a,b\})$ such that $k \eta_{\minn} < \eta_{\maxx}$ and choose $\delta > 0$ small enough to guarantee that
\begin{equation*}
    L_{R_{\delta}}\|\mathcal{K}_s^\eta\|_{\eta,s} < \frac{1}{4}, \quad \forall \eta \in [\eta_{\minn},\eta_{\maxx}], \ s \in \mathbb{R},
\end{equation*}
which is possible since $L_{R_\delta} \to 0$ as $\delta \downarrow 0$ proven in \Cref{prop:Lrdelta}. Following the proof again of \Cref{thm:sol lipschitz} we obtain for any $\eta \in [\eta_{\minn},\eta_{\maxx}]$ a unique fixed point $u_{\eta,s}^\star : X_0(s) \to \BC_s^\eta(\mathbb{R},X)$ of the equation $u = \mathcal{G}_s(u,\varphi)$. Denote for real numbers $\eta_1 \leq \eta_2$ the continuous embedding operator as $\mathcal{J}_s^{\eta_2,\eta_1} : \BC_s^{\eta_1}(\mathbb{R},X) \hookrightarrow \BC_s^{\eta_2}(\mathbb{R},X)$, then for $\eta_1,\eta_2 \in [\eta_{\minn},\eta_{\maxx}]$ we have that $u_{\eta_2,s}^\star = \mathcal{J}_s^{\eta_2,\eta_1} \circ u_{\eta_1,s}^\star$. These embedding operators will play the role of $J_0$ and $J$ defined in \Cref{lemma:fixedpointsmooth}. The following proof is a combination of \cite[Theorem IX.7.7]{Diekmann1995}, \cite[Theorem 7.1.1]{Church2018} and \cite[Theorem 7.1]{Hupkes2008} to our setting.

\begin{theorem} \label{thm:smoothnesscmt}
For each $l \in \{1,\dots,k\}$ and $\eta \in (l\eta_{\minn},\eta_{\maxx}],$ the mapping $\mathcal{J}_s^{\eta,\eta_{\minn}} \circ {u}_{\eta_{\minn},s}^\star : X_0(s) \to \BC_s^\eta(\mathbb{R},X)$ is of the class $C^l$ provided that $\delta > 0$ is sufficiently small.
\end{theorem}
\begin{proof}
We prove this by induction on $l$. Let $l = k =1$ and $\eta \in (\eta_{\minn},\eta_{\maxx}]$. We show that \Cref{lemma:fixedpointsmooth} applies with the Banach spaces
\begin{equation*}
    Y_0 = V_\delta^{\eta_{\minn},s}(\mathbb{R},X), \quad Y = \BC_s^{\eta_{\minn}}(\mathbb{R},X), \quad Y_1 = \BC_s^{\eta}(\mathbb{R},X), \quad \Lambda = X_0(s)
\end{equation*}
and operators
\begin{align*}
    f(u,\varphi) &= U(\cdot,s)\varphi + \mathcal{K}_s^{\eta_{\minn}}(\tilde{R}_{\delta,s}(u)), \quad \forall (u,\varphi) \in \BC_s^{\eta_{\minn}}(\mathbb{R},X) \times X_0(s), \\
    f^{(1)}(u,\varphi) &=  \mathcal{K}_s^\eta \circ \tilde{R}_{\delta,s}^{(0,1)}(u) \in \mathcal{L}(\BC_s^{\eta}(\mathbb{R},X)), \quad \forall (u,\varphi) \in  V_{\delta,s}^{\eta}(\mathbb{R},X) \times X_0(s), \\
    f_1^{(1)}(u,\varphi) &=  \mathcal{K}_s^{\eta_{\minn}} \circ \tilde{R}_{\delta,s}^{(0,1)}(u) \in \mathcal{L}(\BC_s^{\eta_{\minn}}(\mathbb{R},X)), \quad \forall (u,\varphi) \in  V_{\delta,s}^{\eta_{\minn}}(\mathbb{R},X) \times X_0(s),
\end{align*}
with embeddings $J = \mathcal{J}_s^{\eta, \eta_{\minn}}$ and $J_0 : V_\delta^{\eta_{\minn},s}(\mathbb{R},X) \hookrightarrow \BC_s^{\eta_{\minn}}(\mathbb{R},X)$. To verify condition 1 of \Cref{lemma:fixedpointsmooth}, we must show that the map
\begin{equation*}
    V_\delta^{\eta_{\minn},s}(\mathbb{R},X) \times X_0(s) \ni (u,\varphi) \mapsto g(u,\varphi) = \mathcal{J}_s^{\eta, \eta_{\minn}} [U(\cdot,s)\varphi + \mathcal{K}_s^{\eta_{\minn}}(\tilde{R}_{\delta,s}(J_0u))] \hspace{-2pt} \in \BC_s^{\eta}(\mathbb{R},X)
\end{equation*}
is $C^1$-smooth. Notice that the embedding operator $J$ is $C^1$-smooth, as well as $\varphi \mapsto U(\cdot,s)\varphi$. Furthermore, from \Cref{lemma: smoothness1} the map $J_0u \mapsto \tilde{R}_{\delta,s}(J_0u)$ is $C^1$-smooth and hence $g$ is $C^1$-smooth by the continuity of the linear embedding $J_0$. Verification of the equalities $D_1 g(y_0,\lambda) \xi = Jf^{(1)}(J_0y_0,\lambda)J_0$ and $ Jf^{(1)}(J_0y_0,\lambda)y = f_1^{(1)}(J_0y_0,\lambda)Jy$ is straightforward. 

Let us now verify condition 2. The Lipschitz claim follows immediately from the small Lipschitz constant for $U(\cdot,s)\varphi + \mathcal{K}_s^{\eta_{\minn}}(\tilde{R}_{\delta,s}(u))$ by choosing $\delta$ sufficiently small. Furthermore, the uniform boundedness claims hold because the embedding operators are bounded. 

For condition 3, the unique fixed point is $u_{\eta_{\minn},s}^\star = J_0 \circ \Phi$, where $\Phi : X_0(s) \to V_{\delta,s}^{\eta_{\minn}}(\mathbb{R},X)$ is defined by $\Phi(\varphi) := u_{\eta_{\minn},s}^\star(\varphi)$ for all $\varphi \in X_0(s)$. The map $\Phi$ is well-defined due to \Cref{lemma:P0hat} and is continuous due to \Cref{thm:sol lipschitz}. 

To verify condition 4, we must check that the map
\begin{equation*}
    V_\delta^{\eta_{\minn},s}(\mathbb{R},X) \times X_0(s) \ni (u,\varphi) \mapsto f(J_0u,\varphi) = U(\cdot,s)\varphi + \mathcal{K}_s^{\eta_{\minn}}(\tilde{R}_{\delta,s}(J_0u)) \in \BC_s^{\eta_{\minn}}(\mathbb{R},X)
\end{equation*}
has continuous partial derivative in the second variable. This is clear since the map $\varphi \mapsto f(J_0u,\varphi)$ is linear. 

To verify condition 5, we have to check that the map
\begin{equation*}
    (u,\varphi) \mapsto J \circ f^{(1)}(J_0u,\varphi) = \mathcal{J}_s^{\eta,\eta_{\minn}} \circ  \mathcal{K}_s^\eta \circ \tilde{R}_{\delta,s}^{(1)}(u)
\end{equation*}
from $V_\delta^{\eta_{\minn},s}(\mathbb{R},X) \times X_0(s)$ to $\mathcal{L}(X_0(s),\BC_s^\eta(\mathbb{R},X))$
is continuous. This again follows from the fact that the embedding operators are continuous and the smoothness of $\tilde{R}_{\delta,s}$ from \Cref{lemma: smoothness1}.

Since all conditions of \Cref{lemma:fixedpointsmooth} are satisfied, we conclude that $\mathcal{J}_s^{\eta,\eta_{\minn}} \circ {u}_{\eta_{\minn},s}^\star$ is $C^1$-smooth and the Fr\'echet derivative $D(\mathcal{J}_s^{\eta,\eta_{\minn}} \circ {u}_{\eta_{\minn},s}^\star) \in \mathcal{L}(X_0(s), \BC_s^\eta(\mathbb{R},X))$ is the unique solution $w^{(1)}$ of the equation
\begin{equation} \label{eq:w^(1)}
    w^{(1)} = \mathcal{K}_s^{\eta_{\minn}} \circ \tilde{R}_{\delta,s}^{(1)}(u_{s,\eta_{\minn}}^\star(\varphi))w^{(1)} + U(\cdot,s) =: F_{\eta_{\minn}}^{(1)}(w^{(1)},\varphi),
\end{equation}
where $F_{\eta_{\minn}}^{(1)} : \mathcal{L}(X_0(s),\BC_s^\eta(\mathbb{R},X)) \times X_0(s) \to \mathcal{L}(X_0(s),\BC_s^\eta(\mathbb{R},X))$. Notice that $F_{\eta_{\minn}}^{(1)}(\cdot,\varphi)$ is a uniform contraction for each $\eta \in [\eta_{\minn},\eta_{\maxx}]$ and hence its unique fixed point $u_{\eta_{\minn},s}^{\star,(1)}(\varphi) \in \mathcal{L}(X_0(s),\BC_s^{\eta_{\minn}}(\mathbb{R},X)) \subseteq \mathcal{L}(X_0(s), \BC_s^\eta(\mathbb{R},X)))$ for $\eta \geq \eta_{\minn}$. Also, the mapping $u_{\eta_{\minn},s}^{\star,(1)} : X_0(s) \to \BC_s^\eta(\mathbb{R},X))$ is continuous if $\eta \in (\eta_{\minn},\eta_{\maxx}]$.

Now, consider any integer $1 \leq l < k$ and suppose that for all $1 \leq q \leq l$ and all $ \eta \in (q \eta_{\minn},\eta_{\maxx})$ the mapping $\mathcal{J}_s^{\eta,\eta_{\minn}} \circ {u}_{\eta_{\minn},s}^\star$ is $C^q$-smooth with $D^q(\mathcal{J}_s^{\eta,\eta_{\minn}} \circ {u}_{\eta_{\minn},s}^\star) = \mathcal{J}_s^{\eta,\eta_{\minn}} \circ {u}_{\eta_{\minn},s}^{\star,(q)}$ and ${u}_{\eta_{\minn},s}^{\star,(q)}(\varphi) \in \mathcal{L}^q(X_0(s)^q, \BC_s^{q\eta_{\minn}}(\mathbb{R},X))$ such that the mapping $\mathcal{J}_s^{\eta,\eta_{\minn}} \circ {u}_{\eta_{\minn},s}^{\star,(q)} : X_0(s) \to \mathcal{L}^q(X_0(s)^q, \BC_s^\eta(\mathbb{R},X))$ is continuous for $\eta \in (q\eta_{\minn},\eta_{\maxx}]$. Suppose also for any $\varphi \in X_0(s)$ that $u_{\eta_{\minn},s}^{\star,(l)}(\varphi)$ is the unique solution $w^{(l)}$ of an equation of the form
\begin{equation*}
    w^{(l)} = \mathcal{K}_s^{\eta_{\minn}p} \circ \tilde{R}_{\delta,s}^{(l)}(u_{s,\eta_{\minn}}^\star(\varphi))w^{(l)} + H_{\eta_{\minn}}^{(l)}(\varphi) =: F_{\eta_{\minn}}^{(l)}(w^{(l)},\varphi),
\end{equation*}
with $H_{\eta_{\minn}}^{(1)}(\varphi) = 0$ and for $\nu \in [\eta_{\minn},\eta_{\maxx}]$ and $l \geq 2$ the map $H_{\nu}^{(l)}(\varphi)$ is a finite sum of terms of the form
\begin{equation*}
    \mathcal{K}_s^{l \nu} \circ \tilde{R}_{\delta,s}^{(0,q)}(u_{\nu,s}^\star(\varphi))(u_{\nu,s}^{\star,(r_1)}(\varphi),\dots, u_{\nu,s}^{\star,(r_q)}(\varphi)),
\end{equation*}
with $2 \leq q \leq l$ and $1 \leq r_i < l$ for $i=1,\dots,q$ such that $r_1 + \dots + r_q = l$. Under these assumptions we have that the mapping $F_{\eta}^{(l)} : \mathcal{L}^l(X_0(s)^l, \BC_s^{l\eta}(\mathbb{R},X)) \times X_0(s) \to \mathcal{L}^l(X_0(s)^l,\BC_s^\eta(\mathbb{R},X))$ is a uniform contraction for all $\eta \in [\eta_{\minn},\frac{1}{l}\eta_{\maxx}]$ due to \Cref{lemma: smoothness1}. 

Fix some $\eta \in ((l+1)\eta_{\minn},\eta_{\maxx}]$ and choose $\eta_{\minn} < \sigma < (l+1)\sigma < \mu < \eta$. We show that \Cref{lemma:fixedpointsmooth} applies with the Banach spaces
\begin{align*}
    Y_0 &= \mathcal{L}^l(X_0(s)^l, \BC_s^{l \sigma}(\mathbb{R},X)), \quad Y = \mathcal{L}^l(X_0(s)^l, \BC_s^{\mu}(\mathbb{R},X)), \\
    Y_1 &= \mathcal{L}^l(X_0(s)^l, \BC_s^{\eta}(\mathbb{R},X)), \quad \Lambda = X_0(s)
\end{align*}
and operators
\begin{align*}
    f(u,\varphi) &= \mathcal{K}_s^\mu \circ \tilde{R}_{\delta,s}^{(0,1)}(u_{\eta_{\minn},s}^{\star}(\varphi))u + H_{\mu/l}^{(l)}(\varphi), \quad \forall (u,\varphi) \in \mathcal{L}^l(X_0(s)^l, \BC_s^{\mu}(\mathbb{R},X)) \times X_0(s), \\
    f^{(1)}(u,\varphi) &=  \mathcal{K}_s^\mu \circ \tilde{R}_{\delta,s}^{(0,1)}(u_{\eta_{\minn},s}^{\star}(\varphi)) \in \mathcal{L}(\mathcal{L}^l(X_0(s)^l, \BC_s^{\mu}(\mathbb{R},X))), \\
    f_1^{(1)}(u,\varphi) &=  \mathcal{K}_s^\eta \circ \tilde{R}_{\delta,s}^{(0,1)}(u_{\eta_{\minn},s}^{\star}(\varphi)) \in \mathcal{L}(\mathcal{L}^l(X_0(s)^l, \BC_s^{\eta}(\mathbb{R},X))),
\end{align*}
We start with verifying condition 1. We have to check that the map
\begin{equation*}
    (u,\varphi) \hspace{-2pt} \mapsto  \hspace{-2pt} \mathcal{J}_s^{\eta,\mu}[\mathcal{K}_s^\mu \circ \tilde{R}_{\delta,s}^{(0,1)}(u_{\eta_{\minn},s}^{\star}(\varphi))u + H_{\mu/l}^{(l)}(\varphi)]
\end{equation*}
from $\mathcal{L}^l(X_0(s)^l, \BC_s^{p \sigma}(\mathbb{R},X)) \times X_0(s)$ to $\mathcal{L}^l(X_0(s)^l, \BC_s^{\eta}(\mathbb{R},X))$ is $C^1$-smooth, where now $\mathcal{J}_s^{\eta,\mu} : \mathcal{L}^l(X_0(s)^l, \BC_s^{\mu}(\mathbb{R},X)) \hookrightarrow \mathcal{L}^l(X_0(s)^l, \BC_s^{\eta}(\mathbb{R},X))$ is a continuous embedding. The mapping defined above $C^1$-smooth in the first variable since it is linear. For the second variable, notice that the map $\varphi \mapsto \mathcal{J}_s^{\eta,\mu} \circ \mathcal{K}_s^\mu \circ \tilde{R}_{\delta,s}^{(1)}(u_{\eta_{\minn},s}^{\star}(\varphi))$ is $C^1$ due to \Cref{lemma: smoothness4} with $\mu > (l+1)\sigma$ and the $C^1$ smoothness of $\varphi \mapsto \mathcal{J}_s^{\sigma \eta_{\minn}} \circ u_{\eta_{\minn},s}^\star(\varphi)$ with $\sigma > \eta_{\minn}$. For the $C^1$ smoothness of $\varphi \mapsto \mathcal{J}_s^{\eta,\mu} \circ H_{\mu/l}^{(l)}(\varphi)$, we get differentiability from \Cref{lemma: smoothness4} and hence we have that the derivative of $\varphi \mapsto H_{\mu/l}^{(l)}(\varphi)$ is a sum of terms of the form
\begin{equation*}
\begin{split}
    &\mathcal{K}_s^\mu \circ \tilde{R}_{\delta,s}^{(0,q+1)}(u_{\eta_{\minn},s}^\star(\varphi))(u_{\eta_{\minn},s}^{\star,(r_1)}(\varphi),\dots, u_{\eta_{\minn},s}^{\star,(r_q)}(\varphi)) \\
        &+ \sum_{j=1}^q \mathcal{K}_s^\mu \circ \tilde{R}_{\delta,s}^{(q)}(u_{\eta_{\minn},s}^\star(\varphi))(u_{\eta_{\minn},s}^{\star,(r_1)}(\varphi),\dots,u_{\eta_{\minn},s}^{\star,(r_j + 1)}(\varphi),\dots, u_{\eta_{\minn},s}^{\star,(r_q)}(\varphi))
\end{split}
\end{equation*}
and each $u_{\eta_{\minn},s}^{\star,(r_j)}$ is a map from $X_0(s)$ into $\BC_s^{j\sigma}(\mathbb{R},X)$. Applying \Cref{lemma: smoothness1} with $\mu > (l+1)\sigma$ ensures continuity of $DH_{\mu/l}^{(l)}(\varphi)$ and also then continuity of $\mathcal{J}_s^{\eta,\mu}DH_{\mu/l}^{(l)}(\varphi)$. The remaining calculations from condition 1 are easily checked. Condition 4 can be proven similarly.

The Lipschitz condition and boundedness for condition 2 follows by the choice of $\delta > 0$ defined at the beginning and the uniform contractivity of $H_{\mu/l}^{(l)}$ described above. Let us now prove condition 3. Let us write
\begin{equation*}
     \mathcal{K}_s^\eta \circ \tilde{R}_{\delta,s}^{(0,1)}(u_{\eta_{\minn},s}^{\star}(\varphi)) = \mathcal{J}_{s}^{\eta,\mu} \circ \mathcal{K}_s^\mu \circ R_{\delta,s}^{(0,1)}(u_{\eta_{\minn},s}^{\star})(\varphi)
\end{equation*}
and by applying \Cref{lemma: smoothness1} together with the $C^1$-smoothness of $u_{\eta_{\minn},s}^{\star}$ to obtain continuity of $\varphi \mapsto \tilde{R}_{\delta,s}^{(0,1)}(u_{\eta_{\minn},s}^{\star}(\varphi))$. This also proves condition 5. All the conditions from \Cref{lemma:fixedpointsmooth} are satisfied, and so we conclude that $u_{\eta_{\minn},s}^{(l)} : X_0(s) \to \mathcal{L}^l(X_0(s)^l,\BC_s^\eta(\mathbb{R},X))$ is of the class $C^1$ with derivative $u_{\eta_{\minn},s}^{(l + 1)} = Du_{\eta_{\minn},s}^{(l)} \in \mathcal{L}^{l+1}(X_0(s)^{l+1}, \BC_s^\eta(\mathbb{R},X))$ given by the unique solution $w^{(l+1)}$ of the equation
\begin{equation*}
    w^{(l+1)} = \mathcal{K}_s^{\mu} \circ \tilde{R}_{\delta,s}^{(1)}(u_{\eta_{\minn},s}^\star(\varphi))w^{(l+1)} + H_{\mu/(l+1)}^{(l+1)}(\varphi),
\end{equation*}
where $H_{\mu/(l+1)}^{(l+1)}(\varphi) = \mathcal{K}_s^\mu \circ \tilde{R}_{\delta,s}^{(0,2)}(u_{\eta_{\minn},s}^\star(\varphi))(u_{\eta_{\minn},s}^{\star,(l)}(\varphi),u_{\eta_{\minn},s}^{\star,(1)}(\varphi)) + DH_{\mu/l}^{(l)}(\varphi)$. Similar arguments of the proof of the $l=k=1$ case show that the unique fixed point $u_{\eta_{\minn},s}^{\star,(l + 1)} \in \mathcal{L}^{l+1}(X_0(s)^{l+1},\BC_s^{\eta_{\minn}(l+1)}(\mathbb{R},X))$. Hence, the map $\mathcal{J}_s^{\eta,\eta_{\minn}} \circ {u}_{\eta_{\minn},s}^\star : X_0(s) \to \BC_s^\eta(\mathbb{R},X)$ is of the class $C^{l + 1}$ if $\eta \in ((l+1)\eta_{\minn},\eta_{\maxx}]$ which completes the proof.
\end{proof}

We also show that each partial derivative of the center manifold in the second component is uniformly Lipschitz continuous. The proof is inspired by \cite[Corollary 8.2.1.2]{Church2018}.
\begin{corollary} \label{cor:LipschitzC}
For each $l \in \{0,\dots,k\}$, there exists a constant $L(l) > 0$ such that $\|D_2^l \mathcal{C}(t,\varphi) - D_2^l \mathcal{C}(t,\psi) \| \leq L(l) \|\varphi - \psi\|$ for all $t \in \mathbb{R}$ and $\varphi,\psi \in X_0(t)$.
\end{corollary}
\begin{proof}
For $l = 0$, the result is already proven in \Cref{cor:lipschitzCMT}. Now let $l \in \{1,\dots,k\}$. Then, from the proof of \Cref{thm:smoothnesscmt} we see that $u_{\eta_{\minn},s}^{\star,(l)}$ is the unique solution of a fixed point problem, where the right hand-side is a contraction with a Lipschitz constant $L(l)$ independent of $s$. Using the same strategy as the proof of \Cref{cor:lipschitzCMT}, we obtain the desired result. 
\end{proof}

\begin{corollary} \label{cor:tangent}
The center manifold $\mathcal{W}^c$ is $C^k$-smooth and its tangent bundle is $X_0$ i.e. $D_2 \mathcal{C}(t,0)\varphi = \varphi$ for all $(t,\varphi) \in X_0$.
\end{corollary}
\begin{proof}
Let $ \eta \in [\eta_{\minn},\eta_{\maxx}] \subset (0,\min\{-a,b\})$ such that $k \eta_{\minn} < \eta_{\maxx}$. Define for any $t \in \mathbb{R}$ the evolution map $\evo_t : \BC_t^\eta(\mathbb{R},X) \to X$ as $\evo_t(f) := f(t)$. Then, for all $(t,\varphi) \in X_0$ we get
\begin{equation*}
    \mathcal{C}(t,\varphi) = \evo_t(u_{\eta_{\minn},t}^{\star}(\varphi)) = \evo_t(\mathcal{J}_t^{\eta, \eta_{\minn}}u_{\eta_{\minn},t}^{\star}(\varphi)).
\end{equation*}
It is clear that $\evo_t \in \mathcal{L}(\BC_t^\eta(\mathbb{R},X), X)$ and hence it follows from \Cref{thm:smoothnesscmt} that $\mathcal{C}$ is of the class $C^k$. This shows that the center manifold $\mathcal{W}^c$ is $C^k$-smooth. Moreover,
\begin{equation*}
    D_2\mathcal{C}(t,0)\varphi = \evo_t(D(\mathcal{J}_t^{\eta, \eta_{\minn}} \circ u_{\eta_{\minn},t}^{\star})(0)\varphi) = \evo_t(u_{\eta_{\minn},t}^{\star,(1)}(0)\varphi).
\end{equation*}
As $D\tilde{R}_{\delta,t}(0) = 0$ and $u_{\eta_{\minn},t}^\star(0) = 0$ for all $t \in \mathbb{R}$, we get from \eqref{eq:w^(1)} that $u_{\eta_{\minn},t}^{\star,(1)}(0) = U(\cdot,t)$ and so $D_2\mathcal{C}(t,0)\varphi = \evo_t(U(\cdot,t)\varphi) = \varphi$, as claimed.
\end{proof}

It follows from the previous corollary that the local center manifold $\mathcal{W}_{\loc}^c$ is also $C^k$-smooth and has $X_0$ as a tangent bundle. Let us now take a look into periodicity. 

\begin{theorem} \label{thm:periodic}
If the time-dependent nonlinear perturbation $R : \mathbb{R} \times X \to X^{\odot \star}$ is $T$-periodic in the first variable, then there exists a $\delta > 0$ such that $\mathcal{C}(t+T,\varphi) = \mathcal{C}(t,\varphi)$ for all $t \in \mathbb{R}$ whenever $\|\varphi\| < \delta$.
\end{theorem}
\begin{proof}
The proof of this theorem is essentially the same as \cite[Lemma 8.3.1]{Church2018}, which was obtained for impulsive DDEs. To obtain the result for classical DDEs, one has to ignore the discontinuous impulses and make the logical substitution $\mathcal{R}\mathcal{C}\mathcal{R} \to X$ and put everything towards the sun-star setting.
\end{proof}

\section{Variation-of-constants formulas and one-to-one correspondences} \label{appendix: variation constants}
This section of the appendix consists of two subsections. In the first subsection, we study the interplay between solutions of inhomogeneous linear abstract ODEs and their associated inhomogeneous linear AIEs. In the second subsection, we prove that there is a one-to-one correspondence between solutions of \eqref{eq:T-DDEphi} and \eqref{eq:T-AIEphi} by using the results from \Cref{subsec:interplay}. This result is important when one applies the sun-star machinery towards DDEs, see for example the local center manifold theorem for DDEs in \Cref{cor:CMT DDE}.

\subsection{Inhomogeneous perturbations to linear abstract ODEs and AIEs} \label{subsec:interplay}
In this subsection, we work with the same tools and notation as presented in \Cref{subsec: linear perturbation}. Let $J \subseteq \mathbb{R}$ be an interval and $s \in J$ an initial starting time. Applying an inhomogeneous perturbation $f : J \to X^{\odot \star}$ on the generator $A^{\odot \star}(t)$ to \eqref{eq:T-LAODEphi} yields
\begin{equation} \label{eq:T-LAODEf}
    \begin{dcases}
            d^\star (j \circ u)(t) = A^{\odot \star}(t)ju(t) + f(t), \quad &t \geq s, \\
            u(s) = \varphi, \quad & \varphi \in X,
    \end{dcases}
\end{equation}
which suggest the variation-of-constants formula
\begin{equation} \label{eq:T-LAIEf}
    u(t) = U(t,s)\varphi + j^{-1} \int_s^t U^{\odot \star}(t,\tau) f(\tau) d\tau, \quad \varphi \in X.
\end{equation}
It is also possible to perturb the generator $A_0^{\odot \star}$ by $\varphi \mapsto B(t)\varphi + f$, for some fixed $t \in J$. This yields 
\begin{equation} \label{eq:T-LAODEf2}
    \begin{dcases}
            d^\star (j \circ u)(t) = A_0^{\odot \star}ju(t) + B(t)u(t) + f(t), \quad &t \geq s, \\
            u(s) = \varphi, \quad & \varphi \in X,
    \end{dcases}
\end{equation}
which suggests the variation-of-constants formula 
\begin{equation} \label{eq:T-LAIEf2}
    u(t) = T_0(t-s)\varphi + j^{-1} \int_s^t T_0^{\odot \star}(t-\tau) [B(\tau)u(\tau) + f(\tau)] d\tau, \quad \varphi \in X.
\end{equation}
Solutions to the linear problems above are similarly defined as in \Cref{subsec: linear perturbation}. It is clear by \eqref{eq:defAsunstar s} that a solution to \eqref{eq:T-LAODEf} is also a solution to \eqref{eq:T-LAODEf2} and vice versa. In this sense we call \eqref{eq:T-LAODEf} and \eqref{eq:T-LAODEf2} \emph{equivalent}. We would like to establish a similar equivalence between \eqref{eq:T-LAIEf} and \eqref{eq:T-LAIEf2}. When the perturbation $B$ does not depend on time, one can work with integrated semigroups in the $\odot$-reflexive case to prove the equivalence between the inhomogeneous autonomous problems, see \cite[Proposition 2.5]{Clement1989} and \cite[Lemma III.2.23]{Diekmann1995}. However, if we would succeed to generalize this approach towards time-dependent systems, it would probably only work in a $\odot$-reflexive setting. To overcome this problem, we will generalize the non-$\odot$-reflexive approach by Janssens in \cite[Section 3]{Janssens2020} towards a time-dependent setting while still assuming the $\odot$-reflexivity. The non-$\odot$-reflexive case is still an open problem, see \Cref{sec:conclusions}.

Before generalizing Janssens approach to a time-dependent setting, notice that the weak$^\star$ Riemann integral in \eqref{eq:T-LAIEf} is well-defined when $f$ is assumed to be continuous, see \Cref{lemma:wk*integral Usunstar}. Furthermore, the weak$^\star$ Riemann integral in \eqref{eq:T-LAIEf2} is well-defined when $f$ is assumed to be continuous because then the map $[s,t] \ni \tau \mapsto B(\tau)u(\tau) + f(\tau) \in X^{\odot \star}$ is continuous, see \cite[Lemma 2.2]{Clement1988}. This already indicates that continuity of $f$ is a sufficient condition for the well-definedness of the variation-of-constants formulas.

Before showing any equivalence between the four proposed problems above, let us first prove that at least one of them induces a unique solution on a subinterval of $J$. The following result is inspired by \cite[Proposition 20]{Janssens2020}.

\begin{proposition} \label{prop:unique TLAIEf1}
Let $I$ be a compact subinterval of $J$. The following two statements hold.
\begin{enumerate}
    \item For every $\varphi \in X$ and $f \in C(I,X^{\odot \star})$ there exists a unique solution $u_{\varphi,f}$ of \eqref{eq:T-LAIEf2} on $I$ and the map
    \begin{equation*}
        X \times C(I,X^{\odot \star}) \ni (\varphi,f) \mapsto u_{\varphi,f} \in C(I,X)
    \end{equation*}
    is continuous.
    \item If  $\varphi \in j^{-1} \mathcal{D}(A_0^{\odot \star})$ and $f : I \to X^{\odot \star}$ is locally Lipschitz, then there exist sequences of Lipschitz  functions $u_m : I \to X$ and $f_m : I \to X^{\odot \star}$ such that
    \begin{equation} \label{eq:um fm}
        u_m(t) = T_0(t)\varphi + j^{-1} \int_s^t  T_0^{\odot \star}(t-\tau) [B(\tau)u_m(\tau) + f_m(\tau)] d\tau, \quad \forall t \in I,
    \end{equation}
    and $f_m \to f$ and $u_m \to u_{\varphi,f}$ as $m \to \infty$, uniformly on $I$.
\end{enumerate}
\end{proposition}
\begin{proof}
We show the first claim by a fixed point argument. Choose $M \geq 1$ and $\omega \in \mathbb{R}$ such that $\|T_0(t)\| \leq Me^{\omega t}$. On the space $C(I,X)$, we introduce the one-parameter family of equivalent norms
\begin{equation*}
    \|u\|_\eta := \sup_{t \in I} e^{-\eta t} \|u(t)\|, \quad \eta \in \mathbb{R},
\end{equation*}
that makes $(C(I,X), \|\cdot\|_\eta)$ a Banach space for each $ \eta \in \mathbb{R}$. For each fixed $(\varphi,f) \in X \times C(I,X^{\odot \star})$ define the operator $K_{\varphi,f} : C(I,X) \to C(I,X)$ as
\begin{equation} \label{eq:Kphif}
    (K_{\varphi,f}u)(t) := T_0(t-s)\varphi + j^{-1} \int_s^t T_0^{\odot \star}(t-\tau) [B(\tau)u(\tau) + f(\tau)] d\tau, \quad \forall t \in I.
\end{equation}
Define $N := \sup_{(t,s) \in \Omega_I} W(t,s)$ and notice that $N$ is finite because $I$ is compact and $W$ is continuous, where $W : \Omega_I \to \mathbb{R}$ is defined as $W(t,s) := \sup_{s \leq \tau \leq t} \|B(\tau)\|$. Let $\eta > \omega$, then for all $u_1,u_2 \in C(I,X)$ and $t \in I$ we get
\begin{align*}
    e^{-\eta t} \|(K_{\varphi,f}u_1)(t) - (K_{\varphi,f}u_2)(t)\| & \leq \|j^{-1}\| M N \int_s^t e^{-(\eta-\omega) (t-\tau)} e^{-\eta \tau}\|u_1(\tau) - u_2(\tau)\| d\tau \\
    & \leq \|j^{-1}\| M N \|u_1 - u_2\|_{\eta} \int_s^t e^{-(\eta-\omega) (t-\tau)}  d\tau \\
    &= \frac{\|j^{-1}\| M N (1 - e^{-(t-s)(\eta-\omega)})}{\eta - \omega} \|u_1 - u_2\|_{\eta} \\
    & \leq \frac{\|j^{-1}\| M N}{\eta - \omega} \|u_1 - u_2\|_{\eta}.
\end{align*}
If we choose $\eta > \omega$ large enough such that $\frac{\|j^{-1}\| M N}{\eta - \omega} \leq \frac{1}{2}$, then $K_{\varphi,f}$ is a contraction on $C(I,X)$ with respect to the $\|\cdot\|_\eta$-norm. The uniqueness of $u$ now follows from the Banach fixed point theorem. For a fixed $u \in C(I,X)$, it follows that the map
\begin{equation*}
    X \times C(I,X^{\odot \star}) \ni (\varphi,f) \mapsto K_{\varphi,f}u \in C(I,X)
\end{equation*}
is continuous. 

Let us now show the second assertion. Let $\varphi \in j^{-1} \mathcal{D}(A_0^{\odot \star})$ and $f$ be locally Lipschitz, we will show that $K_{\varphi,f}$ maps $\Lip(I,X)$ into itself, where $\Lip(I,X)$ denotes the subspace of $C(I,X)$ consisting of $X$-valued Lipschitz continuous functions defined on $I$. From the theory of Favard classes of $\mathcal{C}_0$-semigroups and the important equalities \cite[Equation (19)]{Janssens2020}, it follows immediately that $T_0(\cdot)\varphi$ is in $\Lip(I,X)$. Let $u \in \Lip(I,X)$ be given, since $B$ is Lipschitz continuous and $f$ is assumed to be locally Lipschitz we know that $t \mapsto B(t)u(t) + f(t)$ is locally Lipschitz on $I$ and takes values in $X^{\odot \star}$. Hence, the map $v_1(\cdot,s,B(\cdot)u+f) : I \to j(X)$ defined by
\begin{equation*}
    v_1(t,s,B(\cdot)u+f) := \int_s^t T_0^{\odot \star}(t-\tau)[B(\tau)u(\tau)+f(\tau)] d\tau, \quad \forall t \in I,
\end{equation*}
is weak$^\star$ continuously differentiable and so locally Lipschitz by \cite[Remark 16]{Janssens2020}. It follows that $K_{\varphi,f}u = T_0(\cdot)\varphi + j^{-1}v_1(\cdot,s,B(\cdot)u+f)$ is in $\Lip(I,X)$. Now, let $u_0 \in \Lip(I,X)$ be arbitrary. The sequence $(u_m)_{m \in \mathbb{N}}$ defined by
\begin{equation*}
    u_m := K_{\varphi,f} u_{m-1}, \quad m \geq 1,
\end{equation*}
is in $\Lip(I,X)$. We only have to show that there exists a sequence of $X^{\odot \star}$-valued Lipschitz continuous functions $(f_m)_{m \in \mathbb{N}}$ defined on $I$ that satisfies the integral formula. It follows from \eqref{eq:Kphif} that for any $t \in I$ and $m \geq 1$ we have
\begin{align*}
    u_m(t) &= K_{\varphi,f} u_{m-1}(t) \\
    &= T_0(t-s)\varphi + j^{-1} \int_s^t T_0^{\odot \star}(t-\tau)[B(\tau)u_m(\tau) + f(\tau) + B(\tau)[u_{m-1}(\tau) - u_m(\tau)]] d\tau. 
\end{align*}
If we define for any $m \geq 1$ the functions $f_m : I \to X^{\odot \star}$ as $f_m:= f + B(\cdot)(u_{m-1} - u_m)$, then each $f_m$ is Lipschitz continuous and $f_m \to f$ uniformly on $I$ because
\begin{align*}
    \|f_m - f \| &\leq \sup_{(t,s) \in \Omega_I}W(t,s) \|u_{m-1} - u_m\| \\
    &\leq N [\|u_{m-1} - u_{\varphi,f}\| + \|u_{\varphi,f} - u_m\|] \\
    &\to 0, \quad \mbox{ as }  m \to \infty,
\end{align*}
as both $u_{m-1}$ and $u_m$ converge to $u_{\varphi,f}$ uniformly on $I$ as $m \to \infty$.
\end{proof}

The next proposition shows under which conditions on $\varphi$ and $f$ the unique solution to the abstract integral equation \eqref{eq:T-LAIEf2}, proven in \Cref{prop:unique TLAIEf1}, induces a solution to the abstract ordinary differential equation \eqref{eq:T-LAODEf2}. The proof is inspired by \cite[Corollary 19]{Janssens2020}.

\begin{proposition} \label{prop:AIE1 ODE1}
Suppose that $\varphi \in j^{-1} \mathcal{D}(A_0^{\odot \star})$ and $f : J \to X^{\odot \star}$ is locally Lipschitz. If $u$ is a locally Lipschitz  solution of \eqref{eq:T-LAIEf2} on a subinterval $I$ of $J$ then $u$ is a solution of \eqref{eq:T-LAODEf2} on $I$.
\end{proposition}
\begin{proof}
If we apply $j$ to the abstract integral equation in \eqref{eq:T-LAIEf2}, we get for any $t \in I$ that
\begin{equation} \label{eq:u(t) var}
    ju(t) = T_0^{\odot \star}(t-s)j\varphi + \int_s^t T_0^{\odot \star}(t-\tau) [B(\tau)u(\tau) + f(\tau)] d\tau.
\end{equation}
It follows from the theory of Favard classes for $\mathcal{C}_0$-semigroups \cite[Equation (19)]{Janssens2019} that $\mathcal{D}(A_0^{\odot \star})$ is $T_0^{\odot \star}$-invariant. Hence, the first term on the right side takes values in $\mathcal{D}(A_0^{\odot \star})$ and notice from a $\odot$-variant of \cite[Theorem 2.1]{Clement1987} that this term is weak$^\star$ continuously differentiable with weak$^\star$ derivative
\begin{equation*}
    d^\star (T_0^{\odot\star}(\cdot-s) j\varphi)(t) = A_0^{\odot \star} T_0^{\odot \star}(t-s)j\varphi,
\end{equation*}
Now, $f$ and $u$ are locally Lipschitz continuous functions on $I \subseteq J$ and $B$ is by definition of the time-dependent bounded linear perturbation on $J$. Hence, $g : I \to X^{\odot \star}$ defined by $g(\tau) := B(\tau)u(\tau) + f(\tau)$ for all $\tau \in I$ is locally Lipschitz. Define the function $v_1(\cdot,s,g) : I \to j(X)$ as
\begin{equation*}
    v_1(t,s,g) := \int_s^t T_0^{\odot \star}(t-\tau) g(\tau) d \tau, \quad \forall t \in I.
\end{equation*}
It is clear from \cite[Proposition 2.2]{Clement1989} (or \cite[Proposition 18]{Janssens2020}) that $v_1(\cdot,s,g)$ is weak$^\star$ continuously differentiable, takes values in $\mathcal{D}(A_0^{\odot \star})$ and has weak$^\star$ derivative
\begin{equation*}
    d^\star (v_1(\cdot,s,g))(t) = A_0^{\odot \star}v_1(t,s,g) + g(t).
\end{equation*}
By linearity, it is clear from \eqref{eq:u(t) var} that $u$ takes values in $j^{-1} \mathcal{D}(A_0^{\odot \star})$. Combining all the results yield
\begin{align*}
    d^\star(j \circ u)(t) &= A_0^{\odot \star} T_0^{\odot \star}(t-s)j\varphi + A_0^{\odot \star} \int_s^t T_0^{\odot \star}(t-\tau) [B(\tau)u(\tau) + f(\tau)] d\tau + B(t)u(t) + f(t)\\
    &=  A_0^{\odot \star}ju(t) + B(t)u(t) + f(t).
\end{align*}
This shows that $j \circ u$ is weak$^\star$ continuously differentiable and satisfies \eqref{eq:T-LAODEf2} on $I$ since $u(s) = \varphi$. We conclude that $u : I \to X$ is a solution of \eqref{eq:T-LAODEf2} on $I$.
\end{proof}

Let $u$ be the solution of \eqref{eq:T-LAODEf2} on a subinterval $I$ of $J$ generated by \Cref{prop:AIE1 ODE1}. Hence, $u$ is also a solution of \eqref{eq:T-LAODEf} on $I$ by the equivalence between \eqref{eq:T-LAODEf} and \eqref{eq:T-LAODEf2}. Our next goal is to show that solutions of \eqref{eq:T-LAODEf} are precisely given by the variation-of-constants suggestion presented in \eqref{eq:T-LAIEf}. The proof is inspired by \cite[Proposition 21]{Janssens2020}.

\begin{proposition} \label{prop:T-f 2}
Suppose that $f \in C(J,X^{\odot \star})$ and $I$ is a subinterval of $J$. If $u$ is a solution of \eqref{eq:T-LAODEf} on $I$ then $u$ is given by \eqref{eq:T-LAIEf}.
\end{proposition}
\begin{proof}
Let $t \in I$ be given with $t > s$, where $s$ denotes the starting time. Define the function $w: [s,t] \to X^{\odot \star}$ by $w(\tau) := U^{\odot \star}(t,\tau)ju(\tau)$ for all $\tau \in [s,t]$. We claim that $w$ is weak$^\star$ continuously differentiable with weak$^\star$ derivative
\begin{equation} \label{eq:d*w}
    d^\star w(\tau) = U^{\odot \star}(t,\tau)d^{\star}(j \circ u)(\tau) - U^{\odot \star}(t,\tau)A^{\odot \star}(\tau)ju(\tau), \quad \forall \tau \in [s,t].
\end{equation}
To show this claim, let $\tau \in [s,t]$ and $x^\odot \in X^\odot$ be given. For any $h \in \mathbb{R}$ such that $\tau + h \in [s,t]$ we have
\begin{align*} 
    \langle w(\tau+h) - w(\tau), x^\odot \rangle &=  \langle U^{\odot \star}(t,\tau + h)ju(\tau + h) - U^{\odot \star}(t,\tau)ju(\tau),x^\odot \rangle \nonumber \\
    &= \langle U^{\odot \star}(t,\tau + h)[ju(\tau +h) - ju(\tau)],x^\odot \rangle \nonumber \\
    &+ \langle [U^{\odot \star}(t,\tau + h) - U^{\odot \star}(t,\tau)]ju(\tau),x^\odot \rangle \nonumber\\
    &= \langle ju(\tau +h) - ju(\tau), U^{\odot}(\tau + h,t)x^\odot \rangle  \nonumber\\
    &+ \langle [U^{\odot \star}(t,\tau + h) - U^{\odot \star}(t,\tau)]ju(\tau),x^\odot \rangle.
\end{align*}
Because $U^\odot$ is a strongly continuous backward evolutionary system, we have that $U^{\odot}(\tau + h,t)x^\odot \to U^{\odot}(\tau,t)x^\odot$ in norm as $h \to 0$. Moreover, from the definition of the weak$^\star$ derivative we obtain
\begin{equation*}
    \frac{1}{h}(ju(\tau+h) - ju(\tau)) \to d^\star (j \circ u) (\tau) \quad \mbox{ weakly$^\star$ as } h\to 0,
\end{equation*}
if we can show that the difference quotients remains bounded in the limit. Since $u$ is a solution to \eqref{eq:T-LAODEf}, we know that $j \circ u$ is weak$^\star$ continuously differentiable and so locally Lipschitz continuous by \cite[Remark 16]{Janssens2020}. Because $[s,t]$ is compact, $j \circ u$ is Lipschitz continuous on $[s,t]$ and so the difference quotients remain bounded in the limit. Combining these two facts yield
\begin{equation*}
    \frac{1}{h}\langle ju(\tau +h) - ju(\tau), U^{\odot}(\tau + h,t)x^\odot \rangle \to \langle d^\star (j \circ u) (\tau), U^\odot(t,\tau)x^\odot \rangle \quad \mbox{ as } h \to 0.
\end{equation*}
Furthermore, since $ju(\tau) \in \mathcal{D}(A^{\odot \star}(\tau)) = \mathcal{D}(A_0^{\odot \star})$, it follows from \cite[Theorem 5.5]{Clement1988} that
\begin{equation*}
    \frac{1}{h} \langle [U^{\odot \star}(t,\tau + h) - U^{\odot \star}(t,\tau)]ju(\tau),x^\odot \rangle \to \langle -U^{\odot \star}(t,\tau)A^{\odot \star}(\tau) ju(\tau), x^\odot \rangle \quad \mbox{ as } h \to 0.
\end{equation*}
Consequently, it holds
\begin{equation*}
   \frac{1}{h} \langle w(\tau+h) - w(\tau), x^\odot \rangle \to \langle  U^{\odot \star}(t,\tau)d^{\star}(j \circ u)(\tau) - U^{\odot \star}(t,\tau)A^{\odot \star}(\tau)ju(\tau), x^{\odot} \rangle \quad \mbox{ as } h \to 0,
\end{equation*}
which proves \eqref{eq:d*w}. Substituting the differential equation from \eqref{eq:T-LAODEf} into \eqref{eq:d*w} yields
\begin{equation*}
    d^\star w (\tau) = U^{\odot \star}(t,\tau)f(\tau), \quad \forall \tau \in [s,t],
\end{equation*}
and so $d^\star w$ is weak$^\star$ continuous since $f$ was assumed to be (norm) continuous. Now, for any $x^\odot \in X^\odot$ we get
\begin{align*}
    \langle ju(t) - U^{\odot \star}(t,s)ju(s),x^\odot \rangle &= \langle w(t), x^\odot \rangle - \langle w(s), x^\odot \rangle \\
    &= \int_s^t \langle d^\star w(\tau), x^\odot \rangle d\tau
    = \langle \int_s^t U^{\odot \star}(t,\tau)f(\tau) d\tau, x^\odot \rangle .
\end{align*}
As $x^\odot \in X^{\odot}$ and $t > s$ were arbitrary, we conclude that
\begin{equation*} 
    ju(t) - U^{\odot \star}(t,s)ju(s) = \int_s^t U^{\odot \star}(t,\tau)f(\tau) d\tau.
\end{equation*}
and so
\begin{equation*}
    j[u(t) - U(t,s)u(s)] = \int_s^t U^{\odot \star}(t,\tau)f(\tau) d\tau.
\end{equation*}
By $\odot$-reflexivity of $X$ with respect to $T_0$, and recalling that $j$ is an isomorphism on its image $X^{\odot \odot}$ we get
\begin{equation} \label{eq:inhom u}
    u(t) = U(t,s)u(s) + j^{-1} \int_s^t U^{\odot \star}(t,\tau)f(\tau) d\tau, \quad \forall t \in I,
\end{equation}
which shows the claim since $\varphi = u(s)$. The continuity of $f$ ensures from \Cref{lemma:wk*integral Usunstar} that the weak$^\star$ integral takes values in $j(X)$ and so \eqref{eq:inhom u} is well-defined.
\end{proof}

Let us go full circle now by proving that the unique solutions of \eqref{eq:T-LAIEf2} are given by \eqref{eq:T-LAIEf}. The following result is inspired by \cite[Theorem 22]{Janssens2020}.

\begin{proposition}\label{prop:AIEf}
Suppose that $f \in C(J,X^{\odot \star})$ and $I$ is a subinterval of $J$. The unique solution of \eqref{eq:T-LAIEf2} on $I$ is given by \eqref{eq:T-LAIEf}.
\end{proposition}
\begin{proof}
Let us first assume that $I$ is compact. From \Cref{prop:unique TLAIEf1} we get a unique solution $u_{\varphi,f} : I \to X$ of \eqref{eq:T-LAIEf2} and sequences of Lipschitz functions $u_m : I \to X$ and $f_m : I \to X^{\odot \star}$ that satisfy \eqref{eq:um fm}. For each $m \in \mathbb{N}$, let $\hat{f}_m : I \to X^{\odot \star}$ be a Lipschitz extension of $f_m$ such that $\hat{f}_m |_{I} = f_m$. Substituting $f$ with $\hat{f}_m$ and $u$ with $u_m$ in  \Cref{prop:AIE1 ODE1} shows us that each $u_m$ is a solution to the initial value problem
\begin{equation*}
    \begin{dcases}
            d^\star (j \circ u_m)(t) = A_0^{\odot \star}ju_m(t) + B(t)u_m(t) + \hat{f}_m(t), \quad t \in I, \\
            u_m(s) = \varphi.
    \end{dcases}
\end{equation*}
Recall from \eqref{eq:defAsunstar s} that each $u_m$ also is then also a solution of
\begin{equation*}
    \begin{dcases}
            d^\star (j \circ u_m)(t) = A^{\odot \star}(t)ju_m(t)  + \hat{f}_m(t), \quad & t \in I, \\
            u_m(s) = \varphi.
    \end{dcases}
\end{equation*}
It follows from \Cref{prop:T-f 2}, with $u$ replaced by $u_m$ and $f$ replaced by $\hat{f}_m$, that
\begin{equation} \label{eq: um}
    u_m(t) = U(t,s)\varphi + j^{-1} \int_s^t U^{\odot \star}(t,\tau) f_m(\tau) d\tau, \quad \forall m \in \mathbb{N}, \ t \in I,
\end{equation}
since $\hat{f}_m$ restricted to $I$ precisely is $f_m$. Let us take the limit as $m \to \infty$ in \eqref{eq: um} to obtain
\begin{equation} \label{eq: u_varpi,f}
    u_{\varphi,f}(t) = U(t,s)\varphi + j^{-1} \int_s^t U^{\odot \star}(t,\tau) f(\tau) d\tau, \quad  \forall t \in I,
\end{equation}
for all $(\varphi , f) \in  j^{-1} \mathcal{D}(A_0^{\odot \star}) \times \Lip(I,X^{\odot \star})$. As $j^{-1} \mathcal{D}(A_0^{\odot \star}) \times \Lip(I,X^{\odot \star})$ is dense in $X \times C(I,X^{\odot \star})$, the continuity statement from \Cref{prop:unique TLAIEf1} implies that \eqref{eq: u_varpi,f} also holds for all $\varphi \in X$ and $f \in C(I,X^{\odot \star})$. Hence, the unique solution of \eqref{eq:T-LAIEf2} on $I$ is given by \eqref{eq:T-LAIEf} on $I$. To extend this result towards non-compact subintervals $I$ of $J$ the same proof can be followed as in \cite[Theorem 22]{Janssens2020}.
\end{proof}

\subsection{Equivalence between \eqref{eq:T-DDEphi} and \eqref{eq:T-AIEphi}} \label{subsec:equivalence}

Let us now prove the important one-to-one correspondence between solutions of \eqref{eq:T-DDEphi} and \eqref{eq:T-AIEphi}. To prove this result, we assume weaker assumptions on the (nonlinear) time-dependent perturbations because this is not needed for the proof.

\begin{theorem} \label{thm:one-to-one T-DDE AIE}
Consider \eqref{eq:T-DDEphi} with $L \in C(\mathbb{R},\mathcal{L}(X,\mathbb{R}^n))$ and $G \in C(\mathbb{R} \times X, \mathbb{R}^n)$.
\end{theorem}
\begin{enumerate}
    \item Suppose that $y : [s-h,t_\varphi) \to \mathbb{R}^n$ is a solution of \eqref{eq:T-DDEphi}, then the function $u_\varphi : [s,t_\varphi) \to X$ defined by
    \begin{equation*}
        u_\varphi(t) := y_t, \quad \forall t \in [s,t_\varphi),
    \end{equation*}
    is a solution of \eqref{eq:T-AIEphi}.
    \item Suppose that $u_\varphi : [s,t_\varphi) \to X$ is a solution of \eqref{eq:T-AIEphi}, then the function $y : [s-h,t_\varphi) \to \mathbb{R}^n$ defined by
    \begin{equation*}
        y(t):=
        \begin{cases}
        \varphi(t-s), \quad &s-h \leq t \leq s,\\
        u_\varphi(t)(0), \quad &s \leq t \leq t_\varphi,
        \end{cases}
    \end{equation*}
    is a solution of \eqref{eq:T-DDEphi}.
\end{enumerate}
\begin{proof} Before we start proving the first assertion, notice that the differential equation from \eqref{eq:T-DDEphi} is equivalent to the integral equation
\begin{equation} \label{eq:integratedform}
    y(t) = \varphi(0) + \int_s^t L(\tau)y_\tau + G(\tau,y_\tau) d\tau, \quad t \geq s,
\end{equation}
due to the fundamental theorem of calculus. Let us start with proving the first assertion.

1. Notice that the right-hand side of the abstract integral equation in \eqref{eq:T-LAIEf2} with a $C^k$-smooth function $f = R(\cdot,u_\varphi(\cdot))$ is equivalent to
\begin{equation*}
    T_0(t-s)\varphi + j^{-1} \int_s^t T_0^{\odot \star}(t-\tau) [L(\tau)u_\varphi(\tau) + G(\tau,u_\varphi(\tau))]r^{\odot \star}  d\tau, \quad \forall t \in [s,t_\varphi).
\end{equation*}
It then follows from the action of the shift semigroup \eqref{eq:shift}, the assumption $u_\varphi(t) = y_t$ and \cite[Lemma XII.3.3]{Diekmann1995} where in this lemma the map $g$ must be replaced by the continuous map $L(\cdot)u_\varphi(\cdot) + G(\cdot,u_\varphi(\cdot)))$, since $L \in C(\mathbb{R},\mathcal{L}(X,\mathbb{R}^n))$, $ u_\varphi \in C([s,t_\varphi),X)$ and $G \in C(\mathbb{R} \times X, \mathbb{R}^n)$, that this right-hand side evaluated at $\theta \in [-h,0]$ is equivalent to
\begin{align*}
    &(T_0(t-s)\varphi)(\theta) + j^{-1} \bigg(\int_s^t T_0^{\odot \star}(t-\tau) [L(\tau)u_\varphi(\tau) + G(\tau,u_\varphi(\tau))]r^{\odot \star} d\tau\bigg)(\theta) \\
    &= (T_0(t-s)\varphi)(\theta) + \int_s^{\max \{s,t+\theta\}} L(\tau)u_\varphi(\tau) + G(\tau,u_\varphi(\tau)) d\tau \\
    &= (T_0(t-s)\varphi)(\theta) + \int_s^{\max \{s,t+\theta\}} L(\tau)y_\tau + G(\tau,y_\tau) d\tau \\
    &= \begin{dcases}
    \varphi(t+\theta), \quad &s-h \leq t+\theta \leq s, \\
    \varphi(0) + \int_s^{t+\theta} L(\tau)y_\tau + G(\tau,y_\tau) d\tau, \quad & s \leq t + \theta \leq t_\varphi, 
    \end{dcases} \\
    &= y(t+\theta) = u_\varphi(t)(\theta),
\end{align*}
where the fourth equality holds due to \eqref{eq:integratedform}. Hence, $u_\varphi$ is a solution to \eqref{eq:T-LAIEf2} with $f = R(\cdot,u_\varphi(\cdot))$. It follows from \Cref{prop:AIEf} that $u_\varphi$ then also is a solution of \eqref{eq:T-LAIEf} with $f = R(\cdot,u_\varphi(\cdot))$, which is equivalent to saying that $u_\varphi$ is a solution of \eqref{eq:T-AIEphi}. \\

2. Let us first prove that the function $y$ is continuous on $[s-h,t_\varphi).$ As $\varphi \in X$, it is clear that $y$ is continuous for $t \in [s-h,s]$. As point evaluation acts continuously on elements in $X \ni u_\varphi(t)$, it follows that $y$ is continuous on $[s,t_\varphi)$. Since $u_\varphi(s)(0) = \varphi(0)$ we have that $y \in C([s-h,t_\varphi),\mathbb{R}^n)$.

Our next goal is to show that $y$ satisfies \eqref{eq:T-DDEphi} or equivalently \eqref{eq:integratedform}. Because $u_\varphi$ is a solution of \eqref{eq:T-LAIEf} with $f = R(\cdot,u_\varphi(\cdot))$, we know from \Cref{prop:AIEf} that $u_\varphi$ is then also a solution of \eqref{eq:T-LAIEf2} with $f = R(\cdot,u_\varphi(\cdot))$. It follows from \eqref{eq:shift} and \cite[Lemma XII.3.3]{Diekmann1995} that
\begin{align*}
    y(t) &= u_\varphi(t)(0) \\
    &=(T_0(t-s)\varphi)(0) + j^{-1} \bigg(\int_s^t T_0^{\odot \star}(t-\tau) [L(\tau)u_\varphi(\tau) + G(\tau,u_\varphi(\tau))]r^{\odot \star} d\tau\bigg)(0)\\
    &= \varphi(0) + \int_s^t L(\tau)u_\varphi(\tau) + G(\tau,u_\varphi(\tau)) d\tau.
\end{align*}
It remains to show that $u_\varphi(\tau) = y_\tau$ for all $\tau \in [s,t_\varphi)$. Because, then we have shown that $y$ indeed satisfies \eqref{eq:integratedform}. Let $\theta \in [-h,0]$ be given. If $\tau + \theta \in [s-h,s]$ then we have that
\begin{equation*}
    y_\tau(\theta) = y(\tau + \theta) = \varphi(\tau + \theta - s) = (T_0(\tau - s)\varphi)(\theta) = u_\varphi(\tau)(\theta),
\end{equation*}
due to \eqref{eq:shift}. When $\tau + \theta \in [s,t_\varphi)$, it again follows from \eqref{eq:shift} and \cite[Lemma XII.3.3]{Diekmann1995} that
\begin{align*}
    y_\tau(\theta) &= y(\tau + \theta)\\
    &= u_\varphi(\tau + \theta)(0)\\
    &= (T_0(\tau + \theta -s)\varphi)(0) \\
    &+ j^{-1} \bigg(\int_s^{\tau + \theta} T_0^{\odot \star}(\tau + \theta - \sigma) [L(\sigma)u_\varphi(\sigma) + G(\sigma,u_\varphi(\sigma))]r^{\odot \star} d\sigma\bigg)(0)\\
    &= \varphi(0) + \int_0^{\tau + \theta} L(\sigma)u_\varphi(\sigma) + G(\sigma,u_\varphi(\sigma)) d\sigma\\
    &= (T_0(\tau -s)\varphi)(\theta) + j^{-1} \bigg(\int_s^{\tau} T_0^{\odot \star}(\tau - \sigma) [L(\sigma)u_\varphi(\sigma) + G(\sigma,u_\varphi(\sigma))]r^{\odot \star} d\sigma\bigg)(\theta)\\
    &=u_\varphi(\tau)(\theta),
\end{align*}
and so $y_\tau = u_\varphi(\tau)$ for all $\tau \in [s,t_\varphi)$. To conclude,
\begin{equation*}
    y(t) = \varphi(0) + \int_s^t L(\tau)y_\tau + G(\tau,y_\tau) d\tau,
\end{equation*}
and so $y$ satisfies the differential equation of \eqref{eq:T-DDEphi}. By the history property, and the fact that $\varphi \in X$, it follows by the method of steps applied to \eqref{eq:integratedform} that $y \in C^1([s,t_\varphi),\mathbb{R}^n)$. This shows that $y$ indeed is a solution to \eqref{eq:T-DDEphi}.
\end{proof}

\bibliographystyle{siamplain}
\bibliography{references}

\end{sloppypar}
\end{document}